\documentclass[11pt]{amsart}
\usepackage{amssymb,amsfonts,amsmath,amsopn,amstext,amsthm,amscd,color,
graphicx,latexsym,mathrsfs,todonotes,url, verbatim,tikz-cd,enumitem}
\usepackage{hyperref}

\theoremstyle{remark}
  
\newtheorem{para}{\bf}[section]
\newtheorem{example}[para]{\bf Example}

\newtheorem{rem}[para]{\bf Remark}

\theoremstyle{definition}
\newtheorem{dfn}[para]{Definition}

\theoremstyle{plain}

\newtheorem{thm}[para]{Theorem}

\newtheorem{lemma}[para]{Lemma}
\newtheorem{conj}[para]{Conjecture}
\newtheorem{cor}[para]{Corollary}
\newtheorem{prop}[para]{Proposition}

\numberwithin{equation}{section}

\newcommand{\pair}[2]{\langle #1 , #2 \rangle}

\newcommand{\bbD}{{\mathbb D}}

\newcommand{\bbN}{{\mathbb N}}

\newcommand{\bbQ}{{\mathbb Q}}

\newcommand{\bbZ}{{\mathbb Z}}

\newcommand{\bB}{{\bf B}}

\newcommand{\bG}{{\bf G}}
\newcommand{\bH}{{\bf H}}
\newcommand{\bI}{{\bf I}}

\newcommand{\bN}{{\bf N}}

\newcommand{\bT}{{\bf T}}

\newcommand{\bV}{{\bf V}}

\newcommand{\bX}{{\bf X}}

\newcommand{\bZ}{{\bf Z}}

\newcommand{\bb}{{\bf b}}

\newcommand{\frb}{{\mathfrak b}}

\newcommand{\frd}{{\mathfrak d}}

\newcommand{\frg}{{\mathfrak g}}
\newcommand{\frh}{{\mathfrak h}}

\newcommand{\frm}{{\mathfrak m}}
\newcommand{\frn}{{\mathfrak n}}

\newcommand{\frt}{{\mathfrak t}}

\newcommand{\frx}{{\mathfrak x}}

\newcommand{\frA}{{\mathfrak A}}

\newcommand{\frX}{{\mathfrak X}}

\newcommand{\cA}{{\mathcal A}}

\newcommand{\cC}{{\mathcal C}}
\newcommand{\cD}{{\mathcal D}}

\newcommand{\cF}{{\mathcal F}}

\newcommand{\cH}{{\mathcal H}}

\newcommand{\cM}{{\mathcal M}}

\newcommand{\cO}{{\mathcal O}}

\newcommand{\cT}{{\mathcal T}}
\newcommand{\cU}{{\mathcal U}}
\newcommand{\cV}{{\mathcal V}}
\newcommand{\cW}{{\mathcal W}}

\newcommand{\ovB}{{\overline B}}

\newcommand{\ovI}{{\overline I}}

\newcommand{\ovM}{{\overline M}}
\newcommand{\ovN}{{\overline N}}

\newcommand{\ovP}{{\overline P}}

\makeatletter
\newcommand*\bdot{{\mathpalette\bdot@{.9}}}
\newcommand*\bdot@[2]{\mathbin{\vcenter{\hbox{\scalebox{#2}{$\m@th#1\bullet$}}}}}
\makeatother

\newcommand{\an}{{\rm an}}

\newcommand{\BZ}{{\rm BZ}}

\newcommand{\cind}{{{\rm c}\mbox{-}{\rm Ind}}}

\newcommand{\coker}{{\rm coker}}

\newcommand{\End}{{\rm End}}

\newcommand{\GL}{{\rm GL}}

\newcommand{\Hom}{{\rm Hom}}

\newcommand{\Ind}{{\rm Ind}}

\newcommand{\la}{{\rm la}}

\newcommand{\Mod}{{\rm Mod}}

\newcommand{\pol}{{\rm pol}}

\newcommand{\Q}{{\mathbb Q}}

\newcommand{\Qp}{{\bbQ_p}}

\newcommand{\RHom}{{\rm RHom}}
\newcommand{\RiHom}{{R\underline{\mathrm{Hom}}}}
\newcommand{\rig}{{\rm rig}}

\newcommand{\sm}{{\rm sm}}

\newcommand{\Sp}{{\rm Sp}}
\newcommand{\Spa}{{\rm Spa}}
\newcommand{\Spec}{{\rm Spec}}
\newcommand{\Spf}{{\rm Spf}}

\newcommand{\Sym}{{\rm Sym}}

\newcommand{\Tor}{{\rm Tor}}

\newcommand{\wt}{{\rm wt}}

\newcommand{\Z}{{\mathbb Z}}

\newcommand{\mtwo}[4]{\begin{pmatrix}
	#1&#2\\#3&#4
\end{pmatrix}}

\begin{document}

\title[Bernstein-Zelevinsky duality for locally analytic principal series]{Bernstein-Zelevinsky duality for locally analytic principal series representations}
\author{Matthias Strauch}
\address{Indiana University, Department of Mathematics, Rawles Hall, Bloomington, IN 47405, U.S.A.}
\email{mstrauch@indiana.edu}
\author{Zhixiang Wu}
\address{Universit\"at M\"unster, Fachbereich Mathematik und Informatik,
Einsteinstrasse 62,
h48149 Münster, Germany}
\email{zhixiang.wu@uni-muenster.de}
\begin{abstract} 
    We consider certain dual of the Kohlhaase-Schraen resolutions for locally analytic principal series representations of $p$-adic Lie groups in the case of integral weights. The dual complexes calculate the expected Bernstein-Zelevinsky dual of the locally analytic representations and lead to the Grothendieck-Serre duality of coherent sheaves on patched eigenvarieties.
\end{abstract}

\maketitle

\tableofcontents

\section{Introduction}
Let $p$ be a prime number. In this paper, we determine the Bernstein-Zelevinsky dual of a locally analytic principal series representation of a split reductive $p$-adic Lie group, induced from a locally algebraic character of a maximal torus. Take $d\geq 2$ and $G=\GL_d(\Q_p)$ in this introduction.
\subsection{Motivation}
Let $\pi$ be a smooth representation of $G$ over $\mathbb{C}$, the Bernstein-Zelevinsky duality (also called cohomological duality) is given by (\cite[\S IV.5]{bernstein1992notes}, see also \cite{fargues2006dualite})
\[\pi\mapsto \mathbb{D}_{\rm BZ}(\pi):=R\Hom_G(\pi, \mathcal{C}^{\rm sm}_c(G,\mathbb{C}))\] 
where $ \mathcal{C}^{\rm sm}_c(G,\mathbb{C})$ denotes the space of compactly supported locally constant functions on $G$, which is a bimodule over $G$ via the left and right translations. An interesting property is that the duality should intertwine with the Grothendieck-Serre duality for coherent sheaves on the stack of Langlands parameters under the categorical Langlands correspondences (e.g. \cite[Conj. 4.5.1 (1)]{zhu2020coherent}).

For $p$-adic representations of $G$, notably in the categorical $p$-adic local Langlands program of Emerton-Gee-Hellmann presented in \cite{emerton2023introduction}, similar dualities are proposed for smooth representations in natural characteristics \cite[Conj. 6.1.14]{emerton2023introduction}. It is also expected that such duality exists for locally analytic representations \cite[Rem. 6.2.22]{emerton2023introduction}. Recent work of Hellmann-Hernandez-Schraen \cite{hellmann2024patching} gives strong evidence in this direction by establishing Serre duality for some sheaves on (patched) eigenvarieties. The property of these coherent sheaves is vital for \textit{loc. cit.} to produce multiplicities of $p$-adic automorphic eigenforms.

In this paper, for a locally analytic representation $\pi$ over a $p$-adic coefficient field $E$, we naively generalize the duality by 
\begin{equation}\label{equationintroductionoDBZ}
    \pi\mapsto \mathbb{D}_{\rm BZ}(\pi):= R\Hom_G(\pi, \mathcal{C}^{\rm la}_c(G,E))
\end{equation}    
where $\mathcal{C}^{\rm la}_c(G,E)$ is the space of compactly supported locally analytic functions on $G$. 
\begin{rem}
    Since (\ref{equationintroductionoDBZ}) involves the derived category of locally analytic representations (that may not be admissible in the sense of Schneider-Teitelbaum \cite{schneider2003algebras}), we can and will use the solid formalism of locally analytic representations by Rodrigues Jacinto-Rodríguez Camargo \cite{rodrigues2022solid,jacinto2023solid} to make the definition rigorous. 
\end{rem}
\begin{rem}
    Recent work of Claudius Heyer and Lucas Mann explains the cohomological duality for smooth representations using a geometric language and the six-functor formalism (see \cite[Prop. 1.4.3]{heyer20246functor}). Their approach, together with \cite{jacinto2023solid,camargo2024analytic}, should be enough to provide us with a general abstract theory of the duality for locally analytic representations.
\end{rem}
We will not consider the general theory in this paper. Rather, we would like to calculate explicitly $\bbD_{\rm BZ}(\pi)$ when $\pi$ is a locally analytic principal series representation. This will be enough to answer partially the expectation in \cite{emerton2023introduction}. Whatever the definition of $\bbD_{\rm BZ}$, we expect that (Theorem \ref{theoremsoliddualcompactinduction})
\[\bbD_{\rm BZ}(\cind_{I}^G\cW)\simeq\Hom_G(\cind_{I}^G\cW, \mathcal{C}^{\rm la}_c(G,E))=\cind_{I}^G\cW^{\vee}\] 
where $\cind_{I}^G$ denotes the compact induction when $I$ is a compact open subgroup of $G$, $\cW$ is certain locally analytic representation of $I$  such that the continuous $E$-linear dual $\cW^{\vee}$ is also a locally analytic representation. This would allow us to calculate $\mathbb{D}_{\rm BZ}(\pi)$ using a resolution of $\pi$ by compactly induced representations as done for smooth representations by Schneider-Stuhler in \cite{Schneider1997building}. 
\subsection{Duality for Kohlhaase-Schraen resolutions}
The resolution we consider is that of Kohlhaase-Schraen in \cite{kohlhaase2012homological}. Let $B$ be the Borel subgroup of $G$ of upper-triangular matrices with the maximal diagonal torus $T$. Let $\chi: T\rightarrow E^{\times}$ be a continuous character. Let $\pi=\Ind_{B}^G\chi$ be the locally analytic parabolic induction. Kohlhaase-Schraen found a presentation of $\pi$ as a (derived) quotient of certain compactly induced representation:
\[\Ind_{B}^G\chi\simeq \cind_I^G\cW_{\sharp}\otimes^L_\cH\cH/\frm.\] 
Here $I$ is the Iwahori subgroup of $G$, $\cW_{\sharp}$ is a locally analytic Banach representation of $I$ (see (\ref{equationintroductionWsharp}) for the precise definition) and $\cH=E[U_1,\cdots,U_d]$, where $U_1,\cdots,U_d\in \End_G(\cind_I^G \cW_{\sharp})$, is a polynomial (Iwahori-Hecke) algebra acting on $\cind_I^G \cW_{\sharp}$ with the maximal ideal $\frm=(U_1-1,\cdots,U_d-1)$ \footnote{The ideal $\frm$ differs from that in \cite{kohlhaase2012homological} because of our normalizations of the $U_i$-operators (Remark \ref{remarkdefinitinpsit}).}. Such presentation arises from a Koszul resolution:
\begin{equation}\label{equationintroductionresolutionprincipalseries}
    \wedge^{\bullet} \cH^d\otimes_{\cH} \cind_I^G\cW_{\sharp}:=[\cind_I^G\cW_{\sharp}\rightarrow\cdots\rightarrow \wedge^i\cH^d\otimes_{\cH}\cind_I^G\cW_{\sharp}\rightarrow \cdots\rightarrow \cind_I^G\cW_{\sharp}] \stackrel{\sim}{\rightarrow} \Ind_{B}^G\chi
\end{equation}
in the derived category. Taking account of the expectation that $\bbD_{\rm BZ}(\cind_I^G \cW_{\sharp})=\cind_I^G \cW_{\sharp}^{\vee}$, and the self-duality of the Koszul complex $\wedge^{\bullet} \cH^d$, $\bbD_{\rm BZ}(\Ind_{B}^G\chi)$ should be given by the following dual complex
\[\Hom_G(\wedge^{\bullet} \cH^d\otimes_{\cH}\cind_{I}^G\cW_{\sharp}, \mathcal{C}^{\rm la}_c(G,E))\simeq\wedge^{\bullet} \cH^d\otimes_{\cH} \cind_I^G \cW_{\sharp}^{\vee}[-d]\]
with a shift of degree $d$. Let $\overline{B}$ be the opposite Borel subgroup. Let $\mathfrak{g},\overline{\frb}$ be the Lie algebras of $G$ and $\ovB$ with the universal enveloping algebras $U(\frg), U(\overline{\frb})$. The theorem below determines this dual.
\begin{thm}
	Suppose that $\chi$ is a locally algebraic character with weight $\lambda\in\Z^n$ and let $\chi_{\rm sm}:T\rightarrow E^{\times}$ be the smooth part of $\chi$. Then there is a quasi-isomorphism 
	\[\wedge^{\bullet} \cH^d\otimes_{\cH} \cind_I^G \cW_{\sharp}^{\vee}[-d]\simeq \mathcal{F}_{\overline{B}}^G(\ovM(\lambda)^{\vee},\bbD_{\rm BZ}(\chi_{\rm sm}))\]
	where $\ovM(\lambda)^{\vee}=(U(\frg)\otimes_{U(\overline{\frb})}\lambda)^{\vee}$ is the dual Verma module in the BGG category $\cO^{\overline{\mathfrak{b}}}$ for the highest weight $\lambda$.
\end{thm}
Here $\bbD_{\rm BZ}(\chi_{\rm sm})=\chi_{\rm sm}^{-1}[-d]$ and $\mathcal{F}_{\overline{B}}^G(\ovM(\lambda)^{\vee},\chi_{\rm sm}^{-1})$ is a locally analytic representation constructed using the functor $\mathcal{F}_{\overline{B}}^G(-,-)$ in \cite{orlik2015jordan} from certain $\frg$-modules and smooth representations of $T$. It admits the same Jordan-Hölder factors as $\Ind_{\overline{B}}^G\chi^{-1}=\mathcal{F}_{\overline{B}}^G(\ovM(\lambda),\chi_{\rm sm}^{-1})$. Similarly, $\Ind_{B}^G\chi=\cF_B^G(M(-\lambda),\chi_{\rm sm})$ where $M(-\lambda)=U(\frg)\otimes_{U(\frb)}(-\lambda)\in\cO^{\frb}$ is the Verma module. We have a general result for representations that are in the image of the functor $\mathcal{F}_{B}^G$.
\begin{thm}[{Theorem \ref{theoremresolution} and Theorem \ref{theoremdualKS}}]\label{theoremintroductionmain}
    Let $M$ be a $\mathfrak{g}$-module in the BGG category $\mathcal{O}^{\mathfrak{b}}$ with algebraic weights and $\chi_{\rm sm}$ be a smooth character of $T$, then there exists a Koszul resolution of $\cF_{B}^G(M,\chi_{\rm sm})$ 
    \[ \wedge^{\bullet} \cH^d\otimes_{\cH} \cind_I^G \cW_{\sharp}\simeq \cF_{B}^G(M,\chi_{\rm sm})\] 
    for some Banach $I$-representation $\cW_{\sharp}$ such that there exists a quasi-isomorphism 
	\[ \wedge^{\bullet} \cH^d\otimes_{\cH} \cind_I^G \cW_{\sharp}^{\vee}[-d]\simeq \cF_{\overline{B}}^G(\Hom_E(M,E)^{\mathfrak{n}_{\overline{B}}^{\infty}},\mathbb{D}_{\rm BZ}(\chi_{\rm sm}))\]
	where $\Hom_E(M,E)^{\mathfrak{n}^{\infty}_{\ovB}}\in\cO^{\overline{\frb}}$ is certain BGG dual of $M\in\cO^{\frb}$ appeared in Breuil's adjunction formula \cite[Prop. 4.2]{breuil2015versII} (see (\ref{equationintroBreuiladjunction})).
\end{thm}
The calculation leads to the following conjecture.
\begin{conj}\label{conjectureintroduction}
	Let $P$ be a parabolic subgroup of $G$ (containing $B$) with Lie algebra $\mathfrak{p}$ and the opposite $\ovP$. Let $M$ be a $\mathfrak{g}$-module in the parabolic BGG category $\mathcal{O}^{\mathfrak{p}}$ with algebraic weights and $V$ be a finitely presented admissible smooth representation of the Levi factor of $P$, then we have
	\[ \bbD_{\rm BZ}(\cF_{P}^G(M,V))=\cF_{\overline{P}}^G(\Hom_E(M,E)^{\mathfrak{n}_{\overline{P}}^{\infty}},\mathbb{D}_{\rm BZ}(V))\]
	where $\mathfrak{n}_{\overline{P}}$ is the nilradical of the Lie algebra of $\ovP$.
\end{conj}
\begin{thm}[Theorem \ref{theoremBZdual}]
    Conjecture \ref{conjectureintroduction} is true if $P=B$, $V$ is a smooth character of $T$ and $\bbD_{\rm BZ}$ is defined by (\ref{equationintroductionoDBZ}) (Definition \ref{definitionDBZ}). 
\end{thm}
\begin{rem}
    If $M=E$ is the trivial $\frg$-module, then $\cF_P^G(M,V)=(\Ind_{P}^GV)^{\rm sm}$ is the smooth parabolic induction. In this case the conjecture says that $\bbD_{\rm BZ}((\Ind_{P}^GV)^{\rm sm})=(\Ind_{\ovP}^G\bbD_{\rm BZ}(V))^{\rm sm}$, which matches with the duality for smooth representations \cite[Thm. IV.31]{bernstein1992notes}. As in \textit{loc. cit.}, one possible approach to proving the conjecture would be establishing the second adjointness theorem for the functor $\cF_{P}^G$, in some way generalizing Breuil's adjunction formula \cite[Prop. 4.2]{breuil2015versII} (which is based on Emerton's adjunction \cite{emerton2007jacquetII}, see also \cite[Lem. 5.2.1]{breuil2019local}):
    \begin{equation}\label{equationintroBreuiladjunction}
        \Hom_G(\cF_{P}^G(M,V),\Pi)=\Hom_{(\frg,\ovP)}(\Hom_E(M,E)^{\mathfrak{n}_{\overline{P}}^{\infty}}\otimes_E\cC^{\rm sm}_{c}(N_{\ovP},V),\Pi)
    \end{equation}
    to the derived setting. 
\end{rem}
\begin{rem}
    Let $\tau:g\mapsto (g^{t})^{-1}$ be the Chevalley involution (inverse transpose) of $G=\GL_d(\bbQ_p)$ which switches $B$ and $\ovB$. Let $\bbD=\tau\circ \bbD_{\rm BZ}$ be the twist by the Chevalley involution of our duality functor. Then the above theorem says that $\bbD(\cF_{B}^G(M,\chi_{\rm sm}))=\cF_{B}^G(M^{\vee},\chi_{\rm sm})$ up to a degree shift, where $(-)^{\vee}$ denotes the dual in the BGG category $\cO^{\frb}$. This is exactly the duality expected in \cite[Rem. 6.2.22]{emerton2023introduction}, and discussed in \cite[Thm. 1.4]{hellmann2024patching}.
\end{rem}
\subsection{Duality for coherent sheaves}
Using the Koszul resolutions, we can verify the categorical expectation for the Bernstein-Zelevinsky duality in some global setting. The $I$-representations as $\cW_{\sharp}$ and $\cW_{\sharp}^{\vee}$ define coefficient systems on locally symmetric spaces with the Iwahori level at $p$. The cohomologies or homologies of these coefficient systems, which define overconvergent $p$-adic automorphic forms (e.g. \cite{urban2011eigenvarieties, ash2008p,loeffler2011overconvergent, hansen2017universal}), should be related by Poincaré duality and induce certain Serre duality of coherent sheaves on the (patched) eigenvariety after taking finite slope parts. In this paper, we will work in simple, and more local, settings of abstract patched eigenvarieties.

In \S\ref{sectionpatchingmodules}, using patched completed homologies in \cite{caraiani2013patching}, from a locally analytic representation $\pi$ appeared in Theorem \ref{theoremintroductionmain}, we can attach two coherent sheaves $\frA_{\infty}^{\rm rig}(\pi),\frA_{\infty}^{',\rm rig}(\pi)$ supported respectively on two rigid analytic spaces $\frX$ and $\frX'$. Here $\frX$ (resp. $\frX'$) is roughly the space of deformations of a mod-$p$ Galois representation $\overline{\rho}$ (resp. the dual $\overline{\rho}^{\vee}=\tau\circ \overline{\rho}$ up to a twist where $\tau$ is given by the inverse transpose). The Chevalley involution induces an isomorphism $\eta:\frX\simeq\frX'$. The following theorem is a formal consequence of Theorem \ref{theoremintroductionmain} and the Poincaré duality of completed cohomologies. 
\begin{thm}[Theorem \ref{theoremdualitypatchingmodule}]\label{theoremintrodualitypatchingmodule}
    Let $\pi=\cF_{B}^G(M,\chi_{\sm})$ be as in Theorem \ref{theoremintroductionmain}. There exists an isomorphism 
    \[\bbD_{\rm GS}(\frA_{\infty}^{\rm rig}(\pi) )\simeq \eta^*\frA_{\infty}^{',\rig}(\bbD_{\rm BZ}(\pi)).\]
    of coherent sheaves on $\frX$, where $\bbD_{\rm GS}(-)$ denotes the Grothendieck-Serre duality for coherent sheaves on $\frX$.
\end{thm}
\subsection{Construction and proof}
Now we go into more (technical) details of Theorem \ref{theoremintroductionmain}, by explaining the two types of resolutions $\wedge^{\bullet} \cH^d\otimes_{\cH} \cind_I^G \cW_{\sharp}$ and $\wedge^{\bullet} \cH^d\otimes_{\cH} \cind_I^G \cW_{\natural}$ of a principal series $\Ind_{B}^G\chi$ where $\chi$ is locally algebraic of weight $\lambda$. 

Take $n\geq 1$. Let $\cC^{I_n-\an}(I,E)$ be the space of functions on $I$ that is rigid analytic on all cosets for certain open normal subgroup $I_n\subset I$. Let $\cD^{I_n-\an}(I,E):=\cC^{I_n-\an}(I,E)'$ be the continuous dual of $\cC^{I_n-\an}(I,E)$, the $I_n$-analytic distribution algebra. Consider the $I$-subspace 
\[\Ind_B^G\chi(BI)=\{f\in\cC^{\rm la}(G,E) \mid f(bg)=\chi(b)f(g), \forall b\in B, \mathrm{supp}(f)
\subset BI\}\subset \Ind_B^G\chi\] 
of functions in $\Ind_B^G\chi$ supported on $BI\subset G$. Then the representation $\cW_{\sharp}$ of $I$ in (\ref{equationintroductionresolutionprincipalseries}) is defined to be 
\begin{equation}\label{equationintroductionWsharp}
    \cW_{\sharp}:=\Ind_B^G\chi(BI) \cap \cC^{I_n-\an}(I,E)
\end{equation}
of functions that are ``$I_n$-rigid analytic'' which is naturally a module over $\cD^{I_n-\an}(I,E)$.

Let $\ovN$ be the unipotent radical of $\ovB$. Using that $ B\backslash BI= I\cap \ovN$, we may identify the space $\Ind_B^G\chi(BI)=\Ind_{I\cap B}^I\chi$ with $\cC^{\rm la}(I\cap \ovN,E)$, the space of locally analytic functions on $I\cap \ovN$. There are inclusions of subspaces 
\begin{equation}\label{equationinclusionintro}
    \cC^{\rm pol}(I\cap \ovN,E)\subset \cW_{\sharp}\subset \cC^{\rm la}(I\cap \ovN,E)=\Ind_B^G\chi(BI)\subset \Ind_{B}^G\chi
\end{equation}
where $\cC^{\rm pol}(I\cap \ovN,E)$ is the space of polynomial (algebraic) functions on $I\cap \ovN$. The maps in (\ref{equationinclusionintro}) are equivariant for a subalgebra $\cD(\ovB\cap I,\frg)\subset \cD^{I_n-\an}(I,E)$ generated by $U(\frg)$ and the distribution algebra of $\ovB\cap I$. Moreover, as a $\frg$-module, $\cC^{\rm pol}(I\cap \ovN,E)$ can be identified to the dual Verma module $\ovM(\lambda)^{\vee}$ in $\cO^{\overline{\frb}}$ whose appearance is also the first step for (\ref{equationintroBreuiladjunction}). 

The first inclusion of (\ref{equationinclusionintro}) induces a natural map  
\[\cW_{\natural}:=\cD^{I_n-\an}(I,E)\otimes_{\cD(\ovB\cap I,\frg)}\cC^{\rm pol}(I\cap \ovN,E)\rightarrow \cW_{\sharp}.\]
Together with the Kohlhaase-Schraen resolution (\ref{equationintroductionresolutionprincipalseries}), we arrive at least a map of complexes
\begin{equation}\label{equationintroductionresolutionprincipalseriesnew}
    \wedge^{\bullet} \cH^d\otimes_{\cH} \cind_I^G \cW_{\natural}\rightarrow \wedge^{\bullet} \cH^d\otimes_{\cH} \cind_I^G \cW_{\sharp}\rightarrow \Ind_{B}^G\chi.
\end{equation}
Note that in contrast to the locally analytic representation $\cW_{\sharp}$, the $I$-representation $\cW_{\natural}$ is finite over the distribution algebra $\cD^{I_n-\an}(I,E)$ and is usually considered as a (continuous linear) dual of a locally analytic representation as $\cW_{\sharp}^{\vee}$ (even though $\cW_{\natural},\cW_{\sharp}^{\vee}$ are still locally analytic representations). The left-hand side of (\ref{equationintroductionresolutionprincipalseriesnew}) will be the complex that is ``dual'' to the original Kohlhaase-Schraen resolution in Theorem \ref{theoremintroductionmain} for $\bbD_{\rm BZ}(\Ind_B^G\chi)=\mathcal{F}_{\overline{B}}^G(\ovM(\lambda)^{\vee},\bbD_{\rm BZ}(\chi_{\rm sm}))$.

We prove that (\ref{equationintroductionresolutionprincipalseriesnew}) is a quasi-isomorphism, thus obtaining the second type resolution of $\Ind_B^G\chi$ in Theorem \ref{theoremintroductionmain}. The major difficulty lies in the surjectivity of the map 
\begin{equation}\label{equationintrosurjectivity}
    \cind_I^G \cW_{\natural}\rightarrow \Ind_{B}^G\chi.
\end{equation}
The image of $\cW_{\natural}$ is a union of $I$-translations of some Banach/Smith completion of $\cC^{\rm pol}(I\cap \ovN,E)$. Our proof of the surjectivity boils down to some statement similar to Proposition \ref{propositionintroductionDmodule} below on different completions of the $U(\frg)$-module $\ovM(\lambda)^{\vee}\simeq \cC^{\rm pol}(I\cap \ovN,E)\simeq \cC^{\rm pol}(\ovN,E)$. We let $\widehat{U}(\frg)$ be the Arens-Michael envelope of the universal enveloping algebra $U(\frg)$, which is a Fr\'echet completion of $U(\frg)$. We write $\bG=\GL_{n}\supset \overline{\bB}\supset\overline{\bN}$ for the algebraic groups such that $\ovB=\overline{\bB}(\Q_p),\ovN=\overline{\bN}(\Q_p)$. The space $\cC^{\rm an}(\ovN,E)$ of rigid analytic functions on the analytification $\overline{\mathbf{N}}^{\rm an}$ of $\overline{\mathbf{N}}_E$ is a Fr\'echet completion of the space $\cC^{\rm pol}(\ovN,E)$ of all algebraic functions.
\begin{prop}\label{propositionintroductionDmodule}
    The isomorphism $\ovM(\lambda)^{\vee}\simeq \cC^{\rm pol}(\ovN,E)$ extends to a surjection (even an isomorphism) $\widehat{U}(\frg)\otimes_{U(\frg)}\ovM(\lambda)^{\vee}\rightarrow \cC^{\rm an}(\ovN,E)$.
\end{prop}
Except for the case $\bG=\mathrm{SL}_2$ where we calculate explicitly (Example \ref{ExampleSL2}), the proof of the above statement uses essentially the theory of rigid analytic $\widehat{\cD}$-modules and the Beilinson-Bernstein localization for $\widehat{U}(\frg)$-modules of Ardakov-Wadsley (see for example \cite{ardakov2014ICM}, even though the actual proof will only use the theory developed in \cite{ardakov2013irreducible} for Banach completions). The reason is that it's easier to compare the two types of completions of $U(\frg)$-modules locally after the localization. In the case that $\lambda=0$, the dual Verma module $\ovM(\lambda)^{\vee}$ is localized to the $\cD$-module $j_*\cO_{\overline{\bN}}$ where $j:\overline{\bN}\simeq \overline{\bN}w_0\hookrightarrow \bG/\overline{\bB}$. Proposition \ref{propositionintroductionDmodule} is roughly equivalent to that the localization of $\widehat{U}(\frg)\otimes_{U(\frg)}\ovM(\lambda)^{\vee}$ is the $\widehat{\cD}$-module $j_*^{\rm an}\cO_{\overline{\bN}^{\rm an}}$. We emphasize the non-trivial fact that the direct image $j_*^{\rm an}\cO_{\overline{\bN}^{\rm an}}$ is coadmissible over $\widehat{\cD}$ which implies that $\cC^{\rm an}(\ovN,E)$ is coadmissible over $\widehat{U}(\frg)$ by taking the global section. Such direct images of $\widehat{\cD}$-modules were studied by Bitoun-Bode in \cite{BitounBode2021meromorphic} and our proof is largely inspired by their methods.
\begin{rem}
    As shown by \cite{BitounBode2021meromorphic}, the direct image of a coadmissible $\widehat{\cD}$-module may not be coadmissible. This is not the case for us with the integral weights but remains a serious issue for more general weights. We don't know whether the resolution (\ref{equationintroductionresolutionprincipalseriesnew}) holds if the weights of the character $\chi$ are not ``$p$-adically non-Liouville'' in some way, even when $\bG=\mathrm{SL}_2$. This is the reason that we restrict to locally algebraic characters for parabolic inductions. The duality of representations of general weights or families of representations is still mysterious to the authors.
\end{rem}
\subsection{Overview}
    We review the construction of Kohlhaase-Schraen in \S\ref{sectionKSresolutions}. In \S\ref{sectiondualcomplexes}, we establish the resolutions for representations and study dualities between complexes, thus proving the main Theorem \ref{theoremintroductionmain}. The proof of the surjectivity of (\ref{equationintrosurjectivity}) (and Proposition \ref{propositionintroductionDmodule}) is postponed to \S\ref{sectionlocalizationandcompletion}. In the last section \S\ref{sectiondualitysolid}, we embrace the solid formalism to discuss the duality between representations and coherent sheaves (Theorem \ref{theoremintrodualitypatchingmodule}).
\subsection{Acknowledgements}
We thank Eugen Hellmann for introducing us to the problem and for his interest in this work. The second author would like to thank Juan Esteban Rodriguez Camargo, Lucas Mann and Benjamin Schraen for very helpful discussions or answers to questions. Part of the work was done during the second author's visits to Beijing International Center for Mathematical Research and he would like to thank Yiwen Ding and Liang Xiao for hospitality. He also like to thank Claudius Heyer and Yuanyang Jiang for useful exchanges.

This work began during the Trimester Program 2023 on the arithmetic of the Langlands Program at Hausdorff Research Institute for Mathematics in Bonn, funded by the Deutsche Forschungsgemeinschaft (DFG, German Research Foundation) under Germany's Excellence Strategy – EXC-2047/1 – 390685813. The second author was funded by the Deutsche Forschungsgemeinschaft (DFG, German Research Foundation) – Project-ID 427320536 – SFB 1442, as well as under Germany's Excellence Strategy EXC 2044 390685587, Mathematics Münster: Dynamics–Geometry–Structure.
\subsection{Notation and convention}
Let $L$ be a finite extension of $\Q_p$ with the ring of integers $\cO=\cO_L$ and a uniformizer $\varpi=\varpi_L$. We take a finite large enough extension $E$ of $\Qp$ as the coefficient field. 

Fix $\bG$ to be a connected split reductive group over $L$ with a Borel subgroup $\bB=\bT\bN$ where $\bT$ is a split maximal torus. We fix also split models $\bG_{\cO_L}\supset \bB_{\cO_L}\supset \bT_{\cO_L}$ over $\cO_L$ of $\bG$. Let $d$ be the rank of $\bT$. Write $G=\bG(L)\supset B=\bB(L)\supset T=\bT(L)$ and similarly $N=\bN(L)$. Denote by $\ovB=\ovN T$ the opposite Borel subgroup. Let $\Phi\supset \Phi_+\supset \Delta$ be the set of roots, positive roots and simple roots of $\bG$, corresponding to $\bB$. Let $X_*(\bT)$ (resp. $X^*(\bT)$) be the lattice of coweights (resp. weights) of $\bT$.

Write $\frg,\frb,\frt,\frn$, etc. for the base change to $E$ of the $\Q_p$-Lie algebras of $G,B,T,N$, etc. Write $U(\frg),U(\frb)$, etc. for the universal enveloping algebras.

Set $K=\bG_{\cO_L}(\cO_L)$, a hyperspecial maximal compact open subgroup of $G$.  Write $B_0=K\cap B, T_0=T\cap K$, etc. Set $T^{-}=\{t\in T\mid t(K\cap \ovN) t^{-1}\subset K\cap \ovN\}$ and $T^+=(T^-)^{-1}$. 

For a $p$-adic Lie group $G$, we write $\cC^{\la}(G,E)$ for the space of locally $\Q_p$-analytic functions on $G$ with values in $E$. Write $\cD(G,E)=\cC^{\la}(G,E)'$ for the distribution algebra. 

Let $I$ be the preimage of $\bB_{\cO_L}(\cO_L/\varpi_L)$ under the reduction map $K\rightarrow \bG_{\cO_L}(\cO_L/\varpi_L)$. The Iwahori subgroup $I$ admits the Iwahori decomposition $I=(I\cap \ovN)(I\cap T)(I\cap N)$. There exists and we fix a sequence of good open normal subgroups $\cdots \subset I_{n+1}\subset I_n\subset \cdots \subset I$ where $n\geq 1$ forming a neighborhood basis of the neutral element of $G$ and admitting Iwahori decompositions $I_n=(I_n\cap \ovB)(I_n\cap T)(I_n\cap B)$, coming from decompositions of rigid analytic groups $\bI_n$ as in \cite[Prop. 4.1.3]{emerton2006jacquet}. 

We let $\cC^{\an}(I_n,E)$ be the rigid analytic functions on the rigid space $(\mathrm{Res}_{L/\Q_p}\bI_n)\otimes_LE$. Set $\cC^{I_n-\an}(I,E)=\oplus_{g\in I/I_n} g.\cC^{\an}(I_n,E)\subset \cC^{\la}(I,E)$. Since $I_n$ is normal in $I$ and any element $g\in I$ induces a rigid analytic automorphism of $\bI_n$ by conjugation, we have $g.\cC^{\an}(I_n,E)=\cC^{\an}(I_ng^{-1},E)=\cC^{\an}(g^{-1}I_n,E)$ for all $g\in I$. The space $\cC^{I_n-\an}(I,E)$ consists of $I_n$-rigid analytic functions on $I$. Let $\cD^{I_n-\an}(I,E)$ be the continuous linear dual of $\cC^{I_n-\an}(I,E)$.

We use $\tau$ to denote the Chevalley involution (inverse transpose) of $\bG,G$ or $\frg$. For $M\in\cO^{\frb}$, we will write $\tau(M)^{\vee}=\Hom_E(M,E)^{\frn_{\overline{B}}^{\infty}}$, see Remark \ref{remark-dualverma}.
\section{Kohlhaase-Schraen resolutions}\
\label{sectionKSresolutions}
In this section, we review and extend (a little bit of) the work of Kohlhaase-Schraen in \cite[\S 2]{kohlhaase2012homological} for resolutions of parabolic inductions.
\subsection{The Koszul complexes}\label{subsectionKoszulcomplex}
We start with some general considerations of compact inductions (from the Iwahori subgroup $I$) and Hecke operators. Then we turn to Kohlhaase-Schraen resolutions for parabolic inductions (from the Borel subgroup $B$) 

Suppose that $\cW$ is a representation of the Iwahori subgroup $I$ of the split $p$-adic reductive group $G$ over $E$. Consider the abstract compactly induced representation 
\[\cind_{I}^G\cW=\{f: G\rightarrow \cW \text{ compactly supported }\mid f(xg)=x.f(g), \forall x\in I \}\]
on which $G$ acts by right translations. If $\cW$ is a locally analytic representation of $I$, $\cind_{I}^G\cW$ can be equipped with the locally convex direct sum topology so that $\cind_{I}^G\cW$ is a locally analytic representation of $G$.

For $g\in G,w\in \cW$, we use the notation $[g,w]=g[1,w]$ to denote the function in $\cind_{I}^G\cW$ supported on $Ig^{-1}$ sending $ig^{-1}\in Ig^{-1}$ to $i.w$. Then $[g,w]=[gi,i^{-1}w]$ for $i\in I$. We also use the notation $[g,\cW]=[gI,\cW]$ to denote the space of all functions in $\cind_{I}^G\cW$ supported in $Ig^{-1}$. For a subset $S\subset G$, write $[S,\cW]=\sum_{s\in S}[s,\cW]\subset \cind_I^G\cW$. Then $\cind_{I}^G\cW=\oplus_{g\in G/I}[g,\cW]=\oplus_{g\in I\backslash G/I}[IgI,\cW]$.

The Hecke (endomorphism) algebra $\End_G(\cind_{I}^G\cW)=\Hom_I(W,(\cind_{I}^G\cW)|_I)$ acts on $\cind_{I}^G\cW$, where the equality follows from the usual Frobenius reciprocity. For $t\in T$,
\[\Hom_I(\cW,(\Ind_I^{ItI}\cW)|_I)=\Hom_I(\cW,\Ind_{I\cap tIt^{-1}}^I[t,\cW])=\Hom_{I\cap tIt^{-1}}(W,[t,W])\] 
gives rise to a subspace of $\End_G(\cind_{I}^G\cW)$ via $\Ind_I^{ItI}\cW=[ItI,\cW]\subset \cind_{I}^G\cW$.

Assume that for $t\in T^{-}$, there is an element $U_t\in \Hom_{I\cap tIt^{-1}}(\cW,[t,\cW])\subset \End_G(\cind_{I}^G\cW)$. Equivalently, we make $t^{-1}$ act on $\cW$ by an operator $\psi_t:\cW\rightarrow \cW$ in the sense that the map $w\mapsto [t,\psi_t.w]$ is in $\Hom_{I\cap tIt^{-1}}(\cW,[t,\cW])$. More explicitly, $\psi_t$ satisfies the condition
\[(t^{-1}xt)\psi_{t}(w)=\psi_{t}(xw)\] 
for $x\in I\cap tIt^{-1}, w\in\cW$. The corresponding action of $U_t$ on $\cind_{I}^G\cW$ is given by
\begin{equation}\label{equationUtaction}
    U_t[g,w]=\sum_{xtI\in ItI/I}[gxt,\psi_t(x^{-1}w)], \forall g\in G,w\in\cW
\end{equation}
by \cite[(2.2)]{kohlhaase2012homological} (here $xtI\in ItI/I$ is equivalent to $x\in I/(I\cap tIt^{-1})$).
\begin{lemma}\label{lemmacommutativeHecke}
    Suppose that $\psi_{t_1}\circ \psi_{t_2}:\cW\rightarrow \cW$ is equal to $\psi_{t_1t_2}$ for all $t_1,t_2\in T^{-}$ (in other words, $\cW$ becomes a module over the semi-group $IT^+I$), then $U_{t_1}U_{t_2}=U_{t_1t_2}$. 
\end{lemma}
\begin{proof}
    By the Iwahori decomposition, $ItI/I=(I\cap \ovN)/t(I\cap \ovN)t^{-1}\cdot tI/I$ for any $t\in T^{-}$.
    We calculate that 
    \begin{align*}
        &U_{t_1}U_{t_2}[g,w]\\
        =&\sum_{xt_1I\in It_1I/I}\sum_{yt_2I\in It_2I/I}[gyt_2xt_1,\psi_{t_1}(x^{-1}\psi_{t_2}(y^{-1}w))]\\
        =&\sum_{t_2xt_2^{-1}\in t_2(I\cap \ovN)t_2^{-1}/t_1t_2(I\cap \ovN)(t_1t_2)^{-1},y\in (I\cap \ovN)/t_2(I\cap \ovN)t_2^{-1}}[gy(t_2xt_2^{-1})t_2t_1, \psi_{t_1}\psi_{t_2}((yt_2xt_2^{-1})^{-1}w)]\\
        =&\sum_{x\in (I\cap \ovN)/t_1t_2(I\cap \ovN)(t_1t_2)^{-1}}[gxt_2t_1,\psi_{t_1}\psi_{t_2}(x^{-1}w)],
    \end{align*}
    where for the second equality, we used that $x^{-1}\psi_{t_2}(-)=\psi_{t_2}(t_2x^{-1}t_2^{-1}-)$ since $t_2x^{-1}t_2^{-1}\in t_2(I\cap \ovN)t_2^{-1}\subset I\cap t_2It_2^{-1}$. The result follows. 
\end{proof}

There is an isomorphism $T/T_0\simeq X_*(\bT)$ of $\Z$-modules characterized by $\langle \chi,t\rangle=\mathrm{val}_{\varpi}(\chi(t))$ for $\chi\in X^*(\bT)$, $t\in T$, where $\mathrm{val}_{\varpi}$ denotes the $\varpi$-adic valuation. Under the identification $T^{-}/T_0=X_*(\bT)^-=\{\mu\in X_*(\bT)\mid \langle\alpha,\mu\rangle\leq 0,\forall\alpha\in \Phi_+\}$ corresponds to the set of antidominant coweights. It contains $X_*(\bT)_0=\{\mu\in X_*(\bT)\mid \langle\alpha,\mu\rangle= 0,\forall\alpha\in \Phi\}\simeq \Z^{d-|\Delta|}$. Write $\bZ$ for the center of $\bG$ and $Z=\bZ(L)$. Then $\bZ\subset\bT$ (\cite[Prop. 21.7]{milne2017algebraic}) and $Z/(Z\cap K)\simeq X_*(\bT)_0$. Choose generators $z_1,\cdots,z_{d-|\Delta|}$ of $Z/(Z\cap K)$.  And we choose a splitting $X_*(\bT)\simeq X_*(\bT)'\oplus X_*(\bT)_0$ where $X_*(\bT)'\simeq \Z^{|\Delta|}$ is dual to $X^*(\bT)\cap \Q\Phi$. It's easy to see that the splitting induces an isomorphism $X_*(\bT)^-\simeq \bbN^{|\Delta|}\oplus \bbZ^{d-|\Delta|}$ as monoids and we can choose generators $t_{\alpha}\in T^-/T_0,\alpha\in\Delta$  for the free monoid $X_*(\bT)^-\cap X_*(\bT)'\simeq \bbN^{|\Delta|}$ (we use the same notation as in \cite[\S 1]{kohlhaase2012homological} but we don't require $\langle \alpha,t_{\alpha}\rangle=-1$!). 

Let $\cH=E[T^{-}/T_0]=E[X_*(\bT)^-]$ be the monoid ring. Then $\cH=E[z_i^{\pm}, t_{\alpha},i=1,\cdots,d-|\Delta|,\alpha\in\Delta]$.
 
Choose an ordering of $\Delta$ and we rewrite $(t_1,\cdots,t_d)$ for the sequence $(z_1,\cdots,z_{d-|\Delta|}, t_{\alpha},\alpha\in\Delta)$ of elements in $\cH$. We write $ \frm\subset \cH$ for the ideal $(t_1-1,\cdots,t_d-1)$. 

From now on we will always work in the situation that given $\cW$, there are operators $\psi_t$ for $t\in T^{-}$ making $\cW$ an $IT^+I$-module as in Lemma \ref{lemmacommutativeHecke}. This implies that $\cind_I^G\cW$ is an $\cH$-module given by (\ref{equationUtaction}) (note that $U_t$ is the identity map if $t\in T_0$.). And we will write the corresponding Hecke operators $U_t\in\cH$ for $t\in T^{-}/T_0$. If $f:\cW_1\rightarrow \cW_2$ is an $IT^+I$-equivariant map, then the induced $G$-map $\cind_{I}^G\cW_1\rightarrow \cind_{I}^G\cW_2$ is $\cH$-equivariant.

Recall for the sequence of elements $(U_{t_1}-1,\cdots,U_{t_d}-1)$, there is a homological complex $\wedge^{\bullet}\cH^d$ of $\cH$-modules with differentials 
\[\wedge^i\cH^d\rightarrow \wedge^{i-1}\cH^d: f_{1}\wedge\cdots\wedge f_{i}\mapsto\sum_{j=1}^i (-1)^{j+1}\varphi(f_j)f_{1}\wedge\cdots\wedge \widehat{f_{j}}\wedge \cdots\wedge f_i \]
where $\varphi:\cH^d\rightarrow \cH,(h_1,\cdots,h_d)\mapsto \sum_{i=1}^{d}h_i(U_{t_i}-1)$ \cite[\href{https://stacks.math.columbia.edu/tag/0621}{Tag 0621}]{stacks-project}. Note that 
\[H_i(\wedge^{\bullet}\cH^d\otimes_{\cH}\cind_{I}^G\cW)=\mathrm{Tor}_i^{\cH}(\cH/\frm,\cind_{I}^G\cW).\]
\begin{prop}\label{propositionKSresolutionKoszulregular}
    The sequence $(U_{t_1}-1,\cdots,U_{t_d}-1)$ is Koszul regular for the $\cH$-module $\cind_{I}^G\cW$: the natural map
    \[ \wedge^{\bullet}\cH^d\otimes_{\cH}\cind_{I}^G\cW\rightarrow \cind_{I}^G\cW\otimes_{\cH}\cH/\frm\]
    is a quasi-isomorphism of $E[G]$-modules
\end{prop}
\begin{proof}
    This is \cite[Thm. 2.5]{kohlhaase2012homological} with the same proof sketched below. We first reduce the statement to the adjoint case $G/Z$. The operators $U_{z_i},i=1,\cdots, d-|\Delta|$ acts on each $[gIZ,\cW]$ for $g\in G/IZ$. Since $gIZ/gI=gZ/g(Z\cap I)=g(Z/(Z\cap I))=\prod_{i}z_i^{\Z}$, it's easy to see that $U_{z_1}-1, \cdots, U_{z_{d-|\Delta|}}-1$ form regular sequences on $\cind_{I}^G\cW$ and $\cind_{I}^G\cW/(U_{z_i}-1,i=1,\cdots,d-|\Delta|)\simeq \cind_{IZ}^G\cW$. 
    
    We now consider the action of $\cH'=E[U_{t_{\alpha}},\alpha\in \Delta]$ on $\cind_{IZ}^G\cW$. Let $\frm'=U_{t_{\alpha}}-1,\alpha\in \Delta$. The statement is equivalent to that $\Tor_q^{\cH'}(\cH'/\frm',\cind_{IZ}^G\cW)=0$ for $q\geq 1$. Recall \cite[Lem. 1.4]{kohlhaase2012homological}, see also \cite[Lem. 2.20]{herzig2011classification}: $G=\coprod_{t\in T/T_0}KtI$ and for any $g\in G$ there is $t\in T^{-}$ such that $gItI\subset KT^{-}I$. Using the lemma, as in the proof of \cite[Thm. 2.5]{kohlhaase2012homological}, we have $\Tor_q^{\cH'}(\cH'/\frm',\cind_{IZ}^G\cW)\simeq \Tor_q^{\cH'}(\cH'/\frm',\cind_{IZ}^G\cW\otimes_{\cH'}\cH'[\frac{1}{U_{t_{\alpha}}-1},\alpha\in \Delta])\simeq \Tor_q^{\cH'}(\cH'/\frm',\cind_{IZ}^G\cW(G^-))$ where $G^{-}=KT^{-}IZ=\coprod_{\underline{m}\in\bbN^{\Delta}}K\prod_{\alpha}t_{\alpha}^{m_{\alpha}}IZ$ and 
    \[\cind_{IZ}^G\cW(G^-):=[G^{-},\cW]\subset \cind_{IZ}^G\cW.\] 
    
    We remain to show that $U_{t_{\alpha}}-1, \alpha\in \Delta$ form a regular sequence on $\cind_{IZ}^G\cW(G^-)$ as in \cite[Thm. 2.6]{kohlhaase2012homological} for an arbitrary ordering $\alpha_1,\cdots,\alpha_{|\Delta|}$ of $\Delta$. We prove by induction on $1\leq j\leq |\Delta|$. Suppose that $(U_{t_{\alpha_j}}-1)f=\sum_{i=1}^{j-1}(U_{t_{\alpha_i}}-1)f^i$ for some $f,f^1,\cdots,f^{i-1}\in \cind_{IZ}^G\cW(G^-)$. Then formally in $\prod_{g\in G^-/IZ}[g,\cW]$, $f=\sum_{i=1}^{j-1}(U_{t_{\alpha_i}}-1)(U_{t_{\alpha_j}}-1)^{-1}f^i=\sum_{i=1}^{j-1}(U_{t_{\alpha_i}}-1)\sum_{k\geq 0}(-1)U_{t_{\alpha_j}}^kf^i$. Consider 
    \[V_r=\coprod_{\underline{m}\in\bbN^{\Delta},m_{\alpha_j}=r}Kt_{\alpha_1}^{m_{\alpha_1}}\cdots t_{\alpha_{|\Delta|}}^{m_{\alpha_{|\Delta|}}}IZ.\]
    Then $U_{t_{\alpha_j}}:[V_r,\cW]\rightarrow [V_{r+1},\cW]$ and $U_{t_{\alpha_i}}-1:[V_r,\cW]\rightarrow [V_r,\cW]$ for $i\neq j$ by \cite[Lem. 1.3]{kohlhaase2012homological}. Let $g^i=\sum_{k\geq 0}(-1)U_{t_{\alpha_j}}^kf^i$. Write $f=\sum_{r\geq 0}f^{(r)}$ and $g^i=\sum_{r\geq 0}g^{i,(r)}$ where $f^{(r)},g^{i,(r)}\in [V_r,\cW]\subset \cind_{IZ}^G\cW(G)$ (this is possible since $f,f^i$ are compactly supported modulo $Z$). There exists $s$ such that $f^{(r)}=0$ for all $r\geq s$. Then $f=\sum_{i=1}^{j-1}(U_{t_{\alpha_i}}-1)\sum_{r=0}^sg^{i,(r)}$, which implies that $f$ lies in the subspace of $\cind_{IZ}^G\cW(G^-)$ generated by images of $U_{t_{\alpha_i}}-1,i=1,\cdots,j-1$.
\end{proof}
\begin{cor}\label{corinjectiontwoW}
    Let $\cW_1\hookrightarrow \cW_2$ be an injection of $IT^+I$-representations equipped with the corresponding $\psi_t$-actions. Then the induced $\cH$-map $\cind_I^G\cW_1\rightarrow\cind_{I}^G\cW_2$ induces an injection
    \[\cind_I^G\cW_1\otimes_{\cH}\cH/\frm\hookrightarrow \cind_I^G\cW_2\otimes_{\cH}\cH/\frm.\]
\end{cor}
\begin{proof}
    By Proposition \ref{propositionKSresolutionKoszulregular}, $\mathrm{Tor}_1^{\cH}(\cH/\frm, \cind_I^G(\cW_2/\cW_1))=0$.
\end{proof}
Now we consider parabolic inductions. Let $U$ be a locally analytic representation of the torus $T$ inflated to a representation of $B$ with the same notation. Recall the locally analytic parabolic induction $\Ind_B^GU$ as in \cite[\S 2.1]{emerton2007jacquetII}:
\[\Ind_{B}^GU=\{f\in\cC^{\la}(G,U)\mid f(bg)=b.f(g), \forall b\in B\}\]
with the action of $G$ is given by right translations. For an open subset $V\in B\backslash G$, e.g. $V=BI$, let $\Ind_{B}^GU(V)\subset \Ind_{B}^GU$ be the subspace consisting of functions supported in $V$. Let $\mathrm{Res}_V:\Ind_{B}^GU\rightarrow \Ind_{B}^GU(V), f\mapsto f\cdot \mathbf{1}_{V}$ be the restriction map where $\mathbf{1}_{V}$ denotes the characteristic function of $V$.

We take an $I$-invariant subspace $\cW\subset \Ind_{B}^GU(BI)$ such that for any $t\in T^{-}$, $\mathrm{Res}_{BI}(t^{-1}.\cW)\subset \cW$ as subspaces of $\Ind_{B}^GU(BI)$. 
\begin{dfn}\label{definitinpsit}
    In the above situation, we set 
    \[\psi_t: \cW\rightarrow \cW, f(-)\mapsto \mathrm{Res}_{BI}(t^{-1}.f)=\mathrm{Res}_{BI}(f(- t^{-1}))=f(-t^{-1})\mathbf{1}_{BI}\]
    for $t\in T^{-}$. And let $U_t\in T^{-}$ be the corresponding Hecke operator.
\end{dfn}
\begin{rem}\label{remarkdefinitinpsit}
    Our definition of $\psi_t$ is \textit{slightly different} from that of \cite{kohlhaase2012homological}, where $\psi_t.f(-)=\mathrm{Res}_{BI}f(t-t^{-1})=\chi(t)\mathrm{Res}_{BI}(f(-t^{-1}))$ when $U$ is a character $\chi$. Our price is that the representation $\cind_I^G\cW\otimes_{\cH}\cH/\frm$ will depend implicitly on $\chi(t_{i}),1\leq i\leq d$.
\end{rem}
The defined actions satisfy the condition of Lemma \ref{lemmacommutativeHecke}.
\begin{lemma}\label{lemmapsit}
    In the above definition, we have $\psi_t(xw)=(t^{-1}xt)\psi_t(w)$ for $x\in tIt^{-1}\cap I$. And for $t_1,t_2\in T^{-}$, $\psi_{t_1}\psi_{t_2}=\psi_{t_1t_2}$.
\end{lemma}
\begin{proof}
    The map $\mathrm{Res}_{BI}$ is $I$-equivariant. Hence $\psi_t(xw)=\mathrm{Res}_{BI}(t^{-1}xw)=\mathrm{Res}_{BI}(t^{-1}xtt^{-1}w)=(t^{-1}xt)\mathrm{Res}_{BI}(t^{-1}w)=(t^{-1}xt)\psi_t(w)$. And $\psi_{t_2}\psi_{t_1}f=\psi_{t_2}(f(-t_1^{-1})\mathbf{1}_{BI})=f(-t_2^{-1}t_1^{-1})\mathbf{1}_{BIt_2}\mathbf{1}_{BI}=f(-t_2^{-1}t_1^{-1})\mathbf{1}_{BI}$ since $BIt_2\supset BI$.
\end{proof}

By the Frobenius reciprocity, there is a $G$-map 
\begin{equation}\label{equaitionKSresolutionthemap}
    \Phi_{\cW}: \cind_{I}^G\cW\rightarrow \Ind_B^GU: [g,w]\mapsto g.w.
\end{equation}

\begin{prop}\label{propositionKSresolutioninjection}
    For $t_1,t_2\in T^{-}$, $U_{t_1t_2}=U_{t_1}U_{t_2}$ making $\cind_{I}^G\cW$ an $\cH$-module. Moreover, the map (\ref{equaitionKSresolutionthemap}) factors through an injection
    \[ \cind_{I}^G\cW\otimes_{\cH}\cH/\frm\hookrightarrow \Ind_B^GU.\]
\end{prop}
\begin{proof}
    The first statement is Lemma \ref{lemmacommutativeHecke} and Lemma \ref{lemmapsit}. The proof for the injection is the same as for \cite[Prop. 2.4]{kohlhaase2012homological}. We first recall the key input \cite[Prop. 1.2]{kohlhaase2012homological}: suppose that $t\in T^{-},k,k'\in K$ such that $ktIB\cap k'tIB\neq \emptyset$, then $ktIB=k'tIB, ktI=k'tI$ and $kI=k'I$ (the proof of \textit{loc. cit.}, which cites \cite[\S4, Prop. 7]{schneider1991cohomology}, works for the general split reductive group $G$).
    
    We check that for $t\in T^{-}$  
    \begin{align*}
        \Phi_{\cW}(U_t[g,w])=\sum_{xtI\in ItI/I}gxt.\psi_t(x^{-1}w)=\sum_{xtI\in ItI/I}gxt.(t^{-1}x^{-1}w\cdot \mathbf{1}_{BI})\\
        =\sum_{xtI\in ItI/I}gxtt^{-1}x^{-1}(w\mathbf{1}_{BIt^{-1}x^{-1}})=g.w=\Phi_{\cW}([g,w])
    \end{align*}
    since $BI=\coprod_{It^{-1}x^{-1}\in I\backslash It^{-1}I}BIt^{-1}x^{-1}$ (by \cite[Prop. 1.2]{kohlhaase2012homological} and $I/(I\cap tIt^{-1})=(I\cap \ovN)/t(I\cap \ovN)t^{-1}$). 
    
    By definition, we have $U_{z}[g,w]=[gz,z^{-1}.w]$ for $z\in Z,w\in \cW$. Hence $\Psi_{\cW}$ factors through $\cind_{I}^G\cW/(U_{z_1}-1,\cdots,U_{z_{d-|\Delta|}}-1)\simeq \cind_{IZ}^G\cW$ where the action of the center $Z$ on $\cW$ comes from the embedding $\cW\subset \Ind_{B}^GU$. The quotient $\cind_{IZ}^G\cW$ inherits the actions of $U_{t_{\alpha}},\alpha\in\Delta$. We only need to show that the kernel of $\Phi'_{\cW}:\cind_{IZ}^G\cW\rightarrow \Ind_B^GU$ lies in the subspace $\cV$ spanned by $(U_{t_{\alpha}}-1)[g,w]$ for $g\in G,w\in\cW$ and $\alpha\in\Delta$. The proof goes as for \cite[Prop. 2.4]{kohlhaase2012homological}, and we give a sketch in our notation. By \cite[Lem. 1.4]{kohlhaase2012homological} (recalled in the proof of Proposition \ref{propositionKSresolutionKoszulregular}), $G=\coprod_{t\in T/T_0Z}KtIZ$ and for any $g\in G$, there exists $t\in T^{-}$ such that $gItI\subset KT^{-}I$. Assume that $f\in \ker(\Phi'_{\cW})$ and we prove that $f\in\cV$. Since $U_t([g,w])\subset [gItIZ,\cW]$ and $\cV\subset \ker(\Phi'_{\cW})$, we may assume that $f\in [KT^{-}IZ,\cW]$. Since $Kt_1I\cdot It_2I\subset Kt_1t_2I$, up to add elements in $\cV$, we may assume that $f\in [Kt^nIZ,\cW]$ for some $n$ large enough and $t=\prod_{\alpha}t_{\alpha}$. Write $Kt^nIZ=\coprod_{j\in J}k_jt^nIZ$ and $f=\sum_j[k_jt^n,w_j]$. Then $ \Phi'_{\cW}(f)=\sum_{j}k_jt^n.w_j\in \Ind_B^GU$. The support of each $k_jt^n.w_j$ as a function on $G$ is contained in $BI(k_jt^n)^{-1}=(k_jt^nIB)^{-1}$ which are pairwisely disjoint for different $j$ by \cite[Prop. 1.2]{kohlhaase2012homological} recalled above. Hence $\Phi'_{\cW}(f)=0$ implies that $w_j=0$ for all $j$. We conclude that $f=0$.
\end{proof}

\subsection{Change the levels} \label{subsectionchangethegroups}
For the actual application to the duality (Theorem \ref{theoremdualKS}), we need the flexibility to shrink (or change) the Iwahori subgroup $I$ to other open compact subgroups of $G$ for the compact inductions in the Kohlhaase-Schraen resolutions. In this subsection, $I$ denotes an open compact subgroup such that $K\cap T\subset I\subset K$ admitting the Iwahori decomposition $I=(I\cap N)(I\cap T)(I\cap \ovN)$. For example, we can take $I$ to be the opposite Iwahori subgroup. We still let $\cH=E[T^{-}/(T\cap K)]\supset\frm=(U_{t_i}-1)_{i=1,\cdots,d}$. The discussions on Hecke operators at the beginning of \S\ref{subsectionKoszulcomplex}, including Lemma \ref{lemmacommutativeHecke}, apply for the compact induction $\cind_I^G\cW$ which will be an $\cH$-module if $\cW$ is an $IT^+I$-module equipped with the corresponding $\psi_t$-actions.
\begin{rem}
    One can also shrink the torus $T\cap I$: take $I'$ such that $I'\cap T\subsetneq I\cap T$ and replace $\cH$ by the algebra $E[T^{-}/(I'\cap T)]$. See \cite{ollivier2014resolutions} for the example when $I'$ is the pro-$p$ Iwahori subgroup.
\end{rem}
We first take a subgroup $I'$ such that $T\cap K\subset I\cap \ovB \subset I'\subset I$ with the Iwahori decomposition such that 
\begin{equation}\label{equationchangethegroupassumption}
    I'/(tI't^{-1}\cap I')=I/(tIt^{-1}\cap I)=(I\cap \ovN)/t(I\cap \ovN)t^{-1}, \forall t\in T^{-}.
\end{equation}
This assumption implies that $BI'=BI=B(I\cap \ovN)$. 
\begin{example}
    We can and will take $I'=(I\cap \ovN)(I\cap T)(\ovI\cap N)=\ovI \cap I$ where $I$ is the Iwahori subgroup and $\ovI$ is the Iwahori subgroup for the opposite Borel $\ovB$. 
 \end{example}

Suppose that $\cW$ is a representation of $IT^+I$ with operators $\psi_t,t\in T^{-}$ satisfying the condition of Lemma \ref{lemmacommutativeHecke}. We consider $\cW':=\cW|_{I'}$. Then operators $\psi_t$ still act on $\cW'$. We can similarly define Hecke operators $U_{t}$ acting on $\cind_{I'}^G\cW'$ as in (\ref{equationUtaction}) for $t\in T^{-}$ making $\cind_{I'}^G\cW'$ an $\cH$-module. We get a Koszul complex $\wedge^{\bullet} \cH^{d}\otimes_{\cH} \cind_{I'}^G\cW'$.  

Conversely, suppose that $\cW'$ is a representation of $I'T^+I'$ with $\psi_t$-actions   and $\cW=\cind_{I'}^{I}\cW'$. 
\begin{lemma}\label{lemmapsitcompactinduction}
    There are $\psi_t$-actions on $\cind_{I'}^I\cW'$ satisfying the condition of Lemma \ref{lemmacommutativeHecke} such that 
    \[\psi_t[g,w]=[t^{-1}gt,\psi_t(w)]\] for $[g,w]\in \cind_{I'}^I\cW', g\in I\cap N$ and $(t^{-1}xt)\psi_{t}(-)=\psi_{t}(x-)$ for $x\in I\cap tIt^{-1}$. 
\end{lemma}
\begin{proof}
    Since $\cind_{I'}^I\cW'=\oplus_{g\in I/I'}[g,\cW]$ and $I/I'=(I\cap N)/(I'\cap N)$, $\psi_t$ is already defined by the formula and we check it is well-defined. For $n\in B\cap I'$, $\psi_t[gn,n^{-1}w]=[t^{-1}gnt,\psi_t(n^{-1}w)]=[t^{-1}gt,t^{-1}nt\psi_t(n^{-1}w)]=[t^{-1}gt,\psi_t(w)]$ since $n\in I'\cap tI't^{-1}$. For $x\in I\cap tIt^{-1}$ and $g\in I\cap N\subset I\cap tIt^{-1}$, we can write $xg=g'i'$ for $g'\in I\cap N,i'\in I\cap\ovB=I'\cap \ovB$. Since $g',g'i'=xg\in I\cap tIt^{-1}$, we see $i'\in \ovB\cap I\cap tIt^{-1}=t(I\cap \ovB) t^{-1}$. Then $\psi_t[xg,w]=\psi_t[g',i'.w]=[t^{-1}g't,\psi_t(i'w)]=[t^{-1}g't,t^{-1}i't\psi_t(w)]=[t^{-1}xgt,\psi_t(w)]=t^{-1}xt[t^{-1}gt,\psi_t(w)]=(t^{-1}xt)\psi_{t}([g,w])$.
\end{proof}
\begin{prop}\label{propositionchangethegroups1}
    Let $T\cap K\subset I'\subset I$ be subgroups with the Iwahori decomposition satisfying (\ref{equationchangethegroupassumption}).
    \begin{enumerate}
        \item If $\cW'=\cW|_{I'}$ with the restricted $\psi_t$-actions, then there is an $\cH$-equivariant morphism $\cind_{I'}^G\cW'\rightarrow \cind_{I}^G\cW$ which induces a quasi-isomorphism 
        \[\wedge^{\bullet} \cH^{d}\otimes_{\cH} \cind_{I'}^G\cW'\stackrel{\sim}{\rightarrow} \wedge^{\bullet} \cH^{d}\otimes_{\cH} \cind_{I}^G\cW.\]
        \item If $\cW=\cind_{I'}^I\cW'$ with the induced $\psi_t$-actions given by Lemma \ref{lemmapsitcompactinduction}, then there is an $\cH$-equivariant isomorphism $\cind_{I'}^G\cW'\simeq \cind_{I}^G\cW$ which induces a quasi-isomorphism 
        \[\wedge^{\bullet} \cH^{d}\otimes_{\cH} \cind_{I'}^G\cW'\stackrel{\sim}{\rightarrow} \wedge^{\bullet} \cH^{d}\otimes_{\cH} \cind_{I}^G\cW.\]
    \end{enumerate}
\end{prop}
\begin{proof}
    (1) The map $\cind_{I'}^G\cW'\rightarrow \cind_{I}^G\cW$, denoted by $\mathrm{pr}$, is induced by the Frobenius reciprocity applying to the $I'$-inclusion $\cW'=[I,\cW]|_{I'}\subset (\cind_{I}^G\cW)|_{I'}$. It is also induced by $\cind_{I'}^I\cW'\rightarrow \cW:\sum_{i\in I/I'}[i,w_i]\mapsto \sum_{i\in I/I'} iw_i$. Then $\mathrm{pr}([g,[i,w]])=[g,iw]=[gi,w]$ for $g\in G,[i,w]\in\cind_{I'}^I\cW'$. We see $\mathrm{pr}([g,w])=[g,w]$ for $g\in G, w\in\cW'$.
 
    Let $t\in T^{-}$, $w\in \cW'$ and $[i,w]\in\cind_{I'}^I\cW'\subset \cind_{I'}^G\cW'$. Then by definition $U_{t}[i,w]=\sum_{a\in (I\cap \ovN)/t(I\cap \ovN)t^{-1}}[iat, \psi_{t}(a^{-1}w)]$. While for $[1,iw]\in \cind_{I}^G\cW$ and $i\in I$, $U_{t}[1,iw]=U_{t}[i,w]=\sum_{a\in (I\cap \ovN)/t(I\cap \ovN)t^{-1}}[iat, \psi_{t}(a^{-1}w)]$. By the explicit formula, the morphism $\mathrm{pr}$ commutes with the operator $U_t$. The kernel of $\mathrm{pr}$ is spanned by $G$-translations of elements $\sum_{i\in I/I'}[i,w_i]$ such that $\sum_{i\in I/I'}iw_i=0$. If $i\in I$ and $a\in I\cap \ovN$, write (uniquely) $ia=a_ix_{a,i}$ such that $a_i\in I\cap \ovN$ and $x_{a,i}\in I\cap B\subset I\cap tIt^{-1}$. Then $t^{-1}x_{a,i}t\psi_t(a^{-1}w_i)=\psi_t(x_{a,i}a^{-1}w_i)=\psi_t(a_i^{-1}iw_i)$. Suppose that $t$ is chosen such that $t^{-1}(I\cap B)t\subset I'$ (this is possible since $I\cap T=I'\cap T$ by our assumption). In this case we have 
    \[U_t(x)=\sum_{i\in I/I'}\sum_{a\in (I\cap \ovN)/t(I\cap \ovN)t^{-1}}[iat, \psi_{t}(a^{-1}w_i)]=\sum_{i,a}[a_it,\psi_t(a_i^{-1}iw_i)]=\sum_{a'}\sum_{i}[a't,\psi_t(a'^{-1}iw_i)]=0\] 
    for $x=\sum_{i\in I/I'}[i,w_i]$ in the kernel of $\mathrm{pr}$ so that $\sum_{i}iw_i=0$. For the third equality we used that when $a$ ranges over $(I\cap \ovN)/t(I\cap \ovN)t^{-1}=I/(I\cap tIt^{-1})$, $a_i(I\cap tIt^{-1})=i^{-1}a (I\cap tIt^{-1})\in I/(I\cap tIt^{-1})$ ranges over the same set.
    
    Let $\cH'=E[U_{t_i}^{\pm},i=1,\cdots,d]$. The fact that $U_t(x)=0$ for $x\in\ker(\mathrm{pr})$ implies that $\mathrm{pr}$ induces an $\cH'$-isomorphism: $ \cind_{I'}^G\cW'\otimes_{\cH}\cH'\simeq \cind_{I}^G\cW\otimes_{\cH}\cH'$. The map $\mathrm{pr}$ also induces a map of the Koszul complexes 
        \[  \wedge^{\bullet} \cH^{d}\otimes_{\cH} \cind_{I'}^G\cW'\rightarrow \wedge^{\bullet} \cH^{d}\otimes_{\cH} \cind_{I}^G\cW.\]
    The induced maps of homologies are $\Tor_{\bullet}^{\cH}(\cH/\frm, \cind_{I'}^G\cW')\rightarrow \Tor_{\bullet}^{\cH}(\cH/\frm, \cind_{I}^G\cW)$ which are isomorphisms since for any $\cH$-module $M$, $\Tor_{\bullet}^{\cH}(\cH/\frm, M)\simeq \Tor_{\bullet}^{\cH'}(\cH'/\frm, M\otimes_{\cH}\cH')$ \cite[\href{https://stacks.math.columbia.edu/tag/00M8}{Tag 00M8}]{stacks-project}).

    (2) Since $\cind_{I'}^G\cW'=\cind_{I}^G\cW$, we only need to check that the Hecke actions coincide. Take $t\in T^{-}$. Suppose that $[g,w]\in \cind_{I'}^G\cW'$ and $g\in I$. For $x\in I\cap\ovN=I'\cap\ovN$, we can write $gx=x'g'$ where $x'\in  I\cap \ovN$ and $g'\in I\cap B\subset  I\cap tIt^{-1}$. Then $[gxt,\psi_t(x^{-1}w)]=[x'g't,\psi_t(x^{-1}w)]=[x't,[t^{-1}g't,\psi_t(x^{-1}w)]]=[x't,\psi_t([(x')^{-1}g,w])]$. Note that when $x$ ranges over $I\cap\ovN/t(I\cap\ovN)t^{-1}=I/(I\cap tIt^{-1})$, $x'$ ranges over the same set. Hence
    \[U_t[g,w]=\sum_{x\in I'/(I'\cap tI't^{-1})}[gxt,\psi_t(x^{-1}w)]=\sum_{x\in I/(I\cap t I t^{-1})}[xt,\psi_t([x^{-1}g,w])].\]
    The result follows by definitions.
\end{proof}

\begin{cor}\label{corollarychangethegroupinjection}
    Let $I$ be the Iwahori subgroup and $\cW'$ be an $I'$-subrepresentation of $\Ind_B^GU (BI)$ (in the notation of Proposition \ref{propositionKSresolutioninjection}) that is stable under the operators $\psi_t:f\mapsto f(-t^{-1})\cdot \mathbf{1}_{BI}$ for all $t\in T^{-}$. Then the induced map $\cind_{I'}^G\cW'\rightarrow \Ind_B^GU$ factors through an injection
    \[ \cind_{I'}^G\cW'\otimes_{\cH}\cH/\frm\hookrightarrow \Ind_B^GU.\]
\end{cor} 
\begin{proof}
    Write $\cW_0=\Ind_B^GU (BI)$. Let 
    \[\cW'':=\coker(\cind_{I'}^I\cW'\rightarrow \cind_{I'}^I\cW_0|_{I'})=\cind_{I'}^I\coker(\cW'\rightarrow \cW_0|_{I'})\] 
    with the induced $\psi_t$-actions. We have a short exact sequence of $\cH$-modules
    \[ 0\rightarrow \cind_{I'}^G\cW'\rightarrow \cind_{I'}^G\cW_0|_{I'}\rightarrow \cind_{I}^G\cW''\rightarrow 0.\]
    By (the proof of) Proposition \ref{propositionKSresolutionKoszulregular}, $\mathrm{Tor}^{\cH}_1(\cH/\frm,\cind_{I}^G\cW'')=0$. Hence by Proposition \ref{propositionKSresolutioninjection} and Proposition \ref{propositionchangethegroups1}, the composite
    \[ \cind_{I'}^G\cW'\otimes_{\cH}\cH/\frm\hookrightarrow \cind_{I'}^G\cW_0|_{I'}\otimes_{\cH}\cH/\frm\simeq \cind_{I}^G\cW_0\otimes_{\cH}\cH/\frm\hookrightarrow \Ind_B^GU\]
    is an injection.
\end{proof}

We will also need the case when $T\cap K\subset I'\subset I$, $I'\cap \ovN\subsetneq I\cap \ovN$ while $I'\cap N=I\cap N$. Let $\cW$ be an $IT^+I$-representation as before and let $\cW'=\cW|_{I'T^+I'}$ with the restriction of $\psi_t$-actions satisfying the condition of Lemma \ref{lemmacommutativeHecke}. 
\begin{prop}\label{propositionchangethegroups2}
    In the above situation, there is an $\cH$-equivariant morphism $\cind_I^G\cW\rightarrow \cind_{I'}^G\cW'$ which induces a quasi-isomorphism 
    \[\wedge^{\bullet} \cH^{d}\otimes_{\cH} \cind_{I}^G\cW\stackrel{\sim}{\rightarrow} \wedge^{\bullet} \cH^{d}\otimes_{\cH} \cind_{I'}^G\cW'.\]
\end{prop}
\begin{proof}
    There is a natural $I$-injection 
    \[\cW\hookrightarrow \Ind_{I'}^I\cW|_{I'}\simeq \cind_{I'}^I\cW|_{I'}:w\mapsto \frac{1}{|I/I'|}\sum_{x\in I/I'}[x,x^{-1}w].\] 
    Write $\iota$ for the induced map $\cind_I^G\cW\hookrightarrow \cind_{I'}^G\cW'$. We see that 
    \begin{equation*}
    \begin{array}{rl}
        |I/I'|\iota(U_t[g,w])&=\sum_{h\in I/I'}\sum_{x\in I/(I\cap tIt^{-1})}[gxth,h^{-1}\psi_t(x^{-1}w)]\\
        &=\sum_{tht^{-1}\in tIt^{-1}/tI't^{-1}}\sum_{x\in I/(I\cap tIt^{-1})}[gx(tht^{-1})t,h^{-1}\psi_t(x^{-1}w)]\\
        &=\sum_{x\in (I\cap \ovN)/t(I'\cap \ovN)t^{-1}}[gxt,\psi_t(x^{-1}w)].
    \end{array}     
    \end{equation*}
    where we used that $I/(I\cap tIt^{-1})=(I\cap \ovN)/t(I\cap \ovN)t^{-1}, I/I'=(I\cap \ovN)/(I'\cap \ovN)$, and $h^{-1}\psi_t(x^{-1}w)=\psi_t(th^{-1}t^{-1}x^{-1}w)$ if $th^{-1}t^{-1}\in t(I\cap \ovN)t^{-1}\subset I\cap tIt^{-1}$.
    While
    \[|I/I'|U_t\iota([g,w])=\sum_{h\in I/I'}\sum_{x\in I'/(I'\cap tI't^{-1})}[ghxt,\psi_t(x^{-1}h^{-1}w)]=\sum_{x\in I/(I'\cap tI't^{-1})}[gxt,\psi_t(x^{-1}w)].\] 
    Hence $\iota$ is $\cH$-equivariant. For any $[1,w]\in\cind_{I'}^{G}\cW'$, we have
    \[U_t[1,w]=\sum_{x\in (I'\cap \ovN)/t(I'\cap \ovN)t^{-1}}[xt,\psi_t(x^{-1}w)]=\sum_{x\in t^{-1}(I'\cap \ovN)t/(I'\cap N)}[tx,\psi_t(tx^{-1}t^{-1}w)].\]
    Take $t\in T^{-}$ such that $I\cap \ovN\subset t^{-1}(I'\cap\ovN)t$. Then we can write the last term as 
    \begin{align*}
        &\sum_{x\in t^{-1}(I'\cap \ovN)t/(I\cap \ovN)}\sum_{y\in (I\cap \ovN)/(I'\cap \ovN)}[txy,\psi_t(ty^{-1}t^{-1}tx^{-1}t^{-1}w)]\\
        =&\sum_{x\in t^{-1}(I'\cap \ovN)t/(I\cap \ovN)}\sum_{y\in (I\cap \ovN)/(I'\cap \ovN)}[txy,y^{-1}\psi_t(tx^{-1}t^{-1}w)]
    \end{align*}
    since $ty^{-1}t^{-1}\in I\cap tIt^{-1}$. Now 
    \[\sum_{y\in (I\cap \ovN)/(I'\cap \ovN)}[txy,y^{-1}\psi_t(tx^{-1}t^{-1}w)]=\sum_{y\in I/I'}[txy,y^{-1}\psi_t(tx^{-1}t^{-1}w)]=\iota([tx,\psi_t(tx^{-1}t^{-1}w)])\]
    lies in the image of $\iota$. We conclude that $U_t$ annihilates $\mathrm{coker}(\iota)$. Using the argument as for (1) of Proposition \ref{propositionchangethegroups1}, the quasi-isomorphism follows.
\end{proof}
\begin{rem}
    Proposition \ref{propositionchangethegroups2} is an analog of \cite[Prop. 3.4.11]{emerton2006jacquet}.
\end{rem}
\subsection{Examples}\label{subsectionexamples}
We first recall the examples in \cite{kohlhaase2012homological}. Let $\chi:T\rightarrow E^{\times}$ be a continuous character. We take good open normal subgroups $I_n=\bI_n(L), n\geq 1$ of $I$ which admit Iwahori decompositions. Fix $n$ large enough such that $\chi$ is rigid analytic on $T\cap I_n$. Then Kohlhaase-Schraen takes the $I$-subspace 
\[\cW_{\sharp,n}:=\Ind_{B}^G\chi(BI)^{I_n-\mathrm{an}}\subset \Ind_{B}^G\chi\] 
of $I_n$-rigid analytic functions in $\Ind_{B}^G\chi$ supported on the open subset $BI$. Since $\cC^{I_n-\mathrm{an}}(I,E)=\cC^{(I_n\cap B)-\mathrm{an}}(I\cap B,E)\widehat{\otimes}_E\cC^{(I_n\cap \ovN)-\an}(I\cap \ovN,E)$, we see as an $I \cap \ovN$-representation, $\cW_{\sharp,n}$ can be identified with $\cC^{(I_n\cap \ovN)-\mathrm{an}}(I \cap \ovN,E)$ the space of $I_n\cap\ovN$-analytic functions on $I \cap \ovN$. As $t^{-1}(I \cap \ovN)t\supset I \cap \ovN$ for $t\in T^{-}$, there is a factorization
\[\mathrm{Res}_{BI}\circ \mathrm{Ad}(t): \cC^{(I_n\cap \ovN)-\mathrm{an}}(I \cap \ovN,E)\rightarrow \cC^{t^{-1}(I_n\cap \ovN)t-\mathrm{an}}(t^{-1}(I \cap \ovN)t,E)\rightarrow \cC^{(I_n\cap \ovN)-\mathrm{an}}(I \cap \ovN,E).\]
Since $f(-t^{-1})=\chi(t)^{-1}f(t-t^{-1})$, the map $\psi_t$ in Definition \ref{definitinpsit} stabilizes $\cW_{\sharp,n}$.
\begin{prop}[{\cite[Prop. 2.4]{kohlhaase2012homological}}]\label{propositionKSresolutionsujectivity}
    For $\cW_{\sharp,n}$ as above, the map $\Phi_{\cW_{\sharp,n}}:\cind_{I}^G\cW_{\sharp,n}\rightarrow \Ind_{B}^G\chi$ is surjective. And $\Phi_{\cW_{\sharp,n}}$ induces a quasi-isomorphism (see also Proposition \ref{propositionKSresolutioninjection} and Proposition \ref{propositionKSresolutionKoszulregular})
    \[\wedge^{\bullet}\cH^d\otimes_{\cH}\cind_{I}^G\cW_{\sharp,n}\simeq  \Ind_B^G\chi.\]
\end{prop}

\begin{cor}\label{corollaryKSresolution}
    For any sub $I$-representation $\cW_{\sharp}$ such that there exists $m\geq n$ and inclusions $\cW_{\sharp,n}\subset \cW_{\sharp}\subset \cW_{\sharp,m}$ and $\cW_{\sharp}$ is stable under $\psi_t$ for all $t\in T^-$, the map $\Phi_{\cW_{\sharp}}$ induces a quasi-isomorphism
    \[\wedge^{\bullet}\cH^d\otimes_{\cH}\cind_{I}^G\cW_{\sharp}\simeq  \Ind_B^G\chi.\]
\end{cor}
\begin{proof}
    Use Proposition \ref{propositionKSresolutioninjection}, Proposition \ref{propositionKSresolutionKoszulregular} (Corollary \ref{corinjectiontwoW}) and the surjectivity of $\Phi_{\cW_{\sharp,n}}$. 
\end{proof}

We now give examples of $\cW_{\sharp}$ that satisfy the assumption of Corollary \ref{corollaryKSresolution}, whose duals will be more convenient for us. Take an open normal subgroup $H\subset I$ with Iwahori decomposition $H=(H\cap \ovN)(H\cap T)(H\cap N)$ and such that $H$ is uniform pro-$p$, see \cite[\S 4]{schneider2003algebras} and the proof of \cite[Thm. 5.5]{orlik2015jordan}. Equip $H$ with a $p$-valuation $\omega$ and ordered basis $h_1,\cdots,h_{d_H}$ for $d_H=\dim_{\Q_p}H$ compatible with the Iwahori decomposition and root groups, see \cite[\S 3.3.3]{orlik2010irreducibility}.  The ordered basis induces $\Z_p^{d_H}\simeq H$ as $p$-adic manifolds. Upon shrinking $H$, we assume that $\omega(h_1)=\cdots=\omega(h_{d_H})=\omega\in \bbN$ using the method in \cite[Prop. III.3.1.3]{lazard1965groupes}.

The distribution algebra $\cD(H,E)$ is the continuous dual of $\cC^{\la}(H,E)$. We consider its Banach completions 
\[\cD_r(H,E)=\{\sum_{\alpha\in \bbN^{d_H}}d_{\alpha}\bb^{\alpha}\mid d_{\alpha}\in E, \lim_{\alpha\rightarrow \infty}|d_{\alpha}|_pr^{\sum_{i}\alpha_i\omega(h_i)}=0 \}\] 
for $r<1$ where $b_i=h_i-1$ and $\bb^{\alpha}=(h_1-1)^{\alpha_1}\cdots(h_{d_H}-1)^{\alpha_{d_H}}$. Since $I$ normalizes $H$, the conjugation of $x\in I$ induces an automorphism of $D_{r}(H,E)$ (see the discussion in the proof of \cite[Thm. 5.1]{schneider2003algebras}). We let $\cD_r(I,E)=E[I]\otimes_{E[H]}\cD_r(H)$. Let $\cC^r(I,E)$ be the continuous $E$-dual of $\cD_r(I,E)$ with the open compact topology, a Smith space with an open compact lattice over $\cO_L$. Then $\cC^r(I,E)=\prod_{g\in I/H}g.C^r(H,E)$ where 
\[\cC^r(H,E)=\{\sum_{\alpha\in\bbN^{d_H}}c_{\alpha}\binom{x_1}{\alpha_1}\cdots \binom{x_{d_H}}{\alpha_{d_H}}\mid \sup_{\alpha} |c_{\alpha}|_pr^{-\sum_{i}\alpha_i\omega(h_i)}<+\infty\}\] 
for $x=(x_1,\cdots,x_{d_H})\in\Z_p^{d_H}\simeq H$ in terms of Mahler's basis. Moreover, similarly as $\cC^{(I_n\cap \ovN)-\an}(I,E)$, we have (cf. \cite[Thm. 3.40 (2)]{rodrigues2022solid} and \cite[Prop. 3.3.4]{orlik2010irreducibility})
\[\cC^r(I,E)=(\cD_r(I,E))'=(\cD_r(B\cap I,E)\widehat{\otimes}_E\cD_r(I\cap \ovN,E))'=\cC^r(B\cap I,E)\widehat{\otimes}_E\cC^r(I\cap \ovN,E).\]
By \cite[III.1.3.8]{lazard1965groupes} (and see also \cite[\S 4.1.3]{rodrigues2022solid}), $\cC^r(I,E)\supset \cC^{I_m-\an}(I,E)$ for $r$ sufficiently close to $1$. We set
\[ \cW_{\sharp,r}:=\Ind_{B}^G\chi(BI)\cap \cC^r(I,E).\]
Then $\cW_{\sharp,m}\subset \cW_{\sharp,r}$ for $r$ sufficiently close to $1$.
\begin{lemma}
    The operator $\psi_t$ in Definition \ref{definitinpsit} stabilizes $\cW_{\sharp,r}$ for $r<1$ sufficiently close to $1$. Hence the conclusion of Corollary \ref{corollaryKSresolution} holds for such $\cW_{\sharp,r}$.
\end{lemma}
\begin{proof}
    As in the case for $\cC^{I_n-\an}(I,E)$ and $\cW_{\sharp,n}$, we need to show that if $f\in \cC^r(I\cap \ovN,E)$, then $\mathrm{Res}_{I\cap \ovN}(f(t-t^{-1}))$ is still in $\cC^r(I\cap \ovN,E)$. Since $I\cap \ovN$ is a product of its root subgroups and the conjugation of $t$ acts on each of them, we reduce to prove the similar result for $\cO_L\subset L=\bG_{m/L}(L)$ and for the action of $\varpi_L$ on $\cO_L$ or $\cC^r(\cO_L,E)$, namely to show that $ f(-)\mapsto f(\varpi_L-)$ sends $\cC^r(\cO_L,E)$ into $\cC^r(\cO_L,E)\subset \cC(\cO_L,E)$. Choose a $\Z_p$-basis $1=e_1,\cdots,e_{d_L}$ of $\cO_L$ with coordinates $x_1,\cdots,x_{d_L}\in \Z_p$. Then $\times \varpi_L$ factors as $\cO_L\stackrel{\tau}{\rightarrow} \Z_p^{d_L}\stackrel{\eta}{\rightarrow} \cO_L=\Z_p^{d_L}$ where $\tau$ is a $\Z_p$-linear automorphism of $\cO_L$ and $\eta: (x_1,\cdots,x_{d_L})\mapsto (p^{a_1}x_1,\cdots,p^{a_{d_L}}x_{d_L})$. Since the definition of $\cD_r(\cO_L,E)$, as well as $\cC^r(\cO_L,E)$, is independent of the choice of ordered basis, $\tau$ induces an automorphism of $\cC^r(\cO_L,E)$ and we may assume that $L=\Q_p$ and $\varpi_L=p$. For simplicity, we assume that $\omega(1)=1$. Take $n$ such that $p^{-\frac{p^{-n}}{p-1}}<r\leq p^{-\frac{p^{-(n+1)}}{p-1}}$. Then $\cC^{p^{n}\Z_p-\an}(\Z_p,E)\subset \cC^r(\Z_p,E)\subset \cC^{p^{n+1}\Z_p-\an}(\Z_p,E)=\{\sum_{\alpha\geq 0}c_{\alpha}\binom{x}{\alpha}|\varinjlim_{\alpha\rightarrow +\infty} |c_{\alpha}|_pp^{\frac{p^{-{(n+1)}}\alpha-s(\alpha)}{p-1}}=0\}$ by \cite[III.1.3.8]{lazard1965groupes} where $s(\sum \alpha_ip^i)=\sum_{i}\alpha_i$ if $\alpha_i\in\{0,1,\cdots,p-1\}$. The result follows. See also \cite[Cor. 3.3.10]{johansson2019extended} and Lemma \ref{lemmatransposepsit} for the dual version.
\end{proof}

\section{Dual complexes}\label{sectiondualcomplexes}
In this section, we establish the main theorem on resolutions of representations constructed from the functor $\cF_{B}^G$ in \cite{orlik2015jordan} (Theorem \ref{theoremresolution}) and study the duality between these complexes (Theorem \ref{theoremdualKS}).
\subsection{Banach modules over the distribution algebras}
We recall the definitions of various distribution algebras and use them to construct locally analytic $I$-representations as a continuation of \S \ref{subsectionexamples}. 

Let $\cD(\frg,\ovB\cap I)$ be the subalgebra $U(\frg)\cD(\ovB\cap I,E)$ generated by $U(\frg)$ and $\cD(\ovB\cap I,E)$ inside $\cD(I,E)$, which was considered in \cite[\S 3.4]{orlik2015jordan} or \cite[\S 4]{schmidt2016dimensions}. Define similarly $\cD(\frg,\ovB)$. Note that $\cD(\frg,\ovB\cap I)=U(\frg)\otimes_{U(\overline{\frb})}\cD(\ovB\cap I)$ \cite[Prop. 3.5]{orlik2015jordan}.

\begin{dfn}\label{definitionWnaturalM}
    If $\ovM$ is a $\cD(\frg, \ovB)$-module such that $\ovM$ is finitely presented over $U(\frg)$ and $\ovB$-locally finite (i.e., $\ovM$ is the union of its finite-dimensional $\ovB$-subspaces), we let 
    \[\cW_{\natural,r}(\ovM):=\cD_r(I,E)\otimes_{\cD(\frg,\ovB\cap I)}\ovM.\]
\end{dfn}

Since $\ovM$ is assumed to be a finitely presented $\cD(\frg,\ovB\cap I)$-module, $\cW_{\natural,r}(\ovM)$ is a finitely presented $\cD_r(I,E)$-module. 

By definition, $\cD_r(I,E)=E[I]\otimes_{E[H]}\cD_r(H,E)$ while 
\[\cD_r(H,E)=E[H]\otimes_{E[H^m]}D_{r^{p^m}}(H^m)=E[H]\otimes_{E[H^m]}U_r(\frg)\] 
for $m$ such that $ \frac{1}{p}<r^{p^m}<p^{-\frac{1}{p-1}}$ (\cite[(5.5.6)]{orlik2015jordan} for $p\geq 2$) where $U_r(\frg)$ is the closure of $U(\frg)$ in $\cD_r(I,E)$. Moreover, by \cite[Sublemma 5.6]{orlik2015jordan}, the closure of $\cD(\frg,\ovB\cap I)$ in $\cD_r(I,E)$ is a subring $\cD_r(\frg,\ovB\cap I)$ such that $\cD_r(I,E)=\oplus_{I/H^m(\ovB\cap I)}\delta_g\cD_r(\frg,\ovB\cap I)$ where $\delta_g$ denote Dirac distributions and $H^m(\ovB\cap I)=H(\ovB\cap I)\cap \cD_r(\frg,\ovB\cap I)$. 

For $\ovM$ in Definition \ref{definitionWnaturalM}, we write $\ovM_r:=U_r(\frg)\otimes_{U(\frg)}\ovM=\cD_r(\frg,\ovB\cap I)\otimes_{\cD(\frg,\ovB\cap I)}\ovM$ (\cite[Lem. 4.6]{schmidt2016dimensions}). Then we get
\[\cW_{\natural, r}(\ovM)=\oplus_{g\in I/H^m(\ovB\cap I)}\delta_r\ovM_r.\]

Note that $\widehat{U}(\frg)=\varprojlim_rU_r(\frg)$ is a Fr\'echet-Stein algebra with the Fr\'echet-Stein structure (cf. \cite[\S 3]{schneider2003algebras}, \cite[Prop. 4.8]{schmidt2016dimensions}). Let $\widehat{\ovM}:=\widehat{U}(\frg)\otimes_{U(\frg)}\ovM=\varprojlim_r\ovM_r$. Similarly, $\cD(I,E)=\varprojlim_r\cD_r(I,E)$. 
\begin{lemma}\label{lemmaexact}
    The functors $\ovM\mapsto \cW_{\natural,r}(\ovM)$ are exact on $\ovM$ that are finitely presented as $U(\frg)$-modules and $\cW_{\natural,r}(\ovM)\neq 0$ for $r$ sufficiently close to $1$. Moreover, if $\ovM$ is a simple $U(\frg)$-module, then $\ovM_{r}$ is a simple $U_{r}(\frg)$-module.
\end{lemma}
\begin{proof}
    The exactness follows from the flatness of $U_r(\frg)$ over $\widehat{U}(\frg)$ by \cite[Rem. 3.2]{schneider2003algebras} and that $\widehat{U}(\frg)$ is flat over $U(\frg)$ \cite[Thm. 4.3.3]{schmidt2013verma}. The simplicity of $\ovM_{r}$ when $\ovM$ is simple is well known \cite[Thm. 5.7]{orlik2015jordan} and we give a sketch. We may suppose that $\ovM_{r}\neq 0$ (cf. \cite[Lem. 4.3.6]{schmidt2013verma}), hence $\ovM\subset \ovM_{r}$. Let $0\neq N\subset \ovM_{r}$ be a sub-$U_r(\frg)$-module. We may assume that $N$ is finitely generated since $U_{r}(\frg)$ is a Noetherian Banach algebra, $N$ is closed in $\ovM_{r}$ equipped with the induced canonical Banach topology. One then shows that the $\frt$-weight spaces of $N$ are contained in $\ovM$, cf. \cite[Cor. 1.3.22]{feaux1999einige} or see \cite[Prop. 2.0.1]{schmidt2013verma}. If $N\neq 0$, then $N$ contains $\ovM$ and $N=\ovM_{r}$. 
\end{proof}
\begin{lemma}\label{lemma-psit}
    Suppose that $t\in T^{-}$ and $\ovM$ is a $\cD(\frg,\ovB)$-module such that $\ovM|_{U(\frg)}\in \cO^{\overline{\frb}}_{\rm alg}$ (\cite[Def. 2.6]{orlik2015jordan}) and is $\ovB$-locally finite. The action of the Dirac distribution $\delta_{t^{-1}}\in \cD(\frg,\ovB)$ on $\ovM$ extends continuously to a map $\ovM_r\rightarrow \ovM_r$ and induces an operator $\psi_t:\cW_{\natural,r}(\ovM)\rightarrow \cW_{\natural,r}(\ovM)$ satisfying $(t^{-1}xt)\psi_{t}(-)=\psi_{t}(x-), \forall x\in I\cap tIt^{-1}$.
\end{lemma}
\begin{proof}
    The $\cD(\frg,\ovB)$-module $\ovM$ admits a presentation
    \[\cD(\frg,\ovB)\otimes_{\cD(\ovB,E)}\sigma'\rightarrow \cD(\frg,\ovB)\otimes_{\cD(\ovB,E)}\sigma\rightarrow\ovM\rightarrow 0\]
    where $\sigma,\sigma'$ are finite-dimensional $\ovB$-representations over $E$. This is possible since $\ovB$ acts on $\ovM$ locally finitely and the same statement holds for $\cD(\frg,\ovB)\otimes_{\cD(\ovB,E)}\sigma=\cD(\frg,\ovB\cap I)\otimes_{\cD(\ovB\cap I,E)}\sigma=U(\frg)\otimes_{U(\overline{\frb})}\sigma$.
    By base change we get the exact sequence
    \[\cD_r(\frg,\ovB\cap I)\otimes_{\cD(\ovB\cap I,E)}\sigma'\rightarrow \cD_r(\frg,\ovB\cap I)\otimes_{\cD(\ovB\cap I,E)}\sigma\rightarrow\ovM_r\rightarrow 0.\]
    To extend the action of $\delta_{t^{-1}}$ to $\overline{M}_r$, it's enough to verify the existence of the extensions for $\cD_r(\frg,\ovB\cap I)\otimes_{\cD(\ovB\cap I,E)}\tau=U_r(\frg)\otimes_{U(\overline{\frb})}\tau$, $\tau=\sigma,\sigma'$ (and use the uniqueness of such extensions). While $U_r(\frg)\otimes_{U(\overline{\frb})}\tau=U_r(\frg)\otimes_{U_r(\overline{\frb})}\tau=U_r(\frn)\otimes_E\tau$ using $U_r(\frg)=U_r(\frn)\widehat{\otimes}_EU_r(\overline{\frb})$. We can get the $\delta_{t^{-1}}$-action on $U_r(\frn)\otimes_E\tau$ since $\delta_{t^{-1}}U_r(\frn)\subset U_r(\frn)\delta_{t^{-1}}$ for $t\in T^{-}$ and $\tau$ is $\ovB$-stable. As an $I$-representation, $\cW_{\natural,r}(\ovM)=\cind_{H^m(\ovB\cap I)}^I\ovM_r$. The $\psi_t$-operator on $\cW_{\natural,r}(\ovM)$ is defined using Lemma \ref{lemmapsitcompactinduction}. 
\end{proof}
\subsection{The resolution for principal series}\label{subsectionresolutionprincipalseries}
We continue with the principal series $\Ind_B^G\chi$ from \S \ref{subsectionexamples} and with the notation there. We will construct $I$-representations $\cW_{\natural,r}$ for $\Ind_{B}^G\chi$ that do not satisfy the assumption but might satisfy the conclusion of Corollary \ref{corollaryKSresolution}. Let $\lambda=\wt(\chi)\in (E^{d})^{[L:\Q_p]}$ be the weight of the locally analytic character $\chi$. We assume that $\lambda\in (\mathbb{Z}^{d})^{[L:\Q_p]}$, namely $\chi$ is locally algebraic.

Consider the subspace $\cC^{\rm pol}(I\cap \ovN,E)= \cC^{\rm pol}(\ovN,E)\mathbf{1}_{I\cap \ovN}\subset \cW_{\sharp,r}\subset \Ind_{B}^G\chi$ consisting of polynomial functions on $I\cap \ovN$, identified to the coordinate ring of $(\mathrm{Res}_{L/\Q_p}\overline{\bN})\otimes_{\Q_p}E$. By \cite[Lem. 2.5.8]{emerton2007jacquetII}, $\cC^{\rm pol}(I\cap \ovN,E)$ is a $\cD(\frg,\ovB)$-module. Actually, by the explicit description in \textit{loc. cit.} (see also Remark \ref{remark-dualverma} below), its restriction to $U(\frg)$ is the dual Verma module $\ovM(\lambda)^{\vee}$ in the category $\cO^{\overline{\frb}}$ (the BGG dual of the Verma module of the highest weight $\lambda$) when restricted to $\frg$. In the following, we write $\ovM(\lambda)^{\vee}$ for the $\cD(\frg,\ovB)$-module $\cC^{\rm pol}(I\cap \ovN,E)$. 

More generally, by \cite[Prop. 3.6]{breuil2015versII}, there exists a natural injection of $(\frg,\ovB\cap I)$-modules (where $M\in\cO^{\frb}_{\rm alg}$ and $\chi_{\rm sm}$ is a smooth character of $T$, see more in \S\ref{subsectionbeyonddualverma}),
\begin{equation}\label{equation-breuiladjunction}
     \Hom_E(M,E)^{\frn_{\overline{B}}^{\infty}}\otimes_E \chi_{\mathrm{sm}}\hookrightarrow  \Hom_E(M,E)^{\frn_{\overline{B}}^{\infty}}\otimes \cC_c^{\infty}(\overline{N},\chi_{\mathrm{sm}}) \subset \cF_B^G(M,\chi_{\mathrm{sm}}). 
\end{equation}
Here $\Hom_E(M,E)$ is a left $\frg$-module induced by the $\frg$-action on $M$ and the involution of $\frg$ given by multiplying $-1$. The superscript $\frn_{\ovB}^{\infty}$ denotes the subspace of elements killed by finite powers of $\frn_{\ovB}=\overline{\frn}$. The first injection of (\ref{equation-breuiladjunction}) sends $\chi_{\mathrm{sm}}$ to the constant function $\mathbf{1}_{I\cap \ovN}$ on $I \cap \ovN$. And the image of the last injection is the space of locally polynomial functions compactly supported on $\overline{N}$ inside $\cF_B^G(M,\chi_{\mathrm{sm}})$ \cite[(16) \& Prop. 3.6]{breuil2015versII}. The image of $\Hom_E(M,E)^{\frn_{\overline{B}}^{\infty}}\otimes_{E}\chi_{\rm sm}$ in $\cW_{\sharp,r}$ (see \S \ref{subsectionexamples}) consists of polynomial functions on $I \cap \ovN$.

\begin{rem}\label{remark-dualverma}
    We have $\Ind_B^G\chi=\cF_B^G(M(-\lambda),\chi_{\mathrm{sm}})$ in the notation of \cite{orlik2015jordan}. There are isomorphisms of $U(\frg)$-modules (in the last identity the involution of multiplying $-1$ of $\frg$ switches left and right $U(\frb)$-modules)
    \[\cC^{\pol}(I\cap \ovN,E)\simeq \Sym^{\bullet}(\overline{\frn}^{\vee})\simeq \Hom_E(U(\frg)\otimes_{U(\frb)}\chi^{-1},E)^{\overline{\frn}^\infty}=\Hom_{U(\frb)}(U(\frg),\chi)^{\overline{\frn}^\infty}.\]
    Note that $\Hom(U(\frg)\otimes_{U(\frb)}(-\lambda),E)^{\frn_{\overline{B}}^{\infty}}$ is the dual Verma module in $\cO^{\overline{\frb}}$ for the Verma module $U(\frg)\otimes_{U(\overline{\frb})}\lambda \in \cO^{\overline{\frb}}$. (According to \cite[\S3.2]{humphreys2008representations}, $(U(\frg)\otimes_{U(\overline{\frb})}\lambda)^{\vee}$ is constructed by taking the left $U(\frg)$-module $\Hom_E(U(\frg)\otimes_{U(\overline{\frb})}\lambda,E)^{\frn_{B}^{\infty}}$ and then composing with the Chevalley involution (inverse transpose) of $\frg$.)
\end{rem}

Since $\cW_{\sharp,r}$, as well as its dual, is a $\cD_r(I,E)$-module, extending the module structure over $\cD(I,E)$, the $\cD(\frg, \ovB\cap I)$-map $\ovM(\lambda)^{\vee}\hookrightarrow \cW_{\sharp,r}$ from (\ref{equation-breuiladjunction}), extends to a continuous map
\[i_{\cW_{\natural,r}}: \cW_{\natural,r}:=\cW_{\natural,r}(\overline{M}(\lambda)^{\vee})=\cD_r(I,E)\otimes_{\cD(\frg,\ovB\cap I)}\overline{M}(\lambda)^{\vee}\rightarrow \cW_{\sharp,r}\subset \Ind_{B}^G\chi.\]

\begin{lemma}\label{lemma-psitequivariant}
    The map $i_{\cW_{\natural,r}}:\cW_{\natural,r}\rightarrow \cW_{\sharp,r}$ is $\psi_{t}$-equivariant for $t\in T^{-}$, where the $\psi_t$-actions on $\cW_{\natural,r}$ are given by Lemma \ref{lemma-psit} for the $\cD(\frg,\ovB)$-module $\cC^{\rm pol}(I\cap \ovN,E)\simeq \cC^{\rm pol}(\ovN,E)$.
\end{lemma}
\begin{proof}
    Write $\ovM$ for $\ovM(\lambda)^{\vee}$. For $f\in \cC^{\pol}(\ovN,E)\stackrel{\mathrm{Res}_{BI}}{\rightarrow }\cC^{\pol}(I\cap \ovN,E)$, $\psi_t$ is just the action of $t^{-1}$ on the $\cD(\frg,\ovB)$-module $\cC^{\pol}(\ovN,E)$ (composed with $\mathrm{Res}_{BI}$ on $\cC^{\pol}(I\cap \ovN,E)$). Hence the map $\ovM \hookrightarrow \Ind_{B}^G\chi(BI)$ from (\ref{equation-breuiladjunction}) is $\psi_t$-equivariant by Definition \ref{definitinpsit}. By continuity, the same holds for $\ovM_r\subset \cW_{\natural,r}(\ovM)\rightarrow \Ind_{B}^G\chi(BI)$. For $[g,f]\in \cW_{\natural,r}(\ovM)=\cind_{H^m(\ovB\cap I)}^I\ovM_r$, the element $\psi_t(i_{\cW_{\natural,r}}([g,f]))\in \Ind_{B}^G\chi(BI)$, where $g\in N\cap I$, is equal to $f(-t^{-1}g)\cdot \mathbf{1}_{BI}$. While (see Lemma \ref{lemmapsitcompactinduction})
    \[i_{\cW_{\natural,r}}([t^{-1}gt,\psi_t(f)])=(t^{-1}gt)(f(-t^{-1})\mathbf{1}_{BI})=f(-(t^{-1}gt)t^{-1})\mathbf{1}_{BI}=f(-t^{-1}g)|_{BI}\]
    since $t^{-1}gt\in I$ if $g\in N\cap I$ and $t\in T^{-}$. Hence $i_{\cW_{\natural,r}}$ is $\psi_t$-equivariant.
\end{proof}
\begin{prop}\label{propositioninjectiondualverma}
    The map $\Phi_{\cW_{\natural,r}}: \cind_I^G\cW_{\natural,r}\rightarrow \Ind_B^G\chi$ induced by $i_{\cW_{\natural,r}}$ factors through an injection (where $\cH,\frm$ are given in \S\ref{subsectionKoszulcomplex} )
    \[\cind_I^G\cW_{\natural,r}\otimes_{\cH}\cH/\frm\hookrightarrow \Ind_B^G\chi.\]
\end{prop}
\begin{proof}
    By the discussion before Lemma \ref{lemmaexact}, $\cind_I^G\cW_{\natural,r}=\cind_{I'}^G (\overline{M}(\lambda)^{\vee})_r$ where $I'=H^m(\ovB\cap I)$ contains $I\cap \ovB$ with the Iwahori decomposition. Since the map $ \overline{M}(\lambda)^{\vee}\hookrightarrow \cW_{\sharp,r}$ is injective, so is the map $(\overline{M}(\lambda)^{\vee})_r\hookrightarrow \cW_{\sharp,r}$ as $(\overline{M}(\lambda)^{\vee})_r$ has the same length over $U_{r}(\frg)$ as the length of $\overline{M}(\lambda)^{\vee}$ over $U(\frg)$ by Lemma \ref{lemmaexact}. Moreover, the $\psi_t$-actions on $\cW_{\natural,r}$ given by Lemma \ref{lemma-psit} is the one in \S \ref{subsectionchangethegroups} induced from the $\psi_t$-actions on $ (\overline{M}(\lambda)^{\vee})_r$. Apply (2) of Proposition \ref{propositionchangethegroups1} and Corollary \ref{corollarychangethegroupinjection} for $\cW'=(\overline{M}(\lambda)^{\vee})_r$, we get the injectivity.
\end{proof}
\begin{rem}
    If $\ovM(\lambda)^{\vee}$ is simple, then the module $\cW_{\natural,r}$ is topologically of finite length as a $\cD_r(I,E)$-module, with the same length as the $U(\frg)$-module $\overline{M}(\lambda)^{\vee}$ by results in \cite{orlik2015jordan}. In this case the map $i_{W_{\natural,r}}$ is injective and Proposition \ref{propositioninjectiondualverma} follows directly from Proposition \ref{propositionKSresolutioninjection}. But this is not true in general. It may happen that $\ovM\subset \ovM(\lambda)^{\vee}$ is a $\cD(\frg,\ovP\cap I)$-module for some parabolic subgroup $\ovP\supsetneq \ovB$ (for example if $\chi_{\rm sm}$ is trivial, $\ovM=E$ and $\ovP=G$). Then the map $\cD_r(I,E)\otimes_{\cD(\frg,\ovB\cap I)}\overline{M}\rightarrow \cW_{\sharp,r}$ factors through the quotient $\cD_r(I,E)\otimes_{\cD(\frg,\ovP\cap I)}\overline{M}$.
\end{rem}
\subsection{Beyond dual Verma modules}\label{subsectionbeyonddualverma}
We extend the Kohlhaase-Schraen resolutions for principal series in the last sections to representations $\cF_{B}^G(M)$ constructed from some $\cD(\frg,B)$-modules $M$ in \cite{orlik2015jordan} extended in \cite[Appendice]{breuil2016vers} or \cite[\S 4]{schmidt2016dimensions} that we will review below.

Consider a $\cD(\frg,B)$-module $M$ such that $M\in \cO^{\frb}_{\rm alg}$ and such that we can choose a presentation \[0\rightarrow \frd\rightarrow U(\frg)\otimes_{U(\frb)}\sigma\rightarrow M\rightarrow 0\] 
for a finite dimensional $B$-representation $\sigma$. For such a presentation and $r$ sufficiently close to $1$ (which we will assume from now on), consider 
\[\cW_{\sharp,r}(\sigma^{\vee}):=(\Ind_{B}^G\sigma^{\vee})(BI)
\cap \cC^r(I,\sigma)\subset \Ind_{B}^G\sigma^{\vee}\]
where $\sigma^{\vee}=\Hom_E(\sigma,E)$ with the subspace topology of $\cC^r(I,\sigma)$. The continuous $E$-dual of $\cW_{\sharp,r}(\sigma^{\vee})$ is equal to $\cD_r(I,E)\otimes_{D(B,E)}\sigma$. By definition 
\[\cF_{B}^G(M)=(\Ind_{B}^G\sigma^{\vee})^{\frd}=((\cD(G,E)\otimes_{D(B,E)}\sigma)/\cD(G,E)\frd)'\] 
where the last superscript denotes the strong dual. Explicitly, the action of $\frx\otimes y \in U(\frg)\otimes \sigma$ on $\Ind_{B}^G\sigma^{\vee}$ sends $f:G\rightarrow \sigma^{\vee}$ to $g\mapsto (\frx\cdot f)(g)(y)\in \cC^{\la}(G,E)$ where the Lie algebra action $\frx\cdot$ is induced by left $G$-translations. We also consider 
\[\cW_{\natural,r}(\sigma^{\vee}):=\cD_r(I,E)\otimes_{\cD(\frg,I\cap\ovB)}\cC^{\pol}(I\cap \ovN,\sigma^{\vee})\]
where $\cC^{\pol}(I\cap \ovN,\sigma^{\vee})$ is the space of $(I\cap \ovN)$-polynomial functions supported on $I$ and with values in $\sigma^{\vee}$. Then $\cC^{\pol}(I\cap \ovN,\sigma^{\vee})=\Hom_{U(\frb)}(U(\frg),\sigma^{\vee})^{\overline{\frn}^{\infty}}=\Hom_E(U(\frg)\otimes_{U(\frb)}\sigma,E)^{\overline{\frn}^{\infty}}$ (see \cite[(2.5.7)]{emerton2007jacquetII} or discussions around (\ref{equation-breuiladjunction})). As in Lemma \ref{lemma-psitequivariant}, there is a natural $\psi_t$-equivariant $\cD_r(I,E)$-map $i_{\cW_{\natural,r}(\sigma^{\vee})}:\cW_{\natural,r}(\sigma^{\vee})\rightarrow \cW_{\sharp,r}(\sigma^{\vee})$ where we equip $\cW_{\sharp,r}(\sigma^{\vee})$ with the $\psi_t$-actions given by Definition \ref{definitinpsit}.

There is a $\psi_t$-stable sub-$D_r(I,E)$-module
\[\cW_{\natural,r}(\tau(M)^{\vee})=\cD_r(I,E)\otimes_{\cD(\frg,\ovB\cap I)}\Hom_E(M,E)^{\overline{\frn}^{\infty}}\] 
of $\cW_{\natural,r}(\sigma^{\vee})$ where for short we write $\tau(M)^{\vee}:=\Hom_E(M,E)^{\overline{\frn}^{\infty}}$ with the usual dual left actions of $U(\frg)$ and $B$. Consider also the $I$-subspace
\[\cW_{\sharp,r}(\tau(M)^{\vee}):=\cW_{\sharp,r}(\sigma^{\vee})^{\frd}=\cW_{\sharp,r}(\sigma^{\vee})\cap \cF_{B}^G(M)\]
of $ \cW_{\sharp,r}(\sigma^{\vee})$ which is $\psi_t$-stable (since the action of $\frd$ commutes with right translations and restrictions on $G$ which are used to define $\psi_t$) and has the continuous dual given by $\cD_r(I,E)\otimes_{\cD(\frg,B\cap I)}M$ (see the proof of \cite[Prop. 3.11]{orlik2015jordan}).
\begin{lemma}
    The composite map $\cW_{\natural,r}(\tau(M)^{\vee})\subset \cW_{\natural,r}(\sigma^{\vee})\stackrel{i_{\cW_{\natural,r}(\sigma^{\vee})}}{\rightarrow} \cW_{\sharp,r}(\sigma^{\vee})$ factors through a $\psi_t$-equivariant $I$-map $i_{\cW_{\natural,r}(\tau(M)^{\vee})}:\cW_{\natural,r}(\tau(M)^{\vee})\rightarrow \cW_{\sharp,r}(\tau(M)^{\vee})\subset \cW_{\sharp,r}(\sigma^{\vee})$. Moreover, $i_{\cW_{\natural,r}(\tau(M)^{\vee})}$ is independent of the presentations of $M$ and is induced by (\ref{equation-breuiladjunction}).
\end{lemma}
\begin{proof}
    Since $\cW_{\natural,r}(\tau(M)^{\vee})$ is generated as a $\cD_r(I,E)$-module by $\tau(M)^{\vee}$ and $\cW_{\sharp,r}(\tau(M)^{\vee})$ is a $\cD_r(I,E)$-submodule, it's enough to show that $\Hom_E(M,E)^{\overline{\frn}^{\infty}}\subset \cC^{\pol}(I\cap \ovN,\sigma^{\vee})$ is contained in the subspace $\cW_{\sharp,r}(\sigma^{\vee})^{\frd}\cap \cC^{\pol}(I\cap \ovN,\sigma^{\vee})=\cC^{\pol}(I\cap \ovN,\sigma^{\vee})^{\frd}$. By \cite[Lem. 3.3]{breuil2015versII}, $\cC^{\pol}(I\cap \ovN,\sigma^{\vee})^{\frd}=\Hom_E(U(\frg)\otimes_{U(\frb)}\sigma,E)^{\overline{\frn}^{\infty},\frd}=\Hom_E((U(\frg)\otimes_{U(\frb)}\sigma)/\frd,E)^{\overline{\frn}^{\infty}}$. The factorization follows. The independence of the presentations follows from \cite[Prop. 3.4 (ii)]{breuil2015versII}.
\end{proof}
The $\psi_t$-equivariant map $i_{\cW_{\natural,r}(\sigma^{\vee})}$ induces an $\cH$-map 
\[\cind_{I}^G\cW_{\natural,r}(\sigma^{\vee})\rightarrow \cind_{I}^G\cW_{\sharp,r}(\sigma^{\vee})\]
and we have $G$-maps (see Proposition \ref{propositionKSresolutioninjection}) 
\[\Phi_{\cW_{?,r}(\sigma^{\vee})}:\cind_{I}^G\cW_{?,r}(\sigma^{\vee})\rightarrow \cind_{I}^G\cW_{?,r}(\sigma^{\vee})\otimes_{\cH}\cH/\frm\rightarrow \Ind_{B}^G\sigma^{\vee}\]
for $?\in\{\natural,\sharp\}$.
By the above lemma, the map restrict to an $\cH$-map
\[\cind_{I}^G\cW_{\natural,r}(\tau(M)^{\vee})\rightarrow \cind_{I}^G\cW_{\sharp,r}(\tau(M)^{\vee})\]
and finally (use Corollary \ref{corinjectiontwoW})
\[\Phi_{\cW_{?,r}(\tau(M)^{\vee})}:\cind_{I}^G\cW_{?,r}(\tau(M)^{\vee})\rightarrow \cind_{I}^G\cW_{?,r}(\tau(M)^{\vee})\otimes_{\cH}\cH/\frm\rightarrow \cF_{B}^G(M).\]

To simplify discussions, we will focus more on subquotients of principal series. Suppose that $\chi=z^{\lambda}\chi_{\rm sm}$ is a locally algebraic character of $T$ where $\lambda$ is the weight of $\chi$ and $\chi_{\rm sm}$ is the smooth part. Then the principal series $\Ind_{B}^G\chi$ can be written as $\cF_B^G(M)$ for $M=D(\frg,B)\otimes_{D(B)}\chi^{-1}=(U(\frg)\otimes_{U(\frb)}(-\lambda))\otimes_{E}\chi_{\rm sm}^{-1}$ where $B$ acts diagonally (on the first factor by integration) and $\frg$ acts on the second factor trivially. In the notation of \cite{orlik2015jordan}, $\cF_{B}^G(M\otimes_{E}\chi_{\rm sm}^{-1})=\cF_{B}^G(M,\chi_{\sm})$ for $M\in\cO^{\frb}_{\rm alg}$ with the algebraic $B$-action. Note that $\tau(M\otimes_{E}\chi_{\rm sm}^{-1})^{\vee}=\tau(M)^{\vee}\otimes_{E}\chi_{\rm sm}$.

\begin{prop}\label{propositioncomplexesseq}
    Suppose that $0\rightarrow M_0\rightarrow M\rightarrow M_1\rightarrow 0$ is a short exact sequence of $U(\frg)$-modules in $\cO^{\frb}_{\rm alg}$ and let $\chi_{\rm sm}$ be a smooth character of $T$. Then for $r$ sufficiently close to $1$, we have a short exact sequence of complexes for $?=(\natural/\sharp, r)$:
    \begin{align*} 0\rightarrow \wedge^{\bullet} \cH^{d}\otimes_{\cH} \cind_{I}^G\cW_{?}(\tau(M_1)^{\vee}\otimes\chi_{\rm sm})\rightarrow &\wedge^{\bullet} \cH^{d}\otimes_{\cH} \cind_{I}^G\cW_{?}(\tau(M)^{\vee}\otimes\chi_{\rm sm})\\
    &\rightarrow \wedge^{\bullet} \cH^{d}\otimes_{\cH} \cind_{I}^G\cW_{?}(\tau(M_0)^{\vee}\otimes\chi_{\rm sm})\rightarrow 0.
    \end{align*}
    Moreover, we have a commutative diagram 
        \begin{center}
            \begin{tikzcd}[sep=small]
                0\arrow[r]&\cind_{I}^G\cW_{?}(\tau(M_1)^{\vee}\otimes\chi_{\rm sm})/\frm\arrow[r]\arrow[d] & \cind_{I}^G\cW_{?}(\tau(M)^{\vee}\otimes\chi_{\rm sm})/\frm \arrow[r] \arrow[d]
                & \cind_{I}^G\cW_{?}(\tau(M_0)^{\vee}\otimes\chi_{\rm sm})/\frm \arrow[d]\arrow[r]&0 \\
                0\arrow[r]&\cF_{B}^G(M_1,\chi_{\rm sm})\arrow[r]&\cF_{B}^G(M,\chi_{\rm sm})\arrow[r] 
                & \cF_{B}^G(M_0,\chi_{\rm sm}) \arrow[r]&0
            \end{tikzcd}
        \end{center}	
    where the rows are exact and all vertical arrows are injective. Moreover if $?=(\sharp,r)$, the vertical maps are isomorphisms.
\end{prop} 
\begin{proof}
    The functoriality of $\cW_{?}$ is similar as in \cite[Prop. 4.7]{orlik2015jordan} or by \cite[Prop. 3.4]{breuil2015versII}. The exactness of taking $\cW_{?}$ is the proof of \cite[Prop. 4.2]{orlik2015jordan} or Lemma \ref{lemmaexact}. By Proposition \ref{propositionKSresolutionKoszulregular}, the homologies of the Koszul complexes concentrate in degree $0$. Hence we get the short exact sequence of degree $0$ homologies. If $M$ is a Verma module, then $\Phi_{\cW_{\natural,r}(\tau(M)^{\vee}\otimes\chi_{\rm sm})}$ induces an injection $\cind_{I}^G\cW_{\natural,r}(\tau(M)^{\vee}\otimes \chi_{\rm sm})\otimes_{\cH}\cH/\frm\hookrightarrow \cF_{B}^G(M,\chi_{\rm sm})$ by Proposition \ref{propositioninjectiondualverma}. As any irreducible $M$ in $\cO^{\frb}_{\rm alg}$ is a quotient of a Verma module, the injectivity of general $M$ follows by the snake lemma and an induction on the length of $M$. If $?=(\sharp,r)$, the surjectivity follows similarly using the result of Kohlhaase-Schraen, i.e., Proposition \ref{propositionKSresolutionsujectivity}.
\end{proof}
\subsection{The surjectivity}\label{subsectionsurjectivity}
We prove the surjectivity of the map
\[\Phi_{\cW_{\natural,r}}:\cind_I^G \cW_{\natural,r}(\overline{M}(\lambda)^{\vee})\rightarrow \Ind_{B}^G\chi\]
constructed in \S\ref{subsectionresolutionprincipalseries} for locally algebraic characters $\chi$ using Proposition \ref{propositioncompletionlocalization} that will be proved in \S\ref{sectioncompletion} later. Then we can deduce from the surjectivity the main theorem on Kohlhaase-Schraen resolutions (Theorem \ref{theoremresolution}).

We introduce some notations which serve only for the proof below. The exponential map $\exp:\overline{\frn}\rightarrow \ovN$ is an isomorphism of $L$-analytic spaces. Choose basis of $\overline{\frn}$ in the root spaces compatible with the $\cO_L$-lattice we obtain the $L$-analytification $\overline{\frn}^{\rm an}\simeq \mathbf{A}^{\dim_{L}N}$. We obtain an isomorphism of the analytic spaces $\overline{\bN}^{\rm an}\simeq \mathbf{A}^{\dim_{L}N}$ with the ball of radius $s$ denoted by $\overline{\bN}^{\rm an}_{\leq s}$, satisfying that $\overline{\bN}^{\rm an}_{\leq 1}(L)=\ovN\cap K$. For $h\in \mathbb{Z}$, let $\ovN_h=\overline{\bN}^{\rm an}_{\leq |\varpi_L|_p^{h}}(L)$ which may or may not be a group. Using the Baker-Campbell-Hausdorff formula, $\ovN_h$ is a normal subgroup of $\ovN\cap I$ for $h\geq 0$ large enough. We say that a compactly supported locally analytic function $f\in \cC_c^{\la}(\ovN,E)$ is $h$-analytic for some $h\geq 1$ large enough if $f$ is rigid analytic on each left $\ovN_h$-coset. If $f$ is supported in $\ovN\cap I$, this is equivalent to that $f$ is rigid analytic on each right $\ovN_h$-coset since $\cC^{\rm an}(\ovN_hg,E)=\cC^{\rm an}(g\ovN_h,E)$ for any $g\in\ovN\cap I$. Let $\cA_h\subset (\Ind_B^G\chi)(BI)\simeq \cC^{\la}(I\cap \ovN,E)$ be the subspace of functions that are supported on $B\backslash BI$ and is $h$-analytic. For $h$ large enough, $\cA_h$ is the $I$-representation $\cW_{\sharp,h}$ in \S\ref{subsectionexamples}. We recall again the result of Kohlhaase-Schraen.
\begin{prop}[{\cite[Prop. 2.4]{kohlhaase2012homological}}]\label{propositionsurjectivitykohlhaase}
    The map $\cind_I^G\cA_h\rightarrow \Ind_B^G\chi$ is surjective for $h\geq 1$ large enough.
\end{prop}
\begin{prop}\label{propositionsurjectivityPS}
    The map $\Phi_{\cW_{\natural,r}}$ above is surjective if $\chi$ has integral and antidominant (with respect to $\frb$) weights.
\end{prop}
\begin{proof}
    Combine Proposition \ref{propositionsurjectivitykohlhaase} and Lemma \ref{lemma-surjectivity1an} below. The proof of Lemma \ref{lemma-surjectivity1an} will use Proposition \ref{propositioncompletionlocalization}. 
\end{proof}

\begin{lemma}\label{lemma-surjectivity1an}
    The image of $\Phi_{\cW_{\natural,r}}$ contains $\mathcal{A}_h\subset \Ind_B^G\chi(BI)$ the space of all $h$-analytic functions on $\cC^{\la}(I\cap \ovN,E)$ for any $h\geq 1$.
\end{lemma}
\begin{proof}
    The image of $\cW_{\natural,r}=(\cD_r(I,E)\otimes_{\cD(\frg,\ovB\cap I)}\overline{M}(\lambda)^{\vee})$ under the map $i_{\cW_{\natural,r}}$ contains the image of $U_r(\frg)\otimes_{U(\frg)}\overline{M}(\lambda)^{\vee}$, especially the space $\overline{M}(\lambda)^{\vee}=\cC^{\pol}(I\cap \ovN,E)\subset (\Ind_{B}^G\chi)(BI)$ of $\Q_p$-polynomial functions on $I\cap \ovN$.

    For $h\geq 0$, we write $\cC^{\an}(\ovN_{-h},E)\mathbf{1}_{I\cap \ovN}$ for the space of functions on $I\cap \ovN$ that are restrictions of rigid analytic functions on $\ovN_{-h}$ i.e., power series 
    \[\sum_{i\in \bbN^{\dim_{\Q_p}\ovN}}a_{i}\prod_{\alpha\in \Phi^-,j=1,\cdots,|L:\Q_p|}X_{\alpha,j}^{i_{\alpha,j}},a_{i}\in E\] that satisfy certain convergent condition where $X_{\alpha,j}$ are coordinates of the root subgroup $\mathrm{Res}_{L/\Q_p}\mathbf{G}_a$ corresponding to the negative roots $\alpha\in \Phi^-$. This is a Banach space.
    \begin{lemma}\label{lemma-surjectivityhan}
        The image of $U_r(\frg)\otimes_{U(\frg)}\overline{M}(\lambda)^{\vee}$ under the map $i_{\cW_{\natural,r}}$ contains $\cC^{\an}(\ovN_{-h},E)\mathbf{1}_{I\cap \ovN}$ for some large enough $h$.
    \end{lemma}
    \begin{proof}[Proof of Lemma \ref{lemma-surjectivityhan}]
        This is a reformulation of Proposition \ref{propositioncompletionlocalization}, which is proved for $L$-Lie algebras and $L$-analytic functions. We explain how to strengthen the result for $\Q_p$-Lie algebras. Write $\frg=\prod_{\sigma\in \Hom(L,\overline{\Q}_p)}\frg_{\sigma}$. As in \cite[App. 9]{breuil2016vers}, there exists a factorization $U(\frg)=\otimes_{\sigma}U(\frg_{\sigma})\rightarrow \widehat{\otimes}_{\sigma}U_{r'}(\frg_{\sigma})\hookrightarrow U_{r}(\frg)$. Similar factorization exists for $\cC^{\an}(\ovN_{-h},E)$ and $\widehat{\otimes}_{\sigma}\cC^{\sigma-\an}(\ovN_{-h'},E)$ where $\sigma$-an means $\sigma$-rigid analytic functions: we only need to consider $\cO_L\subset \Sp(L\langle T \rangle)$, the rigid unit ball. By definition $\cC^{L-\an}(\cO_L,E)=E\langle T \rangle$ while 
        \[\cC^{\an}(\cO_L,E)=\cO(\mathrm{Res}_{L/\Q_p}(\Sp(L\langle T \rangle))\otimes_{\Q_p}E)=\cO(\prod_{\sigma}\Sp(E\langle T \rangle))=\widehat{\otimes}_{\sigma}E\langle T \rangle,\] 
        see \cite[\S 2.3]{emerton2017locally}. The decomposition of the space of analytic functions into spaces of $\sigma$-analytic functions is compatible with the actions of $\sigma$-Lie algebras. The result follows. 
    \end{proof}
    We now prove that the image of $\Phi_{\cW_{\natural,r}}$ contains $\cA_{h_0}$ for any fixed large enough $h_0\geq 1$. Take $t\in T^{--}$, i.e., take $t$ such that for any $n\in\ovN$, $\varinjlim_{k\rightarrow\infty}t^knt^{-k}=0$. The action of $t^k\in T$ on $\Ind_{B}^G\chi,f(-)\mapsto f(-t^k)=\chi(t)^kf(t^{-k}- t^k)$ induces an isomorphism $\cC^{\an}(\ovN_{-h},E)\mathbf{1}_{I\cap \ovN}\simeq \cC^{\an}(t^k\ovN_{-h}t^{-k},E)\mathbf{1}_{t^k(I\cap \ovN)t^{-k}}$. 
    Hence the image of $\Phi_{\cW_{\natural,r}}$ contains $\cC^{\an}(t^k\ovN_{-h}t^{-k},E)\mathbf{1}_{t^k(I\cap \ovN)t^{-k}}$. Take $k$ large enough such that $t^k\ovN_{-h}t^{-k}\subset \ovN_{h_0}$. The inclusion $t^k\ovN_{-h}t^{-k}\subset \ovN_{h_0}$ is induced by an inclusion of rigid analytic subsets. Then we see the image of $\Phi_{\cW_{\natural,r}}$ contains $\cC^{\an}(\ovN_{h_0},E)\mathbf{1}_{t^k(I\cap \ovN)t^{-k}}$, all rigid analytic functions on $t^{k}(I\cap \ovN)t^{-k}$ that converge over $\ovN_{h_0}\supset t^k\ovN_{-h}t^{-k}$.
    
    For any rigid analytic function $f\in \cC^{\an}(\ovN_{h_0},E)\subset \cA_{h_0}$, write $f_g$ for the restriction of $f$ on each coset $t^k(I\cap \ovN)t^{-k}g, g\in  t^k(I\cap \ovN)t^{-k}\backslash \ovN_{h_0}$. The translation $g.f_g$ is in $\cC^{\an}(\ovN_{h_0},E)\mathbf{1}_{t^k(I\cap \ovN)t^{-k}}$. Using translations by $\ovN_{h_0}$, we see that the image of $\Phi_{\cW_{\natural,r}}$ contains $\cC^{\an}(\ovN_{h_0},E)$. 
    
    Any function in $\cA_{h_0}$ is a linear combination of translations by elements in $I\cap \ovN$ of such functions in $\cC^{\an}(\ovN_{h_0},E)$, hence is also in the image of $\Phi_{\cW_{\natural,r}}$. 
\end{proof}

\begin{example}\label{ExampleSL2}    
    We explain the idea for Lemma \ref{lemma-surjectivityhan} by calculating explicitly in the case $L=\Q_p$ and $\bG=\mathrm{SL}_2$. In this case $\frg=E.\langle e,f,g\rangle$ is spanned by
    \[e=\mtwo{0}{1}{}{0},f=\mtwo{0}{}{1}{0},h=\mtwo{1}{}{}{-1}.\] 
    We take norm $|-|_n$ on $U(\frg)$ such that $|e|_n=|f|_n=|h|_n=p^{n}$. And let $\widehat{U}(\frg)_n$ of $U(\frg)$ be the Banach completion with respect to this norm. 

    Take $x$ for the coordinate of $\ovN=\{\mtwo{1}{}{x}{1}\mid x\in\Q_p\}$ and $\ovN\cap I=\{\mtwo{1}{}{x}{1}\mid x\in p\Z_p\}$.

    Let $\chi$ be a character with weight $\lambda\in\overline{\Q}_p$. Consider $\ovM(\lambda)^{\vee}=\cC^{\rm pol}(\ovN,E)\simeq \cC^{\rm pol}(\ovN,E)\mathbf{1}_{I\cap \ovN} \subset \Ind_{B}^G\chi$. Then $\cC^{\rm pol}(I\cap \ovN,E)$ is spanned linearly by $x^k\mathbf{1}_{I\cap \ovN}$ for $k\in\bbN$. Lemma \ref{lemma-surjectivityhan} reduces to that the image of $\widehat{U}(\frg)_n\otimes_{U(\frg)} \cC^{\rm pol}(I\cap \ovN,E)$ contains
    \[ \cC^{\rm an}(\ovN_{-h},E)=\{\sum_{i}a_ix^i\mid a_i\in E, \varinjlim_i |a_i|p^{hi}=0\}\]
    for some $h$. The element $f$ acts on $\cC^{\rm pol}(\ovN,E)$ by the derivative for $x$. We calculate that
    \begin{align*}
        (e.x^k)(\mtwo{1}{}{x}{1})&=\lim_{t\rightarrow 0}\frac{1}{t}x^k(\mtwo{1}{}{x}{1}\mtwo{1}{t}{}{1})\\
        &=\lim_{t\rightarrow 0}x^k(\mtwo{\frac{1}{1+xt}}{t}{}{1+xt}\mtwo{1}{0}{\frac{x}{1+xt}}{1})\\
        &=\frac{d}{dt}|_{t=0}\chi(1+xt)^{-1}\frac{x^k}{(1+x t)^k}=-(\lambda+k)x^{k+1}.   
    \end{align*}
    And similarly $h.x^k=(\lambda+2k)x^k$. The $U(\frg)$-module $\ovM(\lambda)^{\vee}$ is generated by the highest weight (with respect to $\overline{\frb}$) vector $x^0$ and possibly $x^{k+1}$ if $\lambda=-k\in\Z_{\leq 0}$. Let $\alpha=0$ if $\lambda\notin \Z_{\leq 0}$ and $\alpha=-\lambda+1$ otherwise. The above formula shows that $x^{\alpha+i}=(-1)^{i}\prod_{j=0}^{i-1}\frac{1}{\lambda+\alpha+j} e^i.x^{\alpha}$ for $i\geq 1$.

    For $g=\sum_{i\geq 1}a_ix^{\alpha+i}\in \cC^{\rm an}(\ovN_{-h},E)$, we can rewrite it as
    \[g=(\sum_{i\geq 1} (-1)^ia_i\frac{1}{\prod_{j=0}^{i-1}(\lambda+\alpha+j)}e^i).x^{\alpha}.\]
    The sum $u:=\sum_{i\geq 1} (-1)^ia_i\frac{1}{\prod_{j=0}^{i-1}(\lambda+\alpha+j)}e^i$ converges in $\widehat{U}(\frg)_n$ if and only if $\varinjlim_i \frac{p^{-ni}a_i}{\prod_{j=0}^{i-1}(\lambda+\alpha+j)}=0$. Since $\varinjlim_i a_ip^{-hi}=0$, we find that the element $g=u.x^{\alpha}$ lies in the image of $\widehat{U}(\frg)_n\otimes_{U(\frg)}\ovM(\lambda)^{\vee}$ if 
    \[\varinjlim_i\frac{p^{(h-n)i}}{\prod_{j=0}^{i-1}(\lambda+\alpha+j)}=0.\]
    Such $h$ exists if $-(\lambda+\alpha)$ is of positive type (Definition \ref{definitionpositivetype}). This holds at least for all $\lambda\in\Z$ using that $|k!|_p= p^{-\frac{k-s(k)}{p-1}}\geq p^{-\frac{k}{(p-1)}}$ for a positive integer $k$.

    On the other hand, if $-(\lambda+\alpha)$ is not of positive type, we don't know whether Proposition \ref{propositionsurjectivityPS} is still true.
\end{example}

\begin{thm}\label{theoremresolution}
    Let $M\in\cO_{\rm alg}^{\frb}$ and $\chi_{\rm sm}$ be a smooth character of $T$, then for $?\in\{\natural,\sharp\}$, the resolution in Proposition \ref{propositioncomplexesseq} induces a quasi-isomorphism
    \[ \wedge^{\bullet} \cH^{d}\otimes_{\cH} \cind_{I}^G\cW_{?,r}(\tau(M)^{\vee}\otimes\chi_{\rm sm})\rightarrow \cF_{B}^G(M,\chi_{\rm sm})\] 
    of complexes of $\cD(G,E)$-modules.
\end{thm}
\begin{proof}
    This follows from  Proposition \ref{propositionsurjectivityPS}, Proposition \ref{propositionsurjectivitykohlhaase}, Proposition \ref{propositioncomplexesseq} and the snake lemma.
\end{proof}
\begin{rem}\label{remarkWnaturaln}
    One can consider $\cW_{\natural, n}(\ovM):=\cD^{I_n-{\mathrm{an}}}(I,E)\otimes_{\cD(\frg,\ovB\cap I)}\ovM$ instead of $\cW_{\natural, r}(\ovM)$. There are maps $\cW_{\natural, n}(\ovM)\rightarrow \cW_{\natural, r}(\ovM)\rightarrow \cW_{\natural, n'}(\ovM)$ for $r$ sufficiently close to $1$ and $n$ large enough. These maps are injections at least if $\ovM=\cD(\frg,\ovB)\otimes_{\cD(\ovB,E)}\sigma$ for a finite-dimensional $\ovB$-representation $\sigma$ and induce quasi-isomorphisms (using Proposition \ref{propositionKSresolutionKoszulregular} and Corollary \ref{corinjectiontwoW})
    \[\wedge^{\bullet} \cH^{d}\otimes_{\cH} \cind_{I}^G\cW_{\natural,n}(\ovM)\simeq \wedge^{\bullet} \cH^{d}\otimes_{\cH} \cind_{I}^G\cW_{\natural,r}(\ovM)\]
    between $n$ and $r$-versions. The quasi-isomorphisms also hold for general $\ovM$ in Theorem \ref{theoremresolution} by considering a presentation of $\ovM$ as in the proof of Lemma \ref{lemma-psit} and the right exactness of $-\otimes_{\cH}\cH/\frm$.
\end{rem}
\subsection{Duality for complexes}\label{subsectiondualitycomplex}
We can already observe certain duality between different Koszul resolutions in Theorem \ref{theoremresolution}. At present, we will establish the duality only for complexes (Theorem \ref{theoremdualKS}). The duality between locally analytic representations will be discussed later in \S\ref{sectiondualitysolid}.

We continue with the notation in \S\ref{subsectionbeyonddualverma}. The continuous linear $E$-dual $\cW_{\sharp,r}(\sigma^{\vee})^{\vee}$ of $\cW_{\sharp,r}(\sigma^{\vee})$ is $\cD_r(I,E)\otimes_{D(B\cap I,E)}\sigma$ and the dual of $\cW_{\sharp,r}(\tau(M)^{\vee})$ is $\cD_r(I,E)\otimes_{\cD(\frg,B\cap I)}M$. We adapt the convention that $g.f(-)=f(g^{-1}-)$ for $f\in \Hom_E(\cW,E)$ for an $I$-representation $\cW$ and $g\in I$. As in Lemma \ref{lemma-psit}, for $t\in T^{-}$, there is a $\psi_{t^{-1}}$ operator on $\cD_r(I,E)\otimes_{\cD(\frg,B\cap I)}M$ extending the action of $\delta_{t}$ on $M$. We obtain also $U_{t^{-1}}$ operators on $\cind_{I}^G\cD_r(I,E)\otimes_{\cD(\frg,B\cap I)}M$, as in \S\ref{subsectionKoszulcomplex}.

\begin{lemma}\label{lemmatransposepsit}
    Let $t\in T^{-}$. The transpose $\psi_{t}^{\vee}$ of $\psi_t$ on $\cW_{\sharp,r}(\tau(M)^{\vee})^{\vee}=\cD_r(I,E)\otimes_{\cD(\frg,B\cap I)}M$ coincides with $\psi_{t^{-1}}$.
\end{lemma}
\begin{proof}
    Let $H^m$ be the subgroup after Definition \ref{definitionWnaturalM} such that $\cD_r(I,E)=E[I]\otimes_{E[H^m]}U_r(\frg)$. The canonical pairing $\cD(I,E)\times \cC^{\la}(I,E)\rightarrow E$ is $I$-equivariant of left $I$-modules (the left $I$-action on $\cD(I,E)$ comes from the transpose of its left action on $\cC^{\la}(I,E)$ composed with the inverse involution). The pairing refines to pairings $\delta_gD(H^m,E)\times \delta_g\cC^{\la}(H^m,E)\rightarrow E$ for $g\in I/H^m$ where $\delta_g\in E[I]$ is the Dirac distribution. Hence the dual of $\delta_gU_r(\frg)$ can be identified with the direct summand $\delta_g\cC^r(H^m,E)=\cC^r(H^mg,E)\subset \cC^r(I,E)$ for $g\in I/H^m$. We see for $m\in M_r\subset \cD_r(I,E)\otimes_{\cD(\frg,B\cap I)}M, g\in I$ and $f\in \cW_{\sharp,r}(\tau(M)^{\vee})\subset \Ind_B^G\sigma^{\vee} (BI)$, we have $m(f)=m(f\mathbf{1}_{BH^m})$ and $(\delta_gm)(f)=m(g^{-1}.f)$.

    Hence for $m\in M\subset \cD(G,E)\otimes_{\cD(\frg,B)}M,g\in I$, 
    \[(\psi_t^{\vee}(\delta_gm))(f)=\delta_gm(\psi_tf)=\delta_gm((t^{-1}.f)\mathbf{1}_{BI})=m(((g^{-1}t^{-1}).f) \mathbf{1}_{BI})=m((g^{-1}t^{-1}).f),\]
    where the last equality is calculated for $g^{-1}t^{-1}f\in \Ind_B^G\sigma^{\vee}$. By the construction (see Lemma \ref{lemmapsitcompactinduction}), $\psi_{t^{-1}}(\delta_gm)=\delta_{tgt^{-1}}t.m$ for $g\in I\cap \ovB$. Hence
    \[(\psi_{t^{-1}}(\delta_gm))(f)=(\delta_{tgt^{-1}}t.m)(f)=(t.m)(tg^{-1}t^{-1}f)=m(g^{-1}t^{-1}f).\]

    We see $\psi_t^{\vee}=\psi_{t^{-1}}$ on $\delta_g.M$ for $g\in I\cap \ovB$. Since $M$ is dense in $M_r$, $\sum_{g\in I\cap \ovB}\delta_g.M$ is dense in $\cD_r(I,E)\otimes_{\cD(\frg,B\cap I)}M$ and both operators are continuous, the result follows.
\end{proof}

We discuss some generality of Hecke operators. Let $\cW^{\vee}$ be the (continuous) dual of an $I$-representation $\cW$ (with the open compact topology). Suppose that there are $\psi_t$-actions on $\cW$ which induce $U_t$-actions on $\cind_I^G\cW$. For $t\in T^{-}$, let $U_t^{\vee}$ be the transpose of $U_t$ acting on $\cind_{I}^G\cW^{\vee}$ under the following $G$-invariant pairing 
\begin{equation}\label{equataionpairing}
    \cind_{I}^G \cW^{\vee}\times \cind_{I}^G \cW\rightarrow E: (\sum_{g\in G/I} [g,f_g],\sum_{g\in G/I} [g',w_{g'}])\mapsto \sum_{g\in G/I} f_g(w_g),
\end{equation}
such that $(x,U_{t}y)=(U_t^{\vee}x,y)$ (see Lemma \ref{lemmatransposeUt} or Lemma \ref{lemmadualcompactinduction} below for the existence of the transpose). Let $\psi_{t^{-1}}:=\psi_t^{\vee}$ be the transpose of $\psi_t$ on $\cW^{\vee}$ so that $\psi_{t^{-1}}(x.f)=(txt^{-1})\psi_{t^{-1}}f$ for $x\in I\cap t^{-1}It$. This allows us to define the corresponding Hecke operator $U_{t^{-1}}$ on $\cind_I^G\cW^{\vee}$.
 
\begin{lemma}\label{lemmatransposeUt}
    The transpose $U_t^{\vee}$ of $U_t$ is equal to $U_{t^{-1}}$.
\end{lemma}
\begin{proof}
    We calculate that for $f\in\cW^{\vee},g\in G$ and $w\in \cW$,
    \[\begin{array}{rl}
            U_t^{\vee}([1,f])([g,w])&=[1,f](\sum_{x\in I/(I\cap tIt^{-1})}[gxt,\psi_t(x^{-1}w)])\\
                                &=\sum_{x\in I/(I\cap tIt^{-1})\cap g^{-1}It^{-1}}f((gxt)\psi_t(x^{-1}w))
    \end{array}\]
    The intersection $I\cap g^{-1}It^{-1}\neq \emptyset$ if and only if $g\in It^{-1}I$. And for $g=x't^{-1}$ where $x'\in I$ and $x\in I$, $gxt=x't^{-1}xt$ is in $I$ if and only $x\in I\cap tIt^{-1}$. Thus $U_t^{\vee}([1,f])=\sum_{x'\in I/(I\cap t^{-1}It)}[x't^{-1},(w\mapsto f(x'\psi_t(w)))]=\sum_{x'\in I/(I\cap t^{-1}It)}[x't^{-1},\psi_{t^{-1}}(x'^{-1}.f)]=U_{t^{-1}}([1,f])$.
\end{proof}

We consider the left and right $G$-actions on the space $\cC_c^{\la}(G,E)$ of compactly supported locally analytic functions on $G$ given by right and left translations: $(l_gr_hf)(-)=f(g- h)$. For a locally analytic representation $\pi$, we write 
\[\bbD^0_{\rm BZ}(\pi):=\Hom_G(\pi,\cC_{c}^{\la}(G,E))=\Hom_G(E,\Hom_E(\pi,\cC_{c}^{\la}(G,E)))\] where $G$ acts on $f\in \Hom_E(\pi,\cC_c^{\la}(G,E))$ by $g.f(-)=r_g(f(g^{-1}-))$ and we take the continuous Hom's. The space $\bbD^0_{\rm BZ}(\pi)$ is a right $\cD(G,E)$-module induced by the left translation on the target $\cC_{c}^{\la}(G,E)$ and we make it into a left $\cD(G,E)$-module using the involution $g\mapsto g^{-1}$ of $G$, or consider the left action of $G$ on $\cC_{c}^{\la}(G,E)$ by $g.f=f(g^{-1}-)$. We will ignore the topology of $\bbD^0_{\rm BZ}(\pi)$ in this section (we refer to Theorem \ref{theoremsoliddualcompactinduction} for the discussion on  ``topology'').
\begin{lemma}\label{lemmadualcompactinduction}
    Let $\cW$ be a locally analytic $\cD(I,E)$-module over a Smith/Banach space (so that $\cW^{\vee}$ is a Banach/Smith space). 
    \begin{enumerate}
        \item We have $\bbD^0_{\rm BZ}(\cind_{I}^G\cW)=\cind_{I}^G\cW^{\vee}$. And under this identification, $[g,f]([h,w])=(g'\in gIh^{-1}\mapsto f(g^{-1}g'hw))\in\cC^{\la}(gIh^{-1},E)\subset \cC^{\la}_c(G,E)$.
        \item If $U: \cind_{I}^G\cW\rightarrow \cind_I^G\cW$ is a $G$-equivariant map, then $\bbD^0_{\rm BZ}(U):\cind_{I}^G\cW^{\vee}\rightarrow \cind_I^G\cW^{\vee}$ under the identification in (1) is the transpose of $U$ under the pairing (\ref{equataionpairing}).
    \end{enumerate}
\end{lemma}
\begin{proof}
    (1) $\Hom_{G}(\cind_{I}^G\cW, \cC_{c}^{\la}(G,E))=\Hom_I(\cW,\cC^{\la}(I,E)\otimes_{E[I]}E[G])$. Since $\cW$ is a Banach or Smith space over $E$ and $\cC^{\la}_c(G,E)=\oplus_{g\in I\backslash G}l_g\cC^{\la}(I,E)$, we have (\cite[Cor. 8.9]{schneider2013nonarchimedean} or \cite[Lem. 3.32]{rodrigues2022solid})
    \[\Hom_E(\cW,\cC^{\la}(I,E))=\oplus_{g\in I\backslash G}\Hom_E(\cW,\cC^{\la}(g^{-1}I,E)).\]
    Taking $I$-invariants
    \[
       \Hom_I(\cW,\cC^{\la}(I,E)\otimes_{E[I]}E[G])=\Hom_I(\cW,\cC^{\la}(I,E))\otimes_{E[I]}E[G].
    \]
    Using \cite[Thm. 3.40]{rodrigues2022solid}, there is an $I$-isomorphism 
    \[\Hom_I(\cW,\cC^{\la}(I,E))=\Hom_{\cD(I,E)}(\cD(I,E),\cW^{\vee})=\cW^{\vee}\] 
    which sends $f\in \cW^{\vee}$ to $w\in \cW\mapsto (g\mapsto f(gw))$. Finally, for $f\in \cW^{\vee}$, we see $[g,f]([h,w])= r_h.([g,f]([1,w]))=r_h.(l_{g^{-1}}.(g'\mapsto f(g'w)))=(g'\mapsto f(g^{-1}g'hw)).$ 
    
    (2) By definition, for $f\in \cind_I^G\cW^{\vee}, w\in\cind_I^G\cW,f(w)\in\cC^{\la}_c(G,E)$, we have $(g.f)(w)=l_{g^{-1}}f(w),f(g.w)=r_gf(w)$. Hence the pairing below
    \[\cind_I^G\cW^{\vee}\times\cind_I^G\cW\rightarrow \cC^{\la}_c(G,E)\rightarrow E, (f,w)\mapsto \langle f,w\rangle:=(f(w))(1).\]
    is equivariant for the left actions on $\cind_I^G\cW^{\vee}$ and $\cind_I^G\cW$ and the trivial action on $E$. Using that $[g,f]([h,w])(1)$ is not zero only if $h\in gI$, we can check that the pairing $\langle-,-\rangle$ coincides with (\ref{equataionpairing}).

    On the other hand, the operator $\bbD^0_{\rm BZ}(U)$ is characterized by for $g\in G,f\in\cind_I^G\cW^{\vee},w\in\cind_{I}^G\cW$,
    \[(\bbD^0_{\rm BZ}(U).f)(w)(g)=f(U.w)(g),\]
    and by the $G$-equivariances of $U$ and $\bbD^0_{\rm BZ}(U)$, characterized by the equality
    \[\langle \bbD^0_{\rm BZ}(U).f,w\rangle=\langle f,U.w\rangle\]
    for all $f,w$. Hence $\bbD^0_{\rm BZ}(U)$ is the transpose of $U$.
\end{proof}

We return to Koszul complexes. Let $\frm\subset \cH$ act on $\cind_{I}^G\cW$ for $\cW=\cW_?(M)$ taken as in Proposition \ref{propositioncomplexesseq}. Apply the functor $\bbD_{\rm BZ}^0$ for the Koszul complex $\wedge^{\bullet}\cH^d\otimes_{\cH}\cind_{I}^G\cW$, the ``dual'' $\bbD^0_{\rm BZ}(\wedge^{\bullet}\cH^d\otimes_{\cH}\cind_{I}^G\cW)$ is then a complex of $G$-modules in cohomological degrees $[0,d]$.
\begin{lemma}\label{dualKoszulcomplex}
    The complex $\bbD^0_{\rm BZ}(\wedge^{\bullet}\cH^d\otimes_{\cH}\cind_{I}^G\cW)$ is equal to the Koszul complex 
    \[\wedge^{\bullet}\cH^d\otimes_{\cH}\cind_{I}^G\cW^{\vee}[-d]\] 
    where $\cH$ acts on $\cind_{I}^G\cW^{\vee}$ by the transpose of its action on $\cind_I^G\cW$ and $[-d]$ denotes the cohomological degree shift.
\end{lemma}
\begin{proof}
    Suppose that $\cH$ is a commutative ring, $A$ is a ring, $M$ is an $(A,\cH)$-bimodule with commutative actions, and let $N$ be an $A$-module. Let $\varphi:\cH^{m} \rightarrow \cH^{n}$ be an $\cH$-linear map inducing $\varphi_M:\cH^{m}\otimes_{\cH}M \rightarrow \cH^{n}\otimes_{\cH}M$. For a perfect $\cH$-module $P$, we have functorially $\Hom_{A}(P\otimes_{\cH}M,N)=\Hom_{\cH}(P,\Hom_{A}(M,N))$ and $R\Hom_{\cH}(P,\Hom_{A}(M,N))=P^{\vee}\otimes_{\cH}^L\Hom_{A}(M,N)$, where $P^{\vee}:=R\Hom_{\cH}(P,\cH)$ \cite[\href{https://stacks.math.columbia.edu/tag/08JJ}{Tag 08JJ}]{stacks-project}. Take $P=\cH^m,\cH^n$, we see $\Hom_A(-,N)(\varphi_M): \Hom_{A}(\cH^{n}\otimes_{\cH}M,N)\rightarrow \Hom_{A}(\cH^{m}\otimes_{\cH}M,N)$ can be identified with the tensoring over $\cH$ of the transpose $\varphi^{\vee}:(\cH^n)^{\vee}\rightarrow (\cH^m)^{\vee}$ with $\Hom_A(M,N)$.

    Apply the same reasoning for connecting maps of $\wedge^{\bullet}\cH^d\otimes_{\cH}\cind_{I}^G\cW$, we see the ``dual'' $\bbD^0_{\rm BZ}(\wedge^{\bullet}\cH^d\otimes_{\cH}\cind_{I}^G\cW)$ is calculated by $\Hom_{\cH}(\wedge^{\bullet}\cH^d,\cH)\otimes_{\cH}\cind_{I}^G\cW^{\vee}$ where $\cH$ acts on $\cind_{I}^G\cW^{\vee}$ by the transpose in Lemma \ref{lemmadualcompactinduction}. We get the desired identification using that $\Hom_{\cH}(\wedge^{\bullet}\cH^d,\cH)=\wedge^{\bullet}\cH^d[-d]$ the autoduality of the exterior algebras, cf. \cite[\S A.2, Ch. IV]{serre1997algebre}.
\end{proof}

We can now prove our main theorem on the duality. 
\begin{thm}\label{theoremdualKS}
    Let $M\in \cO^{\frb}_{\rm alg}$ and $\chi_{\sm}$ be a smooth character of $T$. Consider the resolution by $\cind_I^G\cW_{\sharp,r}(\tau(M)^{\vee}\otimes_E\chi_{\rm sm})$ of $\cF_{B}^G(M,\chi_{\sm})$ in Theorem \ref{theoremresolution}. Then there is a quasi-isomorphism 
    \[\bbD^0_{\rm BZ}(\wedge^{\bullet}\cH^d\otimes_{\cH}\cind_{I}^G\cW_{\sharp,r}(\tau(M)^{\vee}\otimes_{E}\chi_{\rm sm}))\simeq \cF_{\ovB}^G(\Hom_E(M,E)^{\overline{\frn}^{\infty}},\bbD_{\BZ}(\chi_{\sm}))\]
    of complexes of $E[G]$-modules.
\end{thm}
\begin{proof}
    By Lemma \ref{dualKoszulcomplex}, the dual complex is $\wedge^{\bullet}\cH^d\otimes_{\cH}\cind_{I}^G\cW_{\sharp,r}(\tau(M)^{\vee}\otimes\chi_{\rm sm})^{\vee}[-d]$ for the action of $U_{t^{-1}}$ operators. By Lemma \ref{lemmatransposepsit} and that $\cW_{\sharp,r}(\tau(M)^{\vee}\otimes\chi_{\rm sm})^{\vee}=\cD_r(I,E)\otimes_{\cD(\frg,I\cap B)}(M\otimes\chi_{\rm sm}^{-1})$, the transpose $U_{t^{-1}}$ comes from the $t$-actions or $\psi_{t^{-1}}$-actions on $M\otimes\chi_{\rm sm}^{-1}$ in the way of Lemma \ref{lemma-psit}.
    
    Let $I'=I\cap \ovI$ where $\ovI$ is the opposite Iwahori. We may assume that the uniform pro-$p$ subgroup $H$ in \S\ref{subsectionexamples} is a normal subgroup for both $I$ and $\ovI$. By Proposition \ref{propositionchangethegroups2}, 
    \begin{equation}\label{equationtheoremchangethegroup1}
        \wedge^{\bullet}\cH^d\otimes_{\cH} \cind_{I}^G\cW_{\sharp,r}(\tau(M)^{\vee}\otimes\chi_{\rm sm})^{\vee}\simeq \wedge^{\bullet}\cH^d\otimes_{\cH} \cind_{I'}^G\cW_{\sharp,r}(\tau(M)^{\vee}\otimes\chi_{\rm sm})^{\vee}|_{I'}.
    \end{equation} 
    While $\cD_r(I,E)\otimes_{\cD(\frg,B\cap I)}(M\otimes\chi_{\rm sm}^{-1})|_{I'}=\cD_r(I',E)\otimes_{\cD(\frg,I'\cap B)}(M\otimes\chi_{\rm sm}^{-1})$
    and 
    \[\cind_{I'}^{\ovI}\cD_r(I',E)\otimes_{\cD(\frg,I'\cap B)}(M\otimes \chi_{\rm sm}^{-1})=\cD_r(\ovI,E)\otimes_{\cD(\frg,\ovI\cap B)}(M\otimes \chi_{\rm sm}^{-1})=:\overline{\cW_{\natural,r}}(M\otimes\chi_{\rm sm}^{-1}).\]
    Hence by (2) of Proposition \ref{propositionchangethegroups1},  
    \begin{equation}\label{equationtheoremchangethegroup2}
        \wedge^{\bullet}\cH^d\otimes_{\cH} \cind_{I'}^G\cW_{\sharp,r}(\tau(M)^{\vee}\otimes\chi_{\rm sm})^{\vee}|_{I'}\simeq \wedge^{\bullet}\cH^d\otimes_{\cH} \cind_{\ovI}^G\overline{\cW_{\natural,r}}(M\otimes\chi_{\rm sm}^{-1}).
    \end{equation}
    Combining (\ref{equationtheoremchangethegroup1}), (\ref{equationtheoremchangethegroup2}) and Theorem \ref{theoremresolution}, we have
    \[\wedge^{\bullet}\cH^d\otimes_{\cH} \cind_{I}^G\cW_{\sharp,r}(\tau(M)^{\vee}\otimes\chi_{\rm sm})^{\vee}\simeq \wedge^{\bullet}\cH^d\otimes_{\cH} \cind_{\ovI}^G\overline{\cW}_{\natural,r}(M\otimes\chi_{\rm sm}^{-1})\simeq \cF_{\ovB}^G(\tau(M)^{\vee},\chi_{\rm sm}^{-1}).\]
    Checking the degree shift, we get the desired quasi-isomorphism.
\end{proof}

\section{Localization and completion}\label{sectionlocalizationandcompletion}
This section aims to prove Proposition \ref{propositioncompletionlocalization} (used for Proposition \ref{propositionsurjectivityPS}) using the Beilinson-Bernstein localization of Ardakov-Wadsley \cite{ardakov2013irreducible}.
\subsection{The localization functor}\label{subsectionlocalization}
We recall the settings and results in \cite{ardakov2013irreducible}. We change the notations from previous chapters. Let $\bG$ be a connected split reductive group over $\cO_L$. We assume that $\bG$ is semisimple and simply-connected. Fix a Borel subalgebra $\bB$ and the unipotent radical $\bN$ and the opposite groups $\overline{\bB},\overline{\bN}$. Write $\bX=\bG/\overline{\bB}$ for the flag variety over $\cO_L$. Let $G=\bG(L), B=\bB(L)$ as before. Let $\frg,\frb$, etc. be the Lie algebras over $\cO_L$.

There is the moment map \[\beta:\widetilde{T^*\bX}=\bG\times^{\overline{\bB}}(\frg/\overline{\frb})^*\rightarrow \frg^*,\] 
inducing a homomorphism $U(\frg)\rightarrow \widetilde{\cD}$ as filtered $\cO_L$-algebras. Here $\bH=\overline{\bB}/\overline{\bN}$, $\widetilde{\cD}= (\xi_{*}\cD_{\widetilde{\bX}})^{\bH}$ for $\xi:\widetilde{\bX}:=\bG/\overline{\bN}\rightarrow\bG/\overline{\bB}$ and $\cD_{\widetilde{\bX}}$ denotes the sheaf of (crystalline) differential operators on $\widetilde{\bX}$. The sheaf of algebras $\widetilde{\cD}$ (locally isomorphic to $\mathcal{\cD}_X\otimes U(\frh)$) is equipped with a natural increasing filtration $F_{m}\widetilde{\cD},m\geq 0$ (induced by order of differential operators) as in \cite[\S 4.6]{ardakov2013irreducible} such that $\mathrm{gr}(\widetilde{\cD})\simeq \Sym_{\cO_{\bX}}\widetilde{\cT}$ for the enhanced tangent sheaf $\widetilde{\cT}=(\xi_*\cT_{\widetilde{\bX}})^{\bH}$. The fiber of $\widetilde{\cD}$ at the point $\overline{\bB}\in \bX$ with respect to the left $\cO_{\bX}$-action is equal to $U(\frg)/\overline{\frn}U(\frg)$.

Identify $\bH=\overline{\bB}/\overline{\bN}=\bT$, the maps $U(\frg)\rightarrow \widetilde{\cD}$ and $U(\frt)\rightarrow \widetilde{\cD}$ coincide on the center $Z(\frg)=U(\frg)^{\bG}$ of $U(\frg)$ via the Harish-Chandra morphism $Z(\frg)\rightarrow U(\frt)^{W,\cdot}$ (\cite[\S 4.10]{ardakov2013irreducible}). And $U(\frg)$ is also equipped with the Poincaré-Birkhoff-Witt filtration so that $\mathrm{gr}(U(\frg))=\Sym_{\cO_L}\frg$. 

There is a set of collections $\mathcal{S}_{\bX}$ of affine Zariski opens of $\bX$ that trivialise $\xi$ as in \cite[\S 4.3]{ardakov2013irreducible}. 

The $n$-th deformation $\widetilde{\cD}_n$ of $\widetilde{\cD}$ is defined to be the subalgebra $ \sum_i\varpi_L^{in}F_i\widetilde{\cD}$ whose $\varpi_L$-adic completion is denoted by $\widehat{{\widetilde{\mathcal{D}}}_n}$ \cite[Def. 5.9]{ardakov2013irreducible} with $\widehat{{\widetilde{\cD}_{n,L}}}:=\widehat{{\widetilde{\mathcal{D}}}_n}\otimes_{\cO_L}L$ which are sheaves for the Zariski topology on $\bX$. Note that $\widehat{{\widetilde{\cD}_{n}}}:=\varprojlim_{a}\widetilde{\cD}_n/\varpi^a\widetilde{\cD}_n$ is supported on the special fiber of $\bX$. Similarly we define $U(\frg)_n=U(\varpi_L^n\frg)$ and $\widehat{U(\frg)_{n,L}}$. The limit $\widehat{U(\frg_L)}=\varprojlim_{n\in\bbN}\widehat{{U(\frg)_{n,L}}}$ is the Arens-Michael envelope of $U(\frg_L)$. 

Fix an $\cO_L$-linear map $\lambda:\varpi_L^n\frh\rightarrow \cO_L$ which induces a character $\lambda: U(\frh)_n:=U(\varpi^n\frh)\rightarrow \cO_L$. The sheaf of deformed twisted differential operators is defined by (\cite[Def. 6.4]{ardakov2013irreducible}) $\cD_n^{\lambda}:=\widetilde{\cD}_{n}\otimes_{U(\frh)_n}\lambda$ with the $\varpi_L$-adic completion $\widehat{\cD_n^{\lambda}}$ and $\widehat{\cD_{n,L}^{\lambda}}:=\widehat{\cD_{n}^{\lambda}}\otimes_{\cO_L}L$. Let $\widehat{\cU_{n}^{\lambda}},\widehat{\cU_{n,L}^{\lambda}}$ be corresponding completions for $\cU_{n}^{\lambda}:=U(\frg)_n\otimes_{(U(\frg)^{\bG})_n}\lambda$ as in \cite[\S 6.10]{ardakov2013irreducible}. And we write $\widehat{\cD_{n,L}},\widehat{\cU_{n,L}}$, etc. for $\widehat{\cD_{n,L}^{\lambda}}, \widehat{\cU_{n,L}^{\lambda}}$, etc. if $\lambda=0$. Then \cite[Thm. 6.10]{ardakov2013irreducible} says that the map $\widehat{\cU_{n,L}^{\lambda}}\rightarrow \widehat{\cD_{n,L}^{\lambda}}$ induces an isomorphism 
\[\widehat{\cU_{n,L}^{\lambda}}\simeq \Gamma(\bX,\widehat{\cD_{n,L}^{\lambda}}).\]
Let $\mathrm{Loc}^{\lambda}:\mathrm{Mod}(\widehat{\cU_{n,L}^{\lambda}})\rightarrow \mathrm{Mod}(\widehat{\cD_{n,L}^{\lambda}})$ be the localization functor 
\[M\mapsto \widehat{\cD_{n,L}^{\lambda}}\otimes_{\widehat{\cU_{n,L}^{\lambda}}}M.\]
Let $\rho$ be the half sum of all positive roots (of $\bB$). 
\begin{thm}[{\cite[Thm. 6.12]{ardakov2013irreducible}}]\label{theoremlocalization}
    Let $\lambda\in \Hom_{\cO_L}(\varpi_L^n\frh,\cO_L)$. 
    \begin{enumerate}
        \item If $\lambda+\rho$ is dominant ($\pair{\lambda+\rho}{\alpha^{\vee}}\notin \{-1,-2,\cdots\}$ for any positive (with respect to $\frb$) coroot $\alpha^{\vee}$ of $\bH$, \cite[\S 6.7]{ardakov2013irreducible}) and regular (i.e., the stabilizer of $\lambda+\rho$ under the action of the Weyl group is trivial), then the functor $\mathrm{Loc}^{\lambda}$ and the functor $\Gamma(\bX,-)$ of taking global sections induce an equivalence of the categories of coherent modules over $\widehat{\cU_{n,L}^{\lambda}}$ and $\widehat{\cD_{n,L}^{\lambda}}$. 
        \item If $\lambda+\rho$ is dominant but not regular, then $\mathrm{Loc}^{\lambda}$ and $\Gamma(\bX,-)$ still induce an equivalence between coherent $\widehat{\cU_{n,L}^{\lambda}}$-modules and the quotient of the category of coherent $\widehat{\cD_{n,L}^{\lambda}}$-modules by the full subcategory of modules that are in the kernel of $\Gamma(\bX,-)$.
    \end{enumerate}
\end{thm}
\begin{rem}
    In \cite{ardakov2013irreducible}, the prime number $p$ needs to be very good for $\bG$ \cite[\S 6.8]{ardakov2013irreducible}. This assumption is removed in \cite[\S 5.3]{ardakov2021equivariant}
\end{rem}
Let $M$ be a coherent $U(\frg_L)$-module with infinitesimal character given by $\lambda$. The localization $\cM:=\cD_{n,L}^{\lambda}\otimes_{\cU_{n,L}^{\lambda}}M$ is a coherent module of $\cD_{L}^{\lambda}=\cD_{n,L}^{\lambda}$ and $\Gamma(\bX,\cM)=M$ (this is essentially the classical localization). Let $\widehat{M_{n,L}}=\widehat{\cU_{n,L}^{\lambda}}\otimes_{\cU_{n,L}^{\lambda}}M$ and $\widehat{\cM_{n,L}}:=\mathrm{Loc}^{\lambda}(\widehat{M_{n,L}})$.
\begin{lemma}\label{lemmacompletionDmodule}
    For $\bV\in \mathcal{S}_{\bX}$,
    \[\widehat{\cM_{n,L}}(\bV)=\widehat{\cD_{n,L}^{\lambda}}(\bV)\otimes_{\cD_{n,L}^{\lambda}(\bV)}\cM(\bV).\]    
\end{lemma} 
\begin{proof}
   Since $\mathcal{S}_{\bX}$ is coherently acyclic (even affine) for the sheaves of rings $\widehat{\widetilde{\cD}_{n,L}}$ and $\widetilde{\cD}_{n,L}$ in the sense of \cite[Def. 5.1]{ardakov2013irreducible} by \cite[Cor. 5.9, Thm. 5.13]{ardakov2013irreducible}, so is for their quotients $\cD_{n,L}^{\lambda}, \widehat{\cD_{n,L}^{\lambda}}$ (\cite[Lem. 6.11]{ardakov2013irreducible}) by the same proof. For any $\bV\in \mathcal{S}_{\bX}$, we get $\cM(\bV)=\cD_{n,L}^{\lambda}(\bV)\otimes_{U_{n,L}^{\lambda}}M$ and $\widehat{\cM_{n,L}}(\bV)=\widehat{\cD_{n,L}^{\lambda}}(\bV)\otimes_{\widehat{\cU_{n,L}^{\lambda}}}\widehat{M_{n,L}}$ (writing a finite presentation of $M$ using free $\cU_{n,L}^{\lambda}$-modules). The result follows.
\end{proof}
Moreover $\widehat{\cD_{n,L}^{\lambda}}(\bV)=\widehat{\cD_{n}^{\lambda}(\bV)}\otimes_{\cO_L}L$ by \cite[Prop. 6.5]{ardakov2013irreducible}. Hence $\widehat{\cM_{n,L}}(\bV)$ is the completion of the coherent module $\cM(\bV)$ over $\cD_{n,L}^{\lambda}(\bV)$ with respect to the $p$-adic topology induced from the lattice $\cD_{n}^{\lambda}(\bV)\subset \cD_{n,L}^{\lambda}(\bV)$.
\subsection{Dual Verma modules for dominant weights} 
We consider the localization of dual Verma modules with dominant weights.

Suppose that $\lambda$ is integral (for the group $\bH$) and $\lambda+\rho$ is dominant with respect to $\frb$ (equivalently, $w_0(\lambda)-\rho=w_0(\lambda+\rho)$ is dominant with respect to $\overline{\frb}$, where $w_0$ is the longest element in the Weyl group). Take $\ovM=\overline{M}(w_0(\lambda))^{\vee}$ the BGG dual of the Verma module $\overline{M}(w_0(\lambda))=U(\frg_L)\otimes_{U(\overline{\frb})}w_0(\lambda)$ in the category $\cO^{\overline{\frb}_L}$ of the $(-\rho)$-dominant (for 
$\overline{\frb}$) highest weight $w_0(\lambda)$. Consider the Bruhat cells $\bX_{w}^{\circ}=\overline{\bB} w\overline{\bB}/\overline{\bB}\subset \bX$. The inclusion $j_{w_0}:\bX_{w_0}^{\circ}\hookrightarrow \bX$ is an open embedding.
\begin{lemma}\label{lemmalocalizationdualverma}
    There is an isomorphism of $U(\frg_L)$-modules,
    \[ \overline{M}(w_0(\lambda))^{\vee}\simeq \Gamma(\bX, ((j_{w_0,*}\cO_{\bX_{w_0}})\otimes_{\cO_{\bX}}\cO(\lambda))\otimes_{\cO_L}L).\]    
    Here $j_{w_0,*}$ is the usual direct image functor for $\cO_{\bX}$-modules and $\cO(\lambda)$ is the line bundle $\bG\times^{\overline{\bB}}\lambda$ over $\bX$ consisting of functions $f$ on $\bG$ such that $f(gb)=\lambda^{-1}(b)f(g)$ for $b\in \overline{\bB}$. And, 
    \[ \cM_{n,L}\simeq ((j_{w_0,*}\cO_{\bX_{w_0}})\otimes_{\cO_{\bX}}\cO(\lambda))\otimes_{\cO_L}L\]
    as $\cD_{n,L}^{\lambda}$-modules under the equivalences in Theorem \ref{theoremlocalization}.
\end{lemma}
\begin{proof}
    This is \cite[Prop. 4.4]{beilinson1999wall}. We give some explanation. The sheaf $(j_{w_0,*}\cO_{\bX_{w_0}})\otimes_{\cO_L}L$ is a $\cD_{L}=\cD_{n,L}$-module for the trivial weight $0$ and the translation $(j_{w_0,*}\cO_{\bX_{w_0}})_L\otimes_{\cO_{\bX}}\cO(\lambda)$ becomes a $\cD_{L}^{\lambda}$-module via the identification $\cD_L^{\lambda}\simeq \cO(\lambda)\otimes \cD_L\otimes \cO(\lambda)^{-1}$. The action of $\widetilde{\cD}=(\xi_{*}\cD_{\bG/\overline{\bN}})^{\bH}$ on $\cO(\lambda)\subset (\xi_{*}\cO_{\bG/\overline{\bN}})^{\bH=\lambda}$ factors through the quotient $\widetilde{\cD}^{\lambda}_0$ since $U(\frh)\subset \widetilde{\cD}$ in \cite[Prop. 4.6]{ardakov2013irreducible} acts via the character $\lambda$, see \cite[\S 6]{kashiwara1988representation}. If $\lambda=0,\ovM(w_0(0))^{\vee}=\overline{M}(ww_0\overline{\cdot} 0)^{\vee}$ for $w=w_0$ (here $w\overline{\cdot}\lambda=w(\lambda-\rho)+\rho$), the result is \cite[Cor. 5.8]{brylinski1981kazhdan}, or Lemma \ref{lemmadualvermaprincipalserieslocalization} below. In general, if $\lambda$ is dominant for $\frb$, $\Gamma(\bX,\cO(\lambda))$ has $B$-highest weight $\lambda$, $\overline{B}$-highest weight $w_0(\lambda)$, see \cite[Prop. II.2.2]{Jantzen2003representations}. The global section of the  geometric translation $\mathrm{Loc}_0(\ovM(0))\otimes \cO(\lambda)$ when $\lambda+\rho$ is $\frb$-dominant is the translation $T_{w_0\overline{\cdot}0}^{w_0\overline{\cdot}w_0(\lambda)}\overline{M}(0)^{\vee}=\overline{M}(w_0(\lambda))^{\vee}$ by \cite[Thm. 7.6]{humphreys2008representations} and the proof of \cite[Prop. 2.8]{beilinson1999wall}. 
\end{proof} 
Taking global sections, the $U(\frg_L)$-module $\overline{M}(w_0(\lambda))^{\vee}=\Gamma(\bX,\cM_{n,L})=\cO_{\bX}(\bX_{w_0})(\lambda)_L$ can be identified with $\cO(\overline{\bN}_L)$ after trivializing the line bundle $\cO(\lambda)_L$ over $\bX_{w_0}^{\circ}=\overline{\bN}w_0\overline{\bB}/\overline{\bB}$. Let $\Phi^-$ be the set of roots of $\ovN$. We fix a decomposition $\overline{\bN}\simeq \prod_{\alpha\in \Phi^-}\bG_{a}$ into root subgroups with coordinates $X_{\alpha},\alpha\in \Phi^-$ so that $\overline{\bN}\simeq \Spec(\cO_E[X_{\alpha},\alpha\in \Phi^-])$. 
\begin{lemma}\label{lemmasectionlinebundle}
    Fix a trivialization of $\cO(\lambda)$ on $\bX_{w_0}^{\circ}$. The section $\widehat{\cM_{n,L}}(\bX_{w_0}^{\circ})$ is equal to the Tate algebra $E\langle X_{\alpha},\alpha\in \Phi^-\rangle$.
\end{lemma}
\begin{proof}
    Without the completion, $\cM_{n,L}(\bX_{w_0}^{\circ})$ is equal to the space $\cC^{\pol}(\ovN ,L)=\cO(\bN_L)$. By Lemma \ref{lemmacompletionDmodule}, $\widehat{\cM_{n,L}}(\bX_{w_0}^{\circ})$ is the completion of $\cC^{\pol}(\ovN ,L)$ for the $p$-adic topology on $\cD_{n,L}^{\lambda}(\bX_{w_0})\simeq \cD_{n,L}(\overline{\bN})=L[X_{\alpha},\partial_{\alpha},\alpha\in \Phi^-]$ \cite[Lem. 6.4]{ardakov2013irreducible}, a Weyl algebra. By definition \cite[\S 3.5, \S 5.7]{ardakov2013irreducible} and by \cite[Prop. 5.7, \S 1.7]{ardakov2013irreducible}, $\widehat{\cD_{n,L}}(\overline{\bN})$ is the completion of $\cD_{n,L}(\overline{\bN})$ for the $p$-adic norm 
    \[|\sum_{i=(i_{\alpha}),j=(j_{\alpha}) \in\bbN^{\Phi^-}}\lambda_{i,j}\prod_{\alpha} X_{\alpha}^{i_{\alpha}}\prod_{\alpha} \partial_{\alpha}^{j_{\alpha}}|=\sup_{i,j}|\lambda_{i,j}||\varpi_L|_p^{-n\sum_{\alpha}j_{\alpha}}.\] 
    Since $\cM_{n,L}(\overline{\bN})\simeq \cD_{n,L}(\overline{\bN})/(\partial_{\alpha},\alpha\in \Phi^-)$ as left $\cD_{n,L}(\overline{\bN})$-modules (up to a twist), we conclude that 
    \[\widehat{\cM_{n,L}}(\overline{\bN})\simeq \widehat{\cD_{n,L}}(\overline{\bN})/(\partial_{\alpha},\alpha\in \Phi^-)=E\langle X_{\alpha},\alpha\in \Phi^-\rangle\]
    is the Tate algebra
\end{proof}

On the other hand, let $\chi_{\lambda}:T\rightarrow L^{\times}$ be an $L$-analytic character with the weight $\lambda$. Consider the locally analytic principal series
\[\Ind_{\overline{B}}^G\chi_{\lambda}=\{f:G\rightarrow L \text{ locally analytic }\mid f(gb)=\chi_{\lambda}^{-1}(b)f(g),\forall b\in \ovB\}\]
with the usual left action of $G$ given by $g.f(-)=f(g^{-1}-)$. Fix an isomorphism $N\simeq N\overline{B}/\ovB$. Consider the subspace $\cC^{\pol}(N,L)=\cO(\bN)_L$ of polynomial functions on $N$, which is equal to the dual Verma module $M(\lambda)^{\vee}=(U(\frg_L)\otimes_{U(\frb_L)}\lambda)^{\vee}$ in $\cO^{\frb_L}$ as in Remark \ref{remark-dualverma}. Via the isomorphism 
\[\cC^{\pol}(N\ovB/\ovB,L)\simeq \cC^{\pol}(\ovN w_0\ovB/\ovB,L):f(-)\mapsto f(w_0-),\]
we see again that the $U(\frg_L)$-module $\cC^{\pol}(\ovN w_0\ovB/\ovB,L)$ is isomorphic to $\overline{M}(w_0(\lambda))^{\vee}$ the twist by $w_0$ of $M(\lambda)^{\vee}$.
\begin{lemma}\label{lemmadualvermaprincipalserieslocalization}
    The two $U(\frg_L)$-module structures on $\cC^{\pol}(\ovN ,L)\simeq \overline{M}(w_0(\lambda))^{\vee}$ (via $w_0\cC^{\pol}(\ovN ,L)\subset \Ind_{\overline{B}}^G\chi_{\lambda}$ and $\cC^{\pol}(\ovN ,L)\simeq \Gamma(\bX,\cM_{n,L})$) coincide.
\end{lemma}
\begin{proof}
    The $U(\frg)$-module structures on $\cO_{\bX_{w_0}}(\lambda)$ both come from the infinitesimal action of $G=\bG(L)$ on $\cO(\overline{\bN} w_0\overline{\bB})^{\overline{\bB}=\lambda}$ given by $(x.f)(g)=\frac{d}{dt}f(\exp(-tx)g)$ for $g\in \bG,x\in \frg$ and $f$ regular near a neighborhood of $g$. 
\end{proof}

\subsection{Completion near the boundary}\label{sectioncompletion}
We will prove Proposition \ref{propositioncompletionlocalization} by studying the completion $\widehat{\cM_{n,L}}$ of the coadmissible $\cD_{n,L}^{\lambda}$-module $\cM_{n,L}$ in Lemma \ref{lemmalocalizationdualverma}.

    We take a finite covering 
    \[S=\{\bV_w:=w\overline{\bN}w_0\overline{B}/\overline{B},w\in W\}\] 
    of the flag variety $\bX$ where $W$ denotes the Weyl group of $\bG$. Let $\partial\bX=\bX-\bX_{w_0}^{\circ}$ be the boundary (Weil or Cartier) divisor which is also the union $\cup_{s}\bX_{w_0s}$ of codimension one Schubert varieties where $s$ runs over all simple reflections \cite[Prop. 2.3.2]{bjorner2005combinatorics}. For all $\bV\in S$, let $f_{\bV}\in \cO(\bV_L)$ be an element cutting out the boundary $\bV_L\cap \partial \bX_L$ (note that the Picard group of $\bV_L$ is trivial \cite[\href{https://stacks.math.columbia.edu/tag/0BDA}{Tag 0BDA}]{stacks-project}).

    Take one $\bV\in S$. Let $A=\cO(\bV_L)=L[Y_1,\cdots,Y_{l}]$ where $l=|\Phi^-|$, $\widehat{A}=L\langle Y_1,\cdots, Y_d \rangle$, $\cD:=\cD_{n,L}(\bV)=A[\varpi_L^n\partial_1,\cdots,\varpi_L^n\partial_l]$ and $\widehat{\cD}_{n,L}=\widehat{\cD_{n,L}}(\bV)$. Write $|\cdot|_{\widehat{A}}$ for the norm 
    \[|\sum_{i=(i_1,\cdots,i_l)\in\bbN^l}\lambda_i\prod_{j=1}^lY_j^{i_j}|_{\widehat{A}}=\sup_i|\lambda_i|_p\]
    and $f=f_{\bV}\in A$.

    The reason that $A[\frac{1}{f}]$ is a finitely generated $\cD$-module lies in the theory of $b$-functions (\cite[Cor. 3.1.2]{Mebkhout1991Bernstein}).
    \begin{thm}[{\cite[Thm. 3.1.1]{Mebkhout1991Bernstein}}]\label{theorembfunction}
        For any $u\in A$, there exist non-zero polynomials $b(s)\in L[s]$ and $P(s)\in \cD[s]$ such that 
        \[ b(s)f^{-s-1}u=P(s)f^{-s}u\]
        after specializing $s\in\Z$. 
    \end{thm} 
    The behavior of the $\cD$-module $A[\frac{1}{f}]$ under completion is studied in \cite{BitounBode2021meromorphic}. We adapt the arguments in \cite{BitounBode2021meromorphic} to our case. We recall the following definition.
    \begin{dfn}[{\cite[Def. 3.1]{BitounBode2021meromorphic}}]\label{definitionpositivetype}
        Let $\lambda\in \overline{\Q}_p$. We say that $\lambda$ is of positive type if $\lambda\in\Z_{\geq 0}$ or there exists $M>0$ such that $\varinjlim_{i\rightarrow \infty}\frac{p^{iM}}{\prod_{j=0}^{i-1}(\lambda-j)}=0$, equivalently (\cite[Lem. B.2]{BitounBode2021meromorphic}) the type of $\lambda$ which is, by \cite[Def. 13.1.1]{kedlaya2022p}, the radius of convergence of the formal power series $\sum_{i\geq 0,i\neq \lambda}\frac{x^i}{\lambda-i}$, is positive.
    \end{dfn}
    \begin{rem}
        Recall that a number $\lambda\in\overline{\Q}_p$ is $p$-adic non-Liouville if $\pm\lambda$ have type $1$ \cite[Def. 13.1.2]{kedlaya2022p}.
    \end{rem}
    \begin{lemma}\label{lemmaconvergesbfunction}
        Let $u$ be an element in $A$ such that the coefficients of the polynomial $b(s)$ in Theorem \ref{theorembfunction} are in $\overline{\Q}[s]$. Then there exists $M>0$ such that the series $\sum_{i\geq 0} a_ip^{Mi}\frac{1}{f^i}u$ for $a_i\in \widehat{A}$ converges in $\widehat{\cD}_{n,L}\otimes_{\cD}(A[\frac{1}{f}])$ provided that $\lim_{i \rightarrow \infty} |a_i|_{\widehat{A}}=0$. 
    \end{lemma}
    \begin{proof}
         Let $s_0\geq 0$ such that $b(s)$ has no zero in $\bbZ_{\geq s_0}$, we get 
        \[f^{-s}u=\frac{P(s-1)\cdots P(s_0)}{b(s-1)\cdots b(s_0)}f^{-s_0}u\] for $s>s_0$.  
        
        By the assumption, all the roots of $b(s)$ are of positive type by \cite[Prop. VI.1.1]{dwork2016introduction}.  By \cite[Lem. 4.1]{BitounBode2021meromorphic}, there exists an integer $M$ such that for all $i\geq 0$, $p^{Mi}\frac{P(i-1)\cdots P(s_0)}{\prod_{s=s_0}^{i-1}b(s)}$ lies in the unit ball of $\widehat{\cD}_{n,L}$: the $p$-adic completion of $\cO_L[Y_1,\cdots,Y_l,\varpi_L^n\partial_1,\cdots,\varpi_L^n\partial_l]$. We take one such $M$.

        Let $g=\sum_{i\geq 0} a_ip^{Mi}\frac{1}{f^i}u\in \widehat{A}\langle \frac{p^M}{f}\rangle$ for $a_i\in \widehat{A}$ with $\lim_{i} |a_i|_{\widehat{A}}=0$. Then $g$ can also be written as $\sum_{i\geq s_0}a_ip^{Mi}\frac{P(i-1)\cdots P(s_0)}{\prod_{s=s_0}^{i-1}b(s)}f^{-s_0}u$ up to finitely many terms. Since $\varinjlim_{i\rightarrow\infty}a_i=0$ and $p^{Mi}\frac{P(i-1)\cdots P(s_0)}{\prod_{s=s_0}^{i-1}b(s)}$ are bounded, the element $\sum_{i\geq s_0}a_i p^{Mi}\frac{P(i-1)\cdots P(s_0)}{\prod_{s=s_0}^{i-1}b(s)}$ converges in $\widehat{\cD}_{n,L}$. Hence the series $g=\sum_{i\geq 0} a_ip^{Mi}\frac{1}{f^i}u$ converges in $\widehat{\cD}_{n,L}\otimes_{\cD}A[\frac{1}{f}]$.
    \end{proof}
    \begin{prop}\label{propositioncompletionlocalization}
        Take $M=\overline{M}(w_0(\lambda))^{\vee}$ for $\lambda+\rho$ dominant with respect to $\frb$ and integral for $\bH$. Let $\widehat{\cM_{n,L}}=\mathrm{Loc}^{\lambda}(\widehat{M_{n,L}})$. Then the restriction map $\widehat{\cM_{n,L}}(\bX)=\widehat{\cU_{n,L}^{\lambda}}\otimes_{\cU_{n,L}^{\lambda}}\overline{M}(w_0(\lambda))^{\vee}\rightarrow \widehat{\cM_{n,L}}(\bX_{w_0}^{\circ})$ contains (in the notation of Lemma \ref{lemmasectionlinebundle})
        \[\cO(\overline{\bN}^{\rm an}(\leq p^{M})):=\{\sum_{i=(i_{\alpha})_{\alpha}\in \bbN^{\Phi^-}}\lambda_ip^{M(\sum_{\alpha}i_{\alpha})}\prod_{\alpha}X_{\alpha}^{i_{\alpha}}\mid \lambda_i\in L, \varinjlim_{|i|\rightarrow\infty} |\lambda_i|_p=0 \}\] 
        for some $M\geq 0$.
    \end{prop}
    \begin{proof}
        By Lemma \ref{lemmalocalizationdualverma}, the localization of $\overline{M}(w_0(\lambda))^{\vee}$ is the pushforward of $\cO_{\bX}(\lambda)|_{X_{w_0}^{\circ}}$. Hence for $\bV\in S$,  $\cM_{n,L}(\bV)=\cO_{\bX}(\lambda)(\bV-\partial\bX)\otimes_{\cO_L} L=\cO_{\bX_L}(\bV_L)[\frac{1}{f_{\bV}}](\lambda)$. Here $\cO_{\bX_L}(\bV_L)[\frac{1}{f_{\bV}}](\lambda)$ denotes the twist by $\cO_{\bX_L}(\lambda)$ for the action of $\cD_{n,L}^{\lambda}$. Denote also by $X_{\alpha}\in \cM_{n,L}(\bX_{w_0}^{\circ})$ the image of the coordinates of $\overline{\bN}\simeq \overline{\bN}w_0\overline{\bB}/\overline{\bB}$ in $\cO_{\bX}(\bV_L)[\frac{1}{f_{\bV}}]=\cO_{\bX}(\bX_{w_0}^{\circ}\cap \bV)[\frac{1}{p}]$. 
            
        We fix a trivialization of $\iota_{\bX_{w_0}^{\circ}}:\cM_{n,L}|_{\bX_{w_0}^{\circ}}=\cO_{\bX}(\lambda)|_{\bX_{w_0}^{\circ}}[\frac{1}{p}]\simeq \cO_{\bX}|_{\bX_{w_0}^{\circ}}[\frac{1}{p}]$. Since $\widehat{\cM_{n,L}}$ is a sheaf, to prove the proposition, we only need take $M$ such that the sequences 
        \[\sum_{i}\lambda_ip^{M|i|}X^i:=\sum_{i\in\bbN^{\Phi^-}}\lambda_ip^{M(\sum_{\alpha}i_{\alpha})}\prod_{\alpha}X_{\alpha}^{i_{\alpha}}\] 
        converge in the completions 
        $\widehat{\cM_{n,L}}(\bV)=\widehat{\cD_{n,L}^{\lambda}}(\bV)\otimes_{\cD^{\lambda}_{n,L}(\bV)}(\cO_{\bX_L}(\bV_L)[\frac{1}{f_{\bV}}](\lambda))$ (Lemma \ref{lemmacompletionDmodule}) for all $\bV\in S$ provided that $\varinjlim_i\lambda_i=0$. 
        
        There are only finitely many $\bV\in S$, thus we can focus on one $\bV$. On $\bV$ the line bundle $\cO(\lambda)$ can be trivialized and there are isomorphisms $\widehat{\cD_{n,L}^{\lambda}}(\bV)\simeq \widehat{\cD_{n,L}}(\bV)$ \cite[Prop. 6.5]{ardakov2013irreducible}. Under this isomorphism, we can identify $\cM_{n,L}(\bV)$ with the $\cD_{n,L}(\bV)$-module $\cO_{\bX_L}(\bV_L)[\frac{1}{f_{\bV}}]$ via $\iota_{\bV}:\cO_{\bX}(\lambda)|_{\bV}[\frac{1}{p}]\simeq \cO_{\bX}|_{\bV}[\frac{1}{p}]$, i.e., we can assume $\lambda=0$ if we focus on the completion over $\bV$. Write $f=f_{\bV}$.

        The trivializations $\iota_{\bX_{w_0}^{\circ}}$ and $\iota_{\bV}$ of $\cO_{\bX}(\lambda)|_{\bX_{w_0}^{\circ}}$ and $\cO_{\bX}(\lambda)|_{\bV}$ may not be compatible over $\bV\cap \bX_{w_0}^{\circ}$. The composite $\iota_{\bV}\circ\iota_{\bX_{w_0}^{\circ}}^{-1}|_{\bV\cap \bX_{w_0}^{\circ}}:\cO_{\bX}(\bV-\partial\bX)[\frac{1}{p}]\simeq \cO_{\bX}(\bV-\partial\bX)[\frac{1}{p}]=\cO_{\bX}(\bV_L)[\frac{1}{f}]$ is $\cO_{\bX}(\bV-\partial\bX)$-linear and is determined by an element $\iota_{\bV}\circ\iota_{\bX_{w_0}^{\circ}}^{-1}|_{\bV\cap\bX_{w_0}^{\circ}}(1)=\frac{u}{f^{n_u}}$ in $\cO_{\bX}(\bX_{w_0}^{\circ}\cap V)^{\times}$ for some $u\in \cO_{\bX}(\bV), n_u\in\Z$. Nevertheless, we can always write in $\cM_{n,L}(\bV_L)[\frac{1}{f}]$ that $X_{\alpha}=\frac{g_{\alpha}}{f^{n_{\alpha}}}\frac{u}{f^{n_u}}$  for $\alpha\in \Phi^-$ and for some $g_{\alpha}\in A:=\cO_{\bX_L}(\bV_L)$ and $n_{\alpha}\geq 0$ under the trivialization $\iota_{\bV}$. Then $X_{\alpha}^i=(\frac{g_{\alpha}}{f^{n_{\alpha}}})^i\frac{u}{f^{n_u}}$ in $A[\frac{1}{f}]$. 
        
        The split reductive group $\bG$ with $\bB\subset \bG$, as well as the inclusion $\bV\subset \bG/\overline{\bB}$ and the line bundle $\cO_{\bX}(\lambda)$ for $\lambda$ integral, can be defined over $\Q$. Hence we can take $u$ such that $b(s)\in \bbQ[s]$ satisfying Theorem \ref{theorembfunction} for our $A,u,f$. 

        Take $M_0$ such that $|g_{\alpha}|_{\widehat{A}}<p^{M_0}$ for all $\alpha$. Let $M_1\geq 0$ such that Lemma \ref{lemmaconvergesbfunction} holds for $M_1,f=f_{\bV}$ and $u$. Then the series 
        \[\sum_{i\in\bbN^{\Phi^-}}\lambda_i p^{M|i|}X^i=\sum_{i\in\bbN^{\Phi^-}}\lambda_ip^{M|i|}\frac{\prod_{\alpha} g_{\alpha}^{i_{\alpha}}}{f^{n_u+\sum_{\alpha} i_{\alpha}n_{\alpha}}}u=\frac{1}{f^{n_u}}\sum_{j\geq 0}(\sum_{\sum_{\alpha} i_{\alpha}n_{\alpha}=j}\lambda_ip^{M|i|-M_1j}\prod_{\alpha}g_{\alpha}^{i_{\alpha}})\frac{p^{M_1j}}{f^j}u\]
        converges in $\widehat{\cD_{n,L}}(\bV)\otimes_{\cD_{n,L}(\bV)}A[\frac{1}{f}]$ if $\varinjlim_{i}\lambda_i=0$ and $M>M_0+M_1(\sum_{\alpha}(1+n_{\alpha}))$ by Lemma \ref{lemmaconvergesbfunction}.
    \end{proof}

\section{Duality of patching modules}\label{sectiondualitysolid}
Using solid formalism of locally analytic representations in \cite{rodrigues2022solid,jacinto2023solid}, we will define the Bernstein-Zelevinsky duality functor $\bbD_{\rm BZ}(-)$ (Definition \ref{definitionDBZ}) and then discuss coherent sheaves arising from the patching functors (Definition \ref{definitionpatchingfunctor}).

Our notations follow largely \cite{jacinto2023solid}. We consider the $p$-adic local field $E$ as a solid condensed ring. A solid $E$-space is roughly the condensed version of a complete locally convex topological $E$-space. Write $\mathrm{Mod}_{E_{\blacksquare}}$ for the ($\infty$-)derived category of solid $E$-modules, associated to the analytic ring $E_{\blacksquare}:=(E,\cO_E)_{\blacksquare}$, and write $\mathrm{Mod}_{E_{\blacksquare}}^{\heartsuit}$ for its heart with respect to the natural $t$-structure.

If $D$ is an associative solid $E$-algebra (we will only consider non-derived/static condensed rings), we write $\mathrm{Mod}_{E_{\blacksquare}}(D)$ for the stable $\infty$-category of left $D$-modules on $E_{\blacksquare}$-spaces. Let $G$ be a split reductive $p$-adic Lie group as before. We will take for example $D=E_{\blacksquare}[G]$ the solid Iwasawa algebra or $\cD(G,E)$ the locally analytic distribution algebra \cite[\S 2.1]{jacinto2023solid}. Note that if $D=E_{\blacksquare}[G]$ or $\cD(G,E)$, there is a canonical involution of $D$ induced by the inverse map of $G$ which induces an equivalence of the categories of left and right $D$-modules. We denote the involution by $\iota$. We write $\iota(-)\otimes_{D,\blacksquare}^L-$, or just $-\otimes_{D,\blacksquare}^L-$, for the relative tensor product (see \cite[\S4.5]{lurie2007derivedII}) of $\mathrm{Mod}_{E_{\blacksquare}}(D)$, an analog of complete tensor products. There are condensed non-derived or derived Hom's for $D$-modules $\underline{\Hom}_D(-,-),R\underline{\Hom}_D(-,-)\in\mathrm{Mod}_{E_{\blacksquare}}$. If $M,N\in \mathrm{Mod}_{E_{\blacksquare}}(D)$ and $S$ is a profinite set, then $R\underline{\Hom}_D(M,N)(S)=R\Hom_{E}(E[S],R\underline{\Hom}_D(M,N))=R\Hom_D(E[S]\otimes^L_{E}M,N)$.

We will focus on the category  $\mathrm{Mod}_{E_{\blacksquare}}(\cD(G,E))$ of solid $\cD(G,E)$-modules. There is a full subcategory $\mathrm{Rep}_{E_{\blacksquare}}^{\rm la}(G)$ of (derived) solid locally analytic representations of $G$, which is the derived category of its heart $\mathrm{Rep}_{E_{\blacksquare}}^{\rm la,\heartsuit}(G)$ \cite[Prop. 3.2.6]{jacinto2023solid}.

\subsection{Cohomological duality}
Let $I\subset G$ be the Iwahori subgroup as in \S\ref{sectiondualcomplexes}. Suppose that $\cW\in \Mod_{E_{\blacksquare}}(\cD(I,E))$, we define the compact induction 
\[\cind_{I}^G\cW:=E_{\blacksquare}[G]\otimes_{E_{\blacksquare}[I]}^L\cW=\cD(G,E)\otimes_{\cD(I,E)}^L\cW.\]
If $\cW\in \Mod_{E_{\blacksquare}}^{\heartsuit}(\cD(I,E))$, then $\cind_{I}^G\cW=\cD(G,E)\otimes_{\cD(I,E)}\cW$ is in $\mathrm{Mod}_{E_{\blacksquare}}^{\heartsuit}(\cD(G,E))$ (since $\cD(G,E)$ is free over $\cD(I,E)$ and $\mathrm{Mod}_{E_{\blacksquare}}^{\heartsuit}$ is stable under direct sums which are exact). 
\begin{rem}
    If $\cW=\underline{\cW(*)}$ (in the notation of \cite[Prop. 1.7]{scholze2019condensed}) arises as a classical LS space of compact type, then $(\cind_{I}^G\cW)(*)=\Hom_{E}(E,\cind_{I}^G\cW)=\cind_{I}^G\cW(*)$ since $E$ is compact in $\mathrm{Mod}_{E_{\blacksquare}}$. One may equip $(\cind_{I}^G\cW)(*)$ with certain topology as in \textit{loc. cit}. However, this topology may not coincide with the locally convex direct sum topology on $\cind_{I}^G\cW(*)$. If $V=\varinjlim_iV_i$ is a locally convex inductive limit of Banach spaces with compact injective transition maps, then $\underline{V}=\varinjlim_i\underline{V_i}$ \cite[Lem. 2.19]{colmez2023arithmetic}, which is not true in general.
\end{rem}
Suppose that $\cW=\underline{\cW(*)}$ for $\cW(*)=\cW_{?}(\tau(M)^{\vee}\otimes\chi_{\rm sm}) $ with $\cind_{I}^G\cW(*)\rightarrow \pi(*)=\cF_{B}^G(M,\chi_{\rm sm})$ as in Theorem \ref{theoremresolution} (notation in \S\ref{subsectionbeyonddualverma}) where $\pi=\underline{\cF_{B}^G(M,\chi_{\rm sm})}$. Then $\cW(*)$ is a classical $\cD(I,E)$-module whose underlying topological space is a Fr\'echet space or a Smith space with elements in $\cH$ acting continuously on $\cind_{I}^G\cW(*)$. If $\cW$ is a Banach space, $\cind_{I}^G\cW(*)$ with the locally convex inductive limit topology is an LF space. The augmented complex
\[ \wedge^{\bullet}\cH^d\otimes_{\cH}\cind_{I}^G\cW(*)\rightarrow \pi(*)\]
is an exact complex of LF-spaces where all morphisms are strict by the open mapping theorem \cite[Thm. 1.1.17]{emerton2017locally}.

\begin{prop}\label{propositionsolidresolution}
    Let $\cW,\pi$ be as above, then the map $\cW\rightarrow\pi$ induces an acyclic augmented complex
    \[\wedge^{\bullet}\cH^d\otimes_{\cH}\cind_{I}^G\cW\rightarrow \pi\] 
    in $\mathrm{Mod}_{E_{\blacksquare}}(\cD(G,E))$.
\end{prop}
\begin{proof}
    The map $\cW(*)\rightarrow \pi(*)$ induces $\cW\rightarrow \pi$ in $\mathrm{Mod}_{E_{\blacksquare}}(\cD(G,E))$ and $\psi_t:\cW(*)\rightarrow \cW(*)$ induces $\psi_t:\cW\rightarrow \cW$ via the fully faithful embedding of \cite[Prop. 1.7]{scholze2019condensed} (note that Fr\'echet and Smith spaces are compactly generated). The general construction in \S\ref{subsectionKoszulcomplex} applies formally to give the desired complex. We remain to show the complex is exact. 
    
    We show that $\wedge^{\bullet}\cH^d\otimes_{\cH}\cind_{I}^G\cW\simeq \cind_{I}^G\cW\otimes_{\cH}\cH/\frm$. We check that for any extremally disconnected set $S$, the complex $\wedge^{\bullet}\cH^d\otimes_{\cH}(\cind_{I}^G\cW)(S)$ concentrates in degree $0$. While $(\cind_{I}^G\cW)(S)=\Hom_E(E_{\blacksquare}[S],\cind_{I}^G\cW)=\cind_{I}^G\Hom_E(E_{\blacksquare}[S],\cW)$ since $E_{\blacksquare}[S]$ is compact. Hence the complex of $S$-points concentrates in degree $0$ by Proposition \ref{propositionKSresolutionKoszulregular} replacing $\cW$ by $ \Hom_E(E_{\blacksquare}[S],\cW)$.

    We assume first that $\cW$ is a (condensed) Banach space so that $\cind_{I}^G\cW$ is an LF space.

    We prove the surjectivity of $\cind_I^G\cW\rightarrow \pi$. We write $K$ for the kernel of $\cind_I^G\cW\rightarrow \pi$, and write $\cind_I^G\cW=\varinjlim_n F_n$ where $F_n$ are Fr\'echet spaces and $K_n=F_n\cap K$, the kernel of $F_n\rightarrow \pi$. Since $\pi$ is LF ($\pi$ is LS of compact type \cite[Cor. 3.38]{rodrigues2022solid}), the map $F_n\rightarrow \pi$ factors through some Fr\'echet subspace of $\pi$ \cite[Lem. 3.32]{rodrigues2022solid}. By the fully faithfulness of \cite[Prop. 3.7]{rodrigues2022solid}, the map $F_n\rightarrow \pi$ in $\Hom_E(F_n,\pi)$ arises from some $F_n(*)\rightarrow \pi(*)$. The kernel is a Fr\'echet space denoted by $K_n(*)$ closed in $F_n(*)$. We let $\pi_n=F_n/K_n$. Then $(F_n/K_n)(*)=F_n(*)/K_n(*)$ as Fr\'echet spaces by \cite[Lem. A.33]{bosco2021p}. By the surjectivity of $\cind_I^G\cW(*)\rightarrow \pi(*)$, we know that $\pi(*)=\varinjlim_n\pi_n(*)$ as sets. Equip $\varinjlim_n\pi_n(*)$ with the locally convex inductive limit topology, then the map $\varinjlim_n\pi_n(*)\rightarrow \pi(*)$ is a topological isomorphism if $\pi(*)$ is equipped with the classical locally convex topology as an LF space by the open mapping theorem. Since filtered colimits are exact \cite[Thm. 1.10]{scholze2019condensed}, it remains to show that $\pi=\varinjlim_n\pi_n$. Write $\pi=\varinjlim_nB_n$ as a colimit of Fr\'echet spaces with injective transition maps. Then $\pi(*)=\varinjlim_nB_n(*)$ if we equip $\pi(*)$ with the locally convex inductive limit topology. As $\Hom_E(\pi(*),\pi(*))=\varprojlim_n\varinjlim_m\Hom_E(\pi_n(*),B_m(*))$ by \cite[Cor. 8.9]{schneider2013nonarchimedean}, we see the two direct systems $(\pi_n(*))_n, (B_n(*))_n$ are equivalent. Then the same holds for $(\pi_n)_n, (B_m)_m$ by the fully faithful embedding of Fr\'echet spaces to solid $E$-spaces. Hence $\pi=\varinjlim_nB_n=\varinjlim_n\pi_n$.

    We need also show that the map $\cH^d\otimes_{\cH}\cind_I^G\cW\rightarrow \cind_I^G\cW$ maps surjectively on the kernel $K=\varinjlim_nK_n$ of $\cind_I^G\cW\rightarrow \pi$. This can be proved in the same way as in the preceding paragraph using the surjectivity of $ \cH^d\otimes_{\cH}\cind_I^G\cW(*)\rightarrow K(*)$.

    Finally, we suppose that $\cW$ is a Smith space, for example, $\cW$ has the form $\cW_{\sharp,r}(\ovM)$. Then there exist Banach spaces $\cW,\cW'$ with injections $\cW''\subset \cW\subset \cW'$ such that the result holds for $\cW',\cW''$. Using Proposition \ref{propositionKSresolutionKoszulregular} and Corollary \ref{corinjectiontwoW}, $\wedge^{\bullet}\cH^d\otimes_{\cH}\cind_{I}^G\cW$ concentrates in degree $0$ and there is an injection $\cH/\frm\otimes_{\cH}\cind_{I}^G\cW\hookrightarrow \cH/\frm\otimes_{\cH}\cind_{I}^G\cW'$. Since the surjection $\cind_{I}^G\cW''\twoheadrightarrow \pi$  factors through the injection $\cH/\frm\otimes_{\cH}\cind_{I}^G\cW\hookrightarrow \pi\simeq \cH/\frm\otimes_{\cH}\cind_{I}^G\cW'$, the map $\cind_{I}^G\cW\rightarrow \pi$ is also surjective.
\end{proof}

We consider the $\cD(G,E)$-bimodule $ \cC^{\la}_c(G,E)=\cC^{\la}(I,E)\otimes_{\cD(I,E)}\cD(G,E)$ as in \S\ref{subsectiondualitycomplex}.
\begin{thm}\label{theoremsoliddualcompactinduction}
    Suppose that $\cW\in \mathrm{Mod}_{E_{\blacksquare}}^{\heartsuit}(\cD(I,E))$ is a $\cD(I,E)$-module over a Smith $E$-space. Then there is an isomorphism of solid $\cD(G,E)$-modules
    \[ \RiHom_{{E_{\blacksquare}[G]}}(\cind_{I}^G\cW, \cC^{\la}_c(G,E))\simeq \cind_{I}^G\cW^{\vee}\]
    where $\cW^{\vee}=\underline{\Hom}_{E}(\cW,E)$ with the usual dual action of $I$.
\end{thm}
\begin{proof}
    Since $\cC^{\la}_c(G,E)$ is nuclear being LB (\cite[Thm. A.43]{bosco2021p}) and $\cW$ is Smith, we have $\RiHom_{E}(\cW,\cC_c^{\la}(G,E))=\underline{\Hom}_{E}(\cW,\cC_c^{\la}(G,E))=\cW^{\vee}\otimes_{_{E,\blacksquare}}^L\cC_c^{\la}(G,E)$ (see \cite[Lem. 3.8 \& Cor. 3.17]{rodrigues2022solid}). We have (use \cite[Lem. 3.1.7]{jacinto2023solid})
    \[\begin{array}{rl}
        \RiHom_{{E_{\blacksquare}[G]}}(\cind_{I}^G\cW, \cC^{\la}_c(G,E))&=\RiHom_{{E_{\blacksquare}[I]}}(\cW, \cC^{\la}_c(G,E))\\
        &=\RiHom_{{E_{\blacksquare}[I]}}(E,\RiHom_{E}(\cW, \cC^{\la}_c(G,E)))\\
        &=\RiHom_{{E_{\blacksquare}[I]}}(E,\cW^{\vee}\otimes^L_{E,\blacksquare} \cC^{\la}_c(G,E)))\\
        &=(\iota(\cW^{\vee})\otimes \chi)\otimes_{E_{\blacksquare}[I]}^L\cC^{\la}_c(G,E)[-\dim G]
    \end{array}\] 
    by \cite[Prop. 3.1.12]{jacinto2023solid} where $\chi=\det(\frg)^{-1}$ is a right $G$-module as in \textit{loc. cit}. Then one sees that
    \[ \RiHom_{{E_{\blacksquare}[G]}}(\cind_{I}^G\cW, \cC^{\la}_c(G,E))\simeq((\iota(\cW^{\vee})\otimes \chi)\otimes_{E_{\blacksquare}[I]}^L\cC^{\la}(I,E))\otimes_{E_{\blacksquare}[I]}^L E_{\blacksquare}[G][-\dim G].\]
    By \cite[Prop. 4.41]{rodrigues2022solid}, $\cW$ is a $\cD^{h}(I,E)$-module (in the notation of \cite[\S 3.1]{jacinto2023solid}) for $h$ large enough and is locally analytic. By the equivalence in \cite[Thm. 4.1.7]{jacinto2023solid} (we use also the notation in \textit{loc. cit.} below), 
    \[\begin{array}{rl}
        (\iota(\cW^{\vee})\otimes \chi)\otimes_{E_{\blacksquare}[I]}^L\cC^{\la}(I,E)[-\dim G]&\simeq \RiHom_{{E_{\blacksquare}[I]}}(\cW, \cC^{\la}(I,E))\\
        &=\RiHom_{\mathrm{Mod}_{E_{\blacksquare}}^{\rm qc}(\cD(I,E))}(j^*\cW,j^*\cC^{\la}(I,E)).
    \end{array}\]    
    since both $\cW$ and $\cC^{\la}(I,E)$ are (derived) locally analytic. Note that $j^*\cC^{\la}(I,E)=(\cD^h(I,E)\otimes \chi^{-1}[\dim G])_h$ (here $\chi^{-1}$ is viewed as a bimodule over $\cD(G,E)$ where $G$ acts trivially on the right and via $\chi^{-1}$ on the left) and $j^*\cW=(\cW)_{h\textrm{ large}}$ by \cite[Exam. 4.1.9]{jacinto2023solid}. Use the trick as before we get (see also \cite[Prop. 4.1.13]{jacinto2023solid})
    \[\begin{array}{rl}
        &\RiHom_{\mathrm{Mod}_{E_{\blacksquare}}^{\rm qc}(\cD(I,E))}(j^*\cW,j^*\cC^{\la}(I,E))\\
        =&R\varprojlim_{h}\RiHom_{E_{\blacksquare}[I]}(\cW,\cD^h(I,E)\otimes \chi^{-1}[\dim G])\\
        =&R\varprojlim_{h}(\iota(\cW^{\vee}) \otimes\chi)\otimes^L_{E_{\blacksquare}[I]}\cD^h(I,E)\otimes \chi^{-1}[\dim G]\\
        =&\iota(\cW^{\vee})[\dim G].
    \end{array}\]
    where we used that $\cD^h(I,E)$ are idempotent over $E_{\blacksquare}[I]$ \cite[Cor. 5.11]{rodrigues2022solid} and that $\cW^{\vee}$ is a $\cD^h(I,E)$-module for $h$ large. The result follows by returning the right module to a left module via the involution $\iota$.
\end{proof}
\begin{dfn}\label{definitionDBZ}
    For $V\in \mathrm{Mod}_{E_{\blacksquare}}(\cD(G,E))$, we define
    \[\bbD_{\BZ}(V):=\RiHom_{E_{\blacksquare}[G]}(V,\cC^{\la}_c(G,E))\]
    which is in $\mathrm{Mod}_{E_{\blacksquare}}(\cD(G,E))$ (a priori may not in $\mathrm{Rep}_{E_{\blacksquare}}^{\rm la}(G)$!) via the left translations of $\cC^{\la}_c(G,E)$ by $G$.
\end{dfn}
\begin{rem}
    Let $\cC_c^{\rm cont}(G,E)$ be the space of compactly supported continuous functions on $G$ with values in $E$. Then the space of derived locally analytic vectors of $\cC_c^{\rm cont}(G,E)$ is $\cC_c^{\rm la}(G,E)$ ($\cC_c^{\rm cont}(G,E)^{R\mathrm{la}}=\oplus_{G/I}\cC^{\rm cont}(gI,E)^{R\mathrm{la}}$ by \cite[Prop. 3.1.6 (3)]{jacinto2023solid} and $\cC^{\rm cont}(gI,E)^{R\mathrm{la}}=\cC^{\rm cont}(gI,E)^{\mathrm{la}}$ since $\cC^{\rm cont}(gI,E)$ is admissible \cite[Prop. 4.48]{rodrigues2022solid}). By the adjunction in \cite[Prop. 6.2.1]{jacinto2023solid}, if $V\in\mathrm{Rep}_{E_{\blacksquare}}^{\rm la}(G)$, we have
    \[\bbD_{\BZ}(V)=\RiHom_{E_{\blacksquare}[G]}(V,\cC^{\rm cont}_c(G,E)).\]
    Similarly by \textit{loc. cit.}, since $R\Gamma(\frg, \cC_c^{\rm la}(G,E))=\cC_c^{\rm sm}(G,E)$ (using the Poincaré lemma in the proof of \cite[Prop. 5.12]{rodrigues2022solid}) is the space of compactly supported smooth functions on $G$, we see if $V$ is a smooth representation of $G$, then
    \[\bbD_{\BZ}(V)=\RiHom_{\cD^{\rm sm}(G,E)}(V,\cC^{\rm sm}_c(G,E))\]
    where $\cD^{\rm sm}(G,E)$ is the smooth distribution algebra in \cite[\S 5.1]{jacinto2023solid}. This recovers the classical definition of the duality for smooth representations \cite[IV.5]{bernstein1992notes}.
\end{rem}
\begin{thm}\label{theoremBZdual}
    Let $M\in \cO^{\frb}_{\rm alg}$ and $\chi_{\sm}$ be a smooth character of $T$. 
    Then there exists an isomorphism of locally analytic representations 
    \[\bbD_{\rm BZ}(\cF_{B}^G(M,\chi_{\sm}))\simeq \cF_{\ovB}^G(\Hom_E(M,E)^{\overline{\frn}^{\infty}},\bbD_{\BZ}(\chi_{\sm})).\]
\end{thm}
\begin{proof}
    The arguments for Theorem \ref{theoremdualKS} together with Proposition \ref{propositionsolidresolution} and Proposition \ref{theoremsoliddualcompactinduction} give the duality.
\end{proof}
\subsection{Stein spaces}\label{subsectionquasistein}
For the discussions on patching functors later, we make some preparation for coherent sheaves on Stein spaces, following \cite{jacinto2023solid}.

We let $s\geq 1$ and $S=\cO_E[[X_1,\cdots,X_s]]\simeq \cO_E[[\Z_p^s]]$ where $\cO_E[[\Z_p^s]]$ denotes the Iwasawa algebra. Let $S\rightarrow R$ be a local morphism of complete Noetherian local rings. We choose a presentation of $R$ over $S$: let $A=S\otimes_{\cO_E,\blacksquare}B$ for some $B= \cO_E[[Y_1,\cdots,Y_t]]$ and suppose that $R=A/I$ for an ideal $I\subset A$. For positive integers $h,k$, we let $S_h^+=\cO_E\langle\frac{X_1^h}{p},\cdots,\frac{X_s^h}{p}\rangle,B_k^+=\cO_E\langle\frac{Y_1^k}{p},\cdots,\frac{Y_t^k}{p}\rangle,S_h=S_h^+[\frac{1}{p}]$, and similarly $A_{h,k}^+:=S_h^+\otimes_{\cO_E,\blacksquare}B_k^+$. Let $R_{h,k}^+:=A_{h,k}^+/I$. We write $A_{h}:=A_{h,h}, R_{h}:=R_{h,h}$ for short. Finally, define Fréchet-Stein algebras $S^{\rig}:=\varprojlim_h S_h, R^{\rig}:=\varprojlim_h R_h,A^{\rig}:=\varprojlim_h A_h$. Then $A^{\rm rig}=S^{\rm rig}\otimes_{E,\blacksquare}^LB^{\rm rig}$ (\cite[Cor. A.65, Cor. A.67, Prop. A.68]{bosco2021p}).

The rigid generic fiber $\Spf(R)^{\rig}$ of the formal scheme $\Spf(R)$ over $\Spf(\cO_E)$ admits a covering by affinoids $\Spa(R_h,R_h^+)$. From any complete Huber pair as $(R_h,R_h^+)$ we obtain an analytic ring $(R_h,R_h^+)_{\blacksquare}$ by \cite[Thm. 3.28]{andreychev2021pseudocoherent}. We write $\Mod_{(R_h,R_h^+)_{\blacksquare}}$ for the ($\infty$-)derived category of $(R_h,R_h^+)_{\blacksquare}$-modules  and write $\Mod_{(R_h,R_h^+)_{\blacksquare}}^{\heartsuit}$ for its heart \cite[Prop. 7.5]{scholze2019condensed}. For $h'>h$, by \cite[Prop. 3.34, Lem. 3.31]{andreychev2021pseudocoherent} (cf. \cite[Lem. 2.1.9]{jacinto2023solid}), the map $(R_{h'},R_{h'}^+)_{\blacksquare}\rightarrow (R_h,R_h^+)_{\blacksquare}$ of analytic rings factors through $(R_{h'},R_{h'}^+)_{\blacksquare}\rightarrow (R_h,\cO_E)_{\blacksquare}=(R_h,\cO_E+R_{h}^{00})_{\blacksquare}\rightarrow (R_h,R_h^+)_{\blacksquare}$ where $R_{h}^{00}$ denotes the subset of topologically nilpotent elements. Here $(R_h,\cO_E)_{\blacksquare}$ denotes the analytic ring induced from $\cO_{E,\blacksquare}$ \cite[Prop. 2.16]{andreychev2021pseudocoherent}, and $\Mod_{(R_h,\cO_E)_{\blacksquare}}=\Mod_{E_{\blacksquare}}(R_h)$ is the category of condensed $R_h$-modules whose underlying condensed $E$-vector spaces are solid. The category of quasi-coherent sheaves on $\Spf(R)^{\rig}$ is then equivalent to the limit (cf. \cite[\S 2.1.2, \S 4.1]{jacinto2023solid})
\[\mathrm{Mod}_{E_{\blacksquare}}^{\rm qc}(R^{\rig}):=\varprojlim_h\Mod_{(R_h,R_h^+)_{\blacksquare}}=\varprojlim_h\Mod_{E_{\blacksquare}}(R_h).\]
Thus a quasi-coherent sheaf on $\Spf(R)^{\rig}$ can be given by a sequence $(M_h)_{h}$ where $M_h$ are solid $R_h$-modules together with $M_{h'}\otimes_{R_{h'},\blacksquare}^LR_h= M_h$ for $h'>h$. If $(M_h)_h,(N_h)_h\in \mathrm{Mod}_{E_{\blacksquare}}^{\rm qc}(R^{\rig})$, then
\[R\Hom_{ \mathrm{Mod}_{E_{\blacksquare}}^{\rm qc}(R^{\rig})}((M_h)_h,(N_h)_h)=R\varprojlim_hR\Hom_{R_h}(M_h,N_h). \]
There is a ``global section'' functor 
\begin{align}\label{equationglobalsection}
    j_*:\mathrm{Mod}^{\rm qc}_{E_{\blacksquare}}(R^{\rig})&\rightarrow \mathrm{Mod}_{E_{\blacksquare}}(R^{\rig}):\\
    (M_h)_{h}&\mapsto R\varprojlim_h M_h.\nonumber
\end{align}
Conversely, we have the ``localization'':
\begin{align}\label{equationloclization}
    j^*:\mathrm{Mod}_{E_{\blacksquare}}(R^{\rig})&\rightarrow \mathrm{Mod}^{\rm qc}_{E_{\blacksquare}}(R^{\rig}):\\
    M&\mapsto (M\otimes_{R^{\rm rig},\blacksquare}^LR_h)_h.\nonumber
\end{align}
Recall by \cite[Prop. 5.10]{rodrigues2022solid}, $A_h$ and $A^{\rm rig}$ are idempotent algebras over $A$, namely $A_h\otimes_{A,\blacksquare}^LA_h=A_h$ and $A^{\rm rig}\otimes_{A,\blacksquare}^LA^{\rm rig}=A^{\rm rig}$. Moreover, $A^{\rm rig}$ and $A_h$ are flat over $A$ \cite[Prop. 4.7]{schneider2003algebras}. Since $R=A/I$, similar statements hold for $R_{h}, R^{\rm rig}$. We see $R_h=A_h\otimes_{A}^LA/I$ and $R^{\rm rig}=A^{\rm rig}\otimes_{A}^LA/I$. Hence $R_h=A_h\otimes_{A^{\rig},\blacksquare}^LR^{\rig}$ and $R_h\otimes_{R,\blacksquare}^LR_h=(A_h\otimes_{A,\blacksquare}^LA_h)\otimes_{A,\blacksquare}^LA/I=R_h$. And $R_h\otimes_{R^{\rm rig},\blacksquare}^LR_h
=R_h\otimes_{R^{\rm rig},\blacksquare}^LR^{\rm rig}\otimes_{R^{\rm rig},\blacksquare}^LR_h=R_h\otimes_{R,\blacksquare}^LR_h=R_h$.
\begin{lemma}\label{lemmafullyfaithfulglobalsection}
    The functor $j_*$ in (\ref{equationglobalsection}) is fully faithful with the left adjoint $j^*$. 
\end{lemma}
\begin{proof}
    The statement is proved in \cite[\S 4.1]{jacinto2023solid} if $R=A=\cO_E[[\Z_p^{d}]]$ is the Iwasawa algebra of a compact $p$-adic Lie group. In this case $A^{\rig}=\cD(\Z_p^{d},E)$. The fully faithfulness and the adjunction property of $j_*$ is \cite[Cor. 4.1.8]{jacinto2023solid}. In general, $R=A/I$ for $A=\cO_E[[\Z_p^d]], d=t+s$. There are functors $i_*:\mathrm{Mod}^{\rm qc}_{E_{\blacksquare}}(R^{\rig})\rightarrow \mathrm{Mod}^{\rm qc}_{E_{\blacksquare}}(A^{\rig}),\mathrm{Mod}_{E_{\blacksquare}}(R^{\rig})\rightarrow \mathrm{Mod}_{E_{\blacksquare}}(A^{\rig})$ given by restrictions via the ring maps $A^{\rig}\rightarrow R^{\rig}$ and $A_h\rightarrow R_h$. The functors $j_*$ and $j^*$ defined above (we use the same notation for $A$ and $R$) commute with $i_*$: since $R_h=A_h\otimes_{A^{\rig},\blacksquare}^LR^{\rig}$, we have $j^*M=(R_h\otimes_{R^{\rig},\blacksquare}^LM)_h=(A_h\otimes_{A^{\rig},\blacksquare}^L M)_h$ for $M\in\Mod_{E_{\blacksquare}}(R^{\rig})$. Using the statement for $A$, for any $\cM=(M_h)_h\in \mathrm{Mod}^{\rm qc}_{E_{\blacksquare}}(R^{\rig})$, the natural morphism $ i_*j^*j_*\cM=j^*j_*i_*\cM\rightarrow i_*\cM$ is an isomorphism. Notice that the functor $i_*$ is conservative, namely a morphism $M\rightarrow N$ of $R^{\rm rig}$-modules (or $R_h$-modules) is an isomorphism if and only if it is an isomorphism of $A^{\rig}$-modules (this can be checked on the abelian level by taking cohomologies). Hence the counit maps $j^*j_*\cM\rightarrow \cM$ are also isomorphisms which implies the fully-faithfulness of the functor $j_*$ for $R$. 
\end{proof}
    The ring maps $S\rightarrow A\rightarrow R$ induce morphisms 
    \begin{center}
        \begin{tikzcd}
            \Spf(R)^{\rig}\arrow[r,"i"]\arrow[rd,"f"]&\Spf(A)^{\rig}=\Spf(S)^{\rig}\times\Spf(B)^{\rig}\arrow[d,"g"]\\
            &\Spf(S)^{\rig}\\
        \end{tikzcd}
    \end{center}
    of rigid analytic spaces. These maps admit $!$-functors: $i^!,i_!=i_*,g^!,g_!$, etc., for the six functor formalism in \cite[\S 3]{camargo2024analytic} (building on \cite{mann2022six}) of quasi-coherent modules using, for example, \cite[Prop. 3.3.6]{camargo2024analytic}. We will not essentially need the general machinery as we will treat only coherent sheaves later as in Lemma \ref{lemmacompactlysupported} below. 
    \begin{lemma}\label{lemmapushforward}
        Suppose that $\cM=(M_{h,k})_{h,k}\in \Mod^{\rm qc}_{E_{\blacksquare}}(A^{\rm rig})$ and $M=j_*\cM\in \Mod_{E_{\blacksquare}}(A^{\rm rig})$. Then the natural map $M\rightarrow R\varprojlim_h(S_{h}\otimes_{S,\blacksquare}^LM)$ is an isomorphism and $g_*\cM$ is the quasi-coherent sheaf attached to the $S^{\rm rig}$-module $M$ via the localization.
    \end{lemma}
    \begin{proof}
        By \cite[Cor. 4.1.5]{jacinto2023solid}, the inverse systems $(S_{h}\otimes_{S,\blacksquare}^LM)_h$ is equivalent to (note that $\underline{\Hom}_E(S_h,E)$ is Smith)
        \[(\RiHom_S(\underline{\Hom}_E(S_h,E)[-s], M))_h.\]
        Taking inverse limit
        \[\begin{array}{rl}
            R\varprojlim_hS_{h}\otimes_{S,\blacksquare}^LM&=R\varprojlim_{h}\RiHom_S(\underline{\Hom}_E(S_h,E)[-s], M)\\
            &=R\varprojlim_{h, h',k}\RiHom_S(\underline{\Hom}_E(S_h,E)[-s], M_{h',k})\\
            &=R\varprojlim_{k,h',h}S_h\otimes_{S,\blacksquare}^LM_{h',k}\\
            &=R\varprojlim_{k,h}M_{h,k}=M
        \end{array}\]
        where we applied \textit{loc. cit.} again for the inverse system over $h$ to get the third equality. 
    \end{proof}
    For an affinoid algebra like $S_h$, an object $M\in \Mod_{E_{\blacksquare}}(S_h)$ is said to be perfect if $M$ is quasi-isomorphic to a finite complex of finite projective $S_h$-modules (cf. \cite[Prop. 5.12]{andreychev2021pseudocoherent}). 
    \begin{lemma}\label{lemmacompactlysupported}
        Suppose that $M\in\mathrm{Mod}_{E_{\blacksquare}}(R^{\rm rig})$ such that $M\otimes^L_{S^{\rm rig},\blacksquare}S_h$ is a perfect $S_h$-module for all $h$. Then $M=j_*\cM\in\Mod_{E_{\blacksquare}}(R^{\rm rig})$ is in the essential image of $j_*$ where $\cM=j^*M$. Moreover, there exists an isomorphism $f_*\cM=f_!\cM$ in $\mathrm{Mod}_{E_{\blacksquare}}^{\rm qc}(S^{\rig})$.
    \end{lemma}
    \begin{proof}
        Fix $h$. We first prove that there exists $k$ such that $M_{S_h}:=M\otimes^L_{S^{\rm rig},\blacksquare}S_h$ is an $R_{h,k}$-module (extending the $R$-module structure). In other words, we find $k$ such that the natural $R$-map $M_{S_h}\rightarrow R_{h,k}\otimes_{R,\blacksquare}^LM_{S_h}=A_{h,k}\otimes_{A,\blacksquare}^LM_{S_h}$ is an isomorphism. Since $A_{h,k}=S_h\otimes_{E,\blacksquare}B_k$ and $M_{S_h}=M_{S_h}\otimes_{S,\blacksquare}^LS_h$, we see $ A_{h,k}\otimes_{A,\blacksquare}^LM_{S_h}=(S_h\otimes_{E,\blacksquare}B_k)\otimes_{S\otimes_{E,\blacksquare}B,\blacksquare}^LM_{S_h}=B_k\otimes_{B,\blacksquare}^L(S\otimes_{E,\blacksquare}B)\otimes_{S\otimes_{E,\blacksquare}B,\blacksquare}^LM_{S_h}=B_k\otimes_{B,\blacksquare}^LM_{S_h}$. Hence it is enough to show that there exists $k$ such that the $B$-action on $M_{S_h}$ extends to $B_k$ (recall that $B_k\otimes_{B,\blacksquare}^LB_k=B_k$). By \cite[Thm. 4.36]{rodrigues2022solid}, it suffices to show that the cohomology groups $H^n(M_{S_h})$ for all $n$ are $B_k$-modules for some $k$. Since $M_{S_h}$ is a perfect $S_h$-module, there exist finitely many $n$ such that $H^n(M_{S_h})\neq 0$ and these cohomology groups are finite $S_h$-modules, in particular Banach spaces. The desired action of some $B_k$ follows from \cite[Prop. 4.41]{rodrigues2022solid}. 
        
        Now we prove $f_*\cM=f_!\cM$. Since $i$ is a closed embedding, $i_!\cM=i_*
        \cM$. We may assume $R=A$. Let $\cM_{S_h}:=(M_{h,k}:=M_{S_h}\otimes_{A,\blacksquare}^LA_{h,k})_k\in \varprojlim_{k}\Mod_{E_{\blacksquare}}(A_{h,k})$. Then $M_{h,k}=M_{S_h}\otimes_{A,\blacksquare}^LA_{h,k}\otimes_{A,\blacksquare}^LA_{h,k'}=M_{h,k'}$ for $k$ above and $k'\geq k$. We calculate that for $N_h\in \mathrm{Mod}_{E_{\blacksquare}}(S_h)$,
        \[
        \begin{array}{rl}
            R\Hom_{S_h}(M_{S_h},N_h)&=R\Hom_S(M_{S_h}\otimes_{A,\blacksquare}^LA_{h,k},N_h)\\
            &=\varinjlim_{k'}R\Hom_A(M_{S_h},\RiHom_{S}(A_{h,k'},N_h))\\
            &=\varinjlim_{k'}R\Hom_B(E,\RiHom_{S_h}(M_{S_h},\RiHom_{S}(A_{h,k'},N_h)))\\
            &=R\Hom_B(E,\RiHom_{S_h}(M_{S_h},\varinjlim_{k'}\RiHom_{S}(A_{h,k'},N_h)))
        \end{array}
        \]
        where for the last equality we used that $E$ is a compact $B$-module (cf. \cite[Thm. 5.7]{rodrigues2022solid}) and $M_{S_h}$ is a compact $S_h$-module being perfect (\cite[Lem. 5.46, Cor. 5.51.1]{andreychev2021pseudocoherent}). The map $\RiHom_{S}(A_{h,k'},N_h)=\RiHom_{E}(B_{k'},N_h)\rightarrow \RiHom_{S}(A_{h,k''},N_h)$ for $k''> k'$ factors through $R\underline{\Hom}_{E}(B_{k'},E)\otimes_{E,\blacksquare}^L N_h$ by \cite[Lem. 4.1.4]{jacinto2023solid}. Hence $\varinjlim_{k'}\RiHom_{S}(A_{h,k'},N_h)=\varinjlim_{k'}R\underline{\Hom}_{E}(B_{k'},E)\otimes_{E,\blacksquare}^L N_h=\underline{\Hom}_E(B^{\rig},E)\otimes_{E,\blacksquare}^LN_h$. And by \cite[Exam. 4.1.9]{jacinto2023solid}, 
        \begin{equation}\label{equationpullbacklocallyanalyticfunctinos}
            (A_{h,k}\otimes_{A,\blacksquare}^L(\underline{\Hom}_E(B^{\rig},E)\otimes_{E,\blacksquare}^LN_h))_k=(B_k\otimes_{B,\blacksquare}^L\underline{\Hom}_E(B^{\rig},E)\otimes_{E,\blacksquare}^LN_h)_k=(B_k\otimes_{E,\blacksquare}^LN_h[t])_k.    
        \end{equation}
        By Lemma \ref{lemmapushforward}, $g_*\cM=(M_{S_h})_h$. Hence
        \[\begin{array}{rl}
            \RHom_{\Mod_{E_{\blacksquare}}^{\rm qc}(S^{\rm rig})}(g_*\cM,(N_h)_h)&=R\varprojlim_hR\Hom_S(M_{S_h},N_h)\\
            &=R\varprojlim_{h}R\Hom_{A}(M_{S_h},\varinjlim_{k'}\RiHom_{E}(B_{k'},E)\otimes_{E,\blacksquare}^L N_h)\\
            &=R\varprojlim_{h}R\Hom_{A}(M_{S_h},\underline{\Hom}_{E}(B^{\rm rig},E)\otimes_{E,\blacksquare}^L N_h)\\
            &=R\varprojlim_{h,k'}R\Hom_A(M_{S_h},A_{h,k'}\otimes_{S,\blacksquare}^LN_h[t])\\
            &=R\Hom_{\Mod_{E_{\blacksquare}}^{\rm qc}(A^{\rm rig})}(\cM,g^!(N_h)_h)\\
            &=R\Hom_{\Mod_{E_{\blacksquare}}^{\rm qc}(S^{\rm rig})}(g_!\cM,(N_h)_h)
        \end{array}\]
        where for the fourth equality we used (\ref{equationpullbacklocallyanalyticfunctinos}) and the fully faithfulness of $j_!$ \cite[Thm. 4.1.7]{jacinto2023solid} since $M_{S_h}$ and $\underline{\Hom}_E(B^{\rig},E)\otimes_{E,\blacksquare}^LN_h$ are derived locally analytic for the action of $\Z_p^{t+s}$. We used that $g$ is cohomologically smooth of relative dimension $t$ for the fifth equality. We conclude that $g_*\cM=g_!\cM$ by the Yoneda lemma.
    \end{proof}
\subsection{Patching functors and patching modules}\label{sectionpatchingmodules}
The Taylor-Wiles-Kisin patching method has been indispensible for $p$-adic Langlands program and motivated Emerton-Gee-Hellmann's categorical $p$-adic local Langlands conjecture \cite[\S 3]{emerton2023introduction}. The abstract formalism of the patching functor for $\GL_n(\cO_L)$-representations was proposed in \cite[\S 6]{emerton2015lattices} and the patching of completed cohomologies or $\GL_n(L)$-representations was carried out in \cite{caraiani2013patching}. The method was applied to locally analytic settings by Breuil-Hellmann-Schraen in \cite{breuil2019local}, etc. and more recently in \cite{hellmann2024patching}. We will work with the abstract patching modules reviewed below. The Serre duality of patching modules in this subsection is well-known at least in modular settings (e.g. \cite{manning2021patching, manning2024modellmultiplicitiescertain}). Our purpose is to explain its direct relationship with the Bernstein-Zelevinsky duality for locally analytic representations.

We assume the existence of the following abstract patching data. We suppose that $G=\GL_d(L),d\geq 2$ with the standard Iwahori subgroup $I$. 

Let $S_{\infty}\rightarrow R_{\infty}$ be a local morphism of complete Noetherian local rings over $\cO_E$ with the residue fields $\cO_E/\varpi_E$. Assume that there exists $s\geq 1$ such that $S_{\infty}=\cO_E[[X_1,\cdots,X_s]] \simeq \cO_E[[J]]$ for $J=\Z_p^s$. Then the Iwasawa algebra $S_{\infty}[[I]]\simeq \cO_L[[\widetilde{I}]]$ where $\widetilde{I}= I\times J=I\times\Z_p^s$. Similarly, write $\widetilde{G}=G\times J$.

Suppose that there is a (big patching) module $M_{\infty}$ over the ring $R_{\infty}[\GL_n(L)]$ such that there exists an isomorphism $M_{\infty}|_H\simeq S_{\infty}[[H]]^a$ as topological $H$-modules for an open normal pro-$p$ subgroup $H\subset I$. Hence $M_{\infty}$ is finite projective over $S_{\infty}[[I]]$. 

We assume the existence of the Poincaré dual of $M_{\infty}$ consisting of the following data. Suppose that there are isomorphisms $\eta:S_{\infty}\rightarrow S_{\infty}',R_{\infty}\rightarrow R'_{\infty}$ of local $\cO_E$-algebras and $M_{\infty}'$ is a big patching module over $R_{\infty}'$. We assume that there exists an $R_{\infty}\times \GL_n(L)$-equivariant $S_{\infty}$-linear isomorphism
\[\Hom_{S_{\infty}[[I]]}(M_{\infty}, S_{\infty}[[I]])\simeq M_{\infty}'\]
where $S_{\infty},R_{\infty}$ acts on $M_{\infty}'$ via $\eta$. The existence of such isomorphism is provided in \cite[Cor. D.9]{ding2024towards} which is a patched version of the Poincaré duality of completed cohomologies (\cite[\S 1.3]{calegari2012completed}). The map $\eta$ in \cite{ding2024towards} is essentially induced by $\rho\mapsto \rho^{\vee}\otimes\epsilon^{1-d}$ of Galois representations where $(-)^{\vee}$ is the linear dual and $\epsilon$ denotes the cyclotomic character (see \cite[(3.9)]{zhu2020coherent} in terms of the Cartan involution of the $C$-group).

With the big patching modules, we can define patching functors for locally analytic representations. Let $M_{\infty}^{\rig}=\cD(\widetilde{I},E)\otimes_{S_{\infty}[[I]],\blacksquare}M_{\infty}$ and define similarly $ M_{\infty}^{',\rm rig}$. Using that $M_{\infty}$ is finite projective over $S_{\infty}[[I]]$, we obtain an $R_{\infty}$-linear $\cD(\widetilde{I},E)$-isomorphism
\[\RiHom_{\cD(\widetilde{G},E)}(M_{\infty}^{\rm rig}, \overline{\cD}(\widetilde{G},E))=\RiHom_{\cD(\widetilde{I},E)}(M_{\infty}^{\rm rig}, \cD(\widetilde{I},E))\simeq M_{\infty}^{',\rm rig}\]
where $\overline{\cD}(\widetilde{G},E)=\RiHom_{E_{\blacksquare}[\widetilde{G}]}(\cC^{\la}_c(\widetilde{G},E),E)$ with two left $\cD(\widetilde{G},E)$-module structures given by left and right multiplications (cf. \cite{schneider2005duality} or \cite[\S 4.2]{jacinto2023solid}).

Set $R_{\infty}^{\rig}=\varprojlim_hR_h,S_{\infty}^{\rig}=\varprojlim_hS_h$ as projective limits of affinoid $E$-algebras as in \S\ref{subsectionquasistein}, which are Fr\'echet-Stein algebras. Recall we have a localization functor $j^*$ (\ref{equationloclization}) for solid $R_{\infty}^{\rig}$-modules. Let $\cM_{\infty}^{\rig}:=j^*M_{\infty}^{\rig}\in \mathrm{Mod}_{E_{\blacksquare}}^{\rm qc}(R_{\infty}^{\rig})$ and similarly $\cM_{\infty}^{',\rig}$.
\begin{dfn}\label{definitionpatchingfunctor}
    We define the patching functor $\frA_{\infty}^{\rig}:\mathrm{Mod}_{E_{\blacksquare}}(\cD(G,E))\rightarrow \mathrm{Mod}_{E_{\blacksquare}}^{\rm qc}(R_{\infty}^{\rig})$ (following the notation of \cite{emerton2023introduction}) by 
    \[\frA_{\infty}^{\rig}(\pi):=j^*M_{\infty}^{\rm rig}\otimes^L_{\cD(G,E),\blacksquare}\pi=j^*(M_{\infty}^{\rm rig}\otimes^L_{\cD(G,E),\blacksquare}\pi)\]
    for $\pi\in \mathrm{Mod}_{E_{\blacksquare}}(\cD(G,E))$. Similarly define $\frA_{\infty}^{',\rig}(\pi):=\cM_{\infty}^{',\rm rig}\otimes^L_{\cD(G,E),\blacksquare}\pi$.
\end{dfn}
For convenience, we also define the functor $A_{\infty}^{\rig}:\mathrm{Mod}_{E_{\blacksquare}}(\cD(G,E))\rightarrow \mathrm{Mod}_{E_{\blacksquare}}(R_{\infty}^{\rig})$ by
\[\pi\mapsto A_{\infty}^{\rig}(\pi):=M_{\infty}^{\rm rig}\otimes^L_{\cD(G,E),\blacksquare}\pi\]
so that $\frA_{\infty}^{\rig}=j^*\circ A_{\infty}^{\rig}$.

If $\cW$ is a locally analytic $I$-representation concentrated in degree $0$, then
\[A_{\infty}^{\rig}(\cind_I^G\cW)=M_{\infty}^{\rm rig}\otimes^L_{\cD(G,E),\blacksquare}\cind_I^G\cW=M_{\infty}^{\rm rig}\otimes_{\cD(I,E),\blacksquare}\cW\]
since $M_{\infty}^{\rm rig}$ is finite projective over $\cD(I,E)\otimes_{E,\blacksquare}\cD(J,E)$. We see $A_{\infty}^{\rig}(\cind_I^G(-))$ is exact on $\mathrm{Mod}_{E_{\blacksquare}}^{\heartsuit}(\cD(I,E))$. Suppose that the Hecke algebra $\cH$ acts on $\cind_{I}^G\cW$ as in \S\ref{subsectionKoszulcomplex} and $\frm$ is the corresponding maximal ideal of $\cH$, then $\cH$ acts on $\cind_I^G\cW$ as homomorphisms of $\cD(G,E)$-modules. Apply $A_{\infty}^{\rig}$ to the complex $\cind_{I}^G\cW\otimes^L_{\cH}\cH/\frm:=\wedge^{\bullet}\cH^d\otimes_{\cH}\cind_{I}^G\cW$ in $\mathrm{Mod}_{E_{\blacksquare}}(\cD(G,E))$, we see
\[A_{\infty}^{\rig}(\cind_{I}^G\cW\otimes^L_{\cH}\cH/\frm)\simeq A_{\infty}^{\rig}(\cind_{I}^G\cW)\otimes^L_{\cH}\cH/\frm.\]
And the isomorphism holds after pulling back via $j^*$.

\begin{lemma}\label{lemmapatchingdualcompactinduction}
    Let $\cW\in\mathrm{Mod}^{\heartsuit}_{E_{\blacksquare}}(\cD(I,E))$ be a solid locally analytic representation over a Smith $E$-space. There exist isomorphisms of solid $S_{\infty}^{\rig}$-modules
    \[M_{\infty}^{',\rm rig}\otimes_{\cD(I,E),\blacksquare}^L\cW^{\vee}\simeq \RiHom_{S_{\infty}^{\rig}}(M_{\infty}^{\rm rig}\otimes^L_{\cD(I,E),\blacksquare}\cW, S_{\infty}^{\rig})=\underline{\Hom}_{S_{\infty}^{\rig}}(M_{\infty}^{\rm rig}\otimes_{\cD(I,E),\blacksquare}\cW, S_{\infty}^{\rig}).\]
\end{lemma}
\begin{proof}
    We calculate that
    \[\begin{array}{rl}
            &M_{\infty}^{',\rm rig}\otimes^L_{\cD(I,E),\blacksquare}\cW^{\vee}\\
        \simeq &\RiHom_{\cD(\widetilde{I},E)}(M_{\infty}^{\rm rig},\cD(\widetilde{I},E))\otimes^L_{\cD(\widetilde{I},E),\blacksquare}(\cD(\widetilde{I},E)\otimes_{\cD(I,E),\blacksquare}^L\cW^{\vee})\\
        \simeq &\RiHom_{\cD(\widetilde{I},E)}(M_{\infty}^{\rm rig}, \cD(\widetilde{I},E)\otimes_{\cD(I,E),\blacksquare}^L\cW^{\vee})    
    \end{array}\]
    where for the last equality we used that $M_{\infty}^{\rig}$ is fintie projective over $\cD(\widetilde{I},E)$. While $\cD(\widetilde{I},E)\otimes_{\cD(I,E),\blacksquare}^L\cW^{\vee}=(S_{\infty}^{\rig}\otimes_{E,\blacksquare}^{L}\cD(I,E))\otimes_{\cD(I,E),\blacksquare}^L\cW^{\vee}=S^{\rig}_{\infty}\otimes_{E,\blacksquare}^L\cW^{\vee}=S^{\rig}_{\infty}\otimes_{E,\blacksquare}\cW^{\vee}$ since $\cW^{\vee}$ is flat over $E$ \cite[Cor. A.65]{bosco2021p}. As $S_{\infty}^{\rig}$ is nuclear \cite[Prop. 3.29]{rodrigues2022solid} and $\cW^{\vee}=\RiHom_E(\cW,E)$ for $\cW$ Smith, we have $S_{\infty}^{\rig}\otimes_{E,\blacksquare}^L\cW^{\vee}=\RiHom_E(\cW,S_{\infty}^{\rig})$ (cf. \cite[Prop. A.55]{bosco2021p}). Take the adjunction (cf. \cite[Lem. 3.1.7]{jacinto2023solid})
    \[\begin{array}{rl}
        \RiHom_{\cD(\widetilde{I},E)}(M_{\infty}^{\rm rig}, \cD(\widetilde{I},E)\otimes_{\cD(I,E),\blacksquare}^L\cW^{\vee})=&\RiHom_{E_{\blacksquare}[\widetilde{I}]}(M_{\infty}^{\rm rig}, \RiHom_E(\cW,S_{\infty}^{\rig}))\\
        =&\RiHom_{E_{\blacksquare}[J]}(E,\RiHom_{E_{\blacksquare}[I]}(M_{\infty}^{\rm rig}, \RiHom_E(\cW,S_{\infty}^{\rig})))\\
        =&\RiHom_{E_{\blacksquare}[J]}(E,\RiHom_{E}(M_{\infty}^{\rm rig}\otimes^L_{E_{\blacksquare}[I]}\cW, S_{\infty}^{\rig}))\\
        =&\RiHom_{E_{\blacksquare}[J]}(M_{\infty}^{\rm rig}\otimes^L_{E_{\blacksquare}[I]}\cW, S_{\infty}^{\rig}),
    \end{array}\]
    we get the result since $S_{\infty}^{\rig}=\cD(J,E)$ is an idempotent algebra over $E_{\blacksquare}[J]$.
\end{proof}
The following proposition is a standard result of the eigenvariety machinery.
\begin{prop}\label{propositionfiniteslopeperfect}
    Let $\cW=\cW_{\natural,r}(\tau(M)^{\vee}\otimes\chi_{\rm sm})$ as in Theorem \ref{theoremresolution}. Then the solid $S_{\infty}^{\rig}$-module $A_{\infty}^{\rig}(\cind_I^G\cW)\otimes_{\cH}^L\cH/\frm$ is a perfect object in $\Mod_{E_{\blacksquare}}^{\rm qc}(S_{\infty}^{\rig})$.
\end{prop}
\begin{proof}
    We need to show that $(S_h\otimes_{S_{\infty}^{\rm rig},\blacksquare} A_{\infty}^{\rig}(\cind_I^G\cW))\otimes_{\cH}^L\cH/\frm$ is a perfect $S_h$-complex for each $h$. We would like to write a proof for $\cW=\cD^{I_n-\mathrm{an}}(I,E)\otimes_{\cD(\frg,I\cap \ovB)}\ovM$ for some $n$ in Remark \ref{remarkWnaturaln} and $\ovM=\tau(M)^{\vee}\otimes\chi_{\rm sm}$. Then $M_{\infty}^{\rig}\otimes_{\cD(I,E),\blacksquare}\cW=M_{\infty}^{\rig}\otimes_{\cD(I,E),\blacksquare}\cD^{I_n-\mathrm{an}}(I,E)\otimes_{\cD(\frg,I\cap \ovB)}\ovM$. This will not change $A_{\infty}^{\rig}(\cind_I^G\cW\otimes_{\cH}^L\cH/\frm)$.

    Since $M_{\infty}^{\rig}$ is finite projective over $\cD(\widetilde{I},E)$, we can write $M_{\infty}^{\rig}=e.\cD(\widetilde{I},E)^a\subset \cD(\widetilde{I},E)^a$ for some integer $a$ and a $\cD(\widetilde{I},E)$-linear projection $e:\cD(\widetilde{I},E)^a\rightarrow \cD(\widetilde{I},E)^a$. Then $S_h\otimes_{S_{\infty}^{\rm rig},\blacksquare} A_{\infty}^{\rig}(\cind_I^G\cW)=e.(S_h\otimes_{E,\blacksquare}\cW)^a$. For $t\in T^{-}$,  $U_t\in \cH$ induces $e.(S_h\otimes_{E,\blacksquare}\cW)^a\rightarrow e.(S_h\otimes_{E,\blacksquare}\cW)^a$. Since $\cW$ is Smith, $e.(S_h\otimes_{E,\blacksquare}\cW)^a$ is a compact object in $\Mod_{E_{\blacksquare}}(S_h)$. We claim that if $t$ is sufficiently regular ($t^{-1}(I_n\cap N)t\subset I_{n+1}\cap N$), then $U_t$ is a trace class map \cite[Def. 13.11]{scholze2020analytic}. 
    
    By the construction (Lemma \ref{lemma-psit} and Lemma \ref{lemmapsitcompactinduction}), $\psi_t:\cW\rightarrow \cW$, is induced from the multiplication $\times t^{-1}$ on 
    \[\cD^{I_n-\mathrm{an}}(I_n(I\cap \ovB),E)\otimes_{\cD(\frg,I\cap \ovB)}\ovM=\cD^{\mathrm{an}}(I_n,E)\otimes_{\cD(\frg,I_n\cap \ovB)}\ovM.\] 
    Hence $\psi_t$ factors through the $E$-map
    \[\iota:\cD^{(t^{-1}I_nt\cap I)-\mathrm{an}}(I,E)\otimes_{\cD(\frg,I\cap \ovB)}\ovM\rightarrow \cD^{I_n-\mathrm{an}}(I,E)\otimes_{\cD(\frg,I\cap \ovB)}\ovM.\]
    We show that $\iota$ is trace-class. By \cite[Lem. 3.36]{rodrigues2022solid}, we need to  show that the map factors through the Banach space $\cD^{I_n-\mathrm{an}}(I,E)^B\otimes_{\cD(\frg,I\cap \ovB)}\ovM$ attached to the target (\cite[Def. 3.34]{rodrigues2022solid}). If $\ovM=\cD(\frg,I\cap \ovB)\otimes_{\cD(I\cap \ovB)}\sigma=U(\frg)\otimes_{U(\overline{\frb})}\sigma$ for a finite-dimensional representation $\sigma$ of $\ovB$, then 
    $\cD^{I_n-\mathrm{an}}(I,E)\otimes_{\cD(\frg,I\cap \ovB)}\ovM=\cD^{(I_n\cap N)-\mathrm{an}}(I\cap N,E)\otimes_{E}\sigma$ and $\iota$ factors through $\cD^{(I_n\cap N)-\mathrm{an}}(I\cap N,E)^B\otimes_{E}\sigma$ since $t^{-1}(I_n\cap N)t\subset I_{n+1}\cap N$ is relatively compact in $I_n\cap N$ (cf. \cite[Prop. 4.2.22]{emerton2006jacquet}). In general, $\ovM$ admits a presentation as in Lemma \ref{lemma-psit}
    \[U(\frg)\otimes_{U(\overline{\frb})}\sigma'\rightarrow U(\frg)\otimes_{U(\overline{\frb})}\sigma\rightarrow \ovM\] 
    for some $\sigma,\sigma'$, 
    which can be used to show that $\iota$ factors through $\cD^{I_n-\mathrm{an}}(I,E)^B\otimes_{\cD(\frg,I\cap \ovB)}\ovM$, the cokernel of $\cD^{(I_n\cap N)-\mathrm{an}}(I\cap N,E)^B\otimes_{E}\sigma'\rightarrow \cD^{(I_n\cap N)-\mathrm{an}}(I\cap N,E)^B\otimes_{E}\sigma$.

    The base change to $S_h$ of $\iota$ is still of trace-class (cf. \cite[Rem. 5.31]{andreychev2021pseudocoherent}), as well as the following map (by checking the definition or using \cite[Lem. A.14]{bosco2023rational})
    \begin{align}\label{equationperfect1}
        (S_h\otimes_{S_{\infty}^{\rm rig},\blacksquare}M_{\infty}^{\rig})\otimes_{\cD(I,E),\blacksquare}\cD^{(t^{-1}I_nt\cap I)-\mathrm{an}}(I,E)\otimes_{\cD(\frg,I\cap \ovB)}\ovM=e.(S_h\otimes_{E,\blacksquare}\cD^{(t^{-1}I_nt\cap I)-\mathrm{an}}(I,E)\otimes_{\cD(\frg,I\cap \ovB)}\ovM)^a\\
        \rightarrow (S_h\otimes_{S_{\infty}^{\rm rig},\blacksquare}M_{\infty}^{\rig})\otimes_{\cD(I,E),\blacksquare}\cD^{I_n-\mathrm{an}}(I,E)\otimes_{\cD(\frg,I\cap \ovB)}\ovM=e.(S_h\otimes_{E,\blacksquare}\cD^{I_n-\mathrm{an}}(I,E)\otimes_{\cD(\frg,I\cap \ovB)}\ovM)^a.\nonumber
    \end{align}
      
    Recall by definition (\ref{equationUtaction}), $U_t$ is given by
    \begin{align}\label{equationperfect2}
        M_{\infty}^{\rig}\otimes_{\cD(I,E),\blacksquare}\cW&\rightarrow M_{\infty}^{\rig}\otimes_{\cD(G,E),\blacksquare}\cind_{I}^G\cW=M_{\infty}^{\rig}\otimes_{\cD(I,E),\blacksquare}\cW\\\nonumber
        m\otimes w&\mapsto m\otimes \sum_{x\in I/(tIt^{-1}\cap I)}[xt,\psi_t(x^{-1}w)]= \sum_{x\in (I\cap \ovN)/t(I\cap \ovN)t^{-1}}t^{-1}x^{-1}.m\otimes \psi_t(x^{-1}w).
    \end{align} 
    The formula for (\ref{equationperfect2}) defines a map
    \[(S_h\otimes_{S_{\infty}^{\rm rig},\blacksquare}M_{\infty}^{\rig})\otimes_{E,\blacksquare}\cD^{I_n-\mathrm{an}}(I,E)\otimes_{\cD(\frg,I\cap \ovB)}\ovM\rightarrow(S_h\otimes_{S_{\infty}^{\rm rig},\blacksquare}M_{\infty}^{\rig})\otimes_{E,\blacksquare}\cD^{(t^{-1}I_nt\cap I)-\mathrm{an}}(I,E)\otimes_{\cD(\frg,I\cap \ovB)}\ovM.\]
    The above map descends to the following map (cf. \cite[Lem. 4.2.11]{emerton2006jacquet}).
    \begin{align}\label{equationperfect3}
        &(S_h\otimes_{S_{\infty}^{\rm rig},\blacksquare}M_{\infty}^{\rig})\otimes_{\cD(I,E),\blacksquare}\cD^{I_n-\mathrm{an}}(I,E)\otimes_{\cD(\frg,I\cap \ovB)}\ovM\\
        &\rightarrow (S_h\otimes_{S_{\infty}^{\rm rig},\blacksquare}M_{\infty}^{\rig})\otimes_{\cD(I,E),\blacksquare}\cD^{(t^{-1}I_nt\cap I)-\mathrm{an}}(I,E)\otimes_{\cD(\frg,I\cap \ovB)}\ovM.\nonumber
    \end{align}
    
    Hence $U_t:e.(S_h\otimes_{E,\blacksquare}\cW)^a\rightarrow e.(S_h\otimes_{E,\blacksquare}\cW)^a$ is trace-class being the composite of (\ref{equationperfect3}) and the trace class map (\ref{equationperfect1}) by \cite[Lem. 8.2]{clausen2022condensed}. By discussions in \cite[\S9]{clausen2022condensed} and the proof of \cite[Prop. 9.11]{clausen2022condensed} or \cite[Lem. 5.51]{andreychev2021pseudocoherent}, $\mathrm{cone}(1-U_t)=S_h\otimes_{S_{\infty}^{\rm rig},\blacksquare} A_{\infty}^{\rig}(\cind_I^G\cW)\otimes_{\cH}^L\cH/(U_t-1)$ is a perfect $S_h$-complex. In a more classical language, the map $U_t$ can be factored as $e.(S_h\otimes_{E,\blacksquare}\cW)^a\stackrel{g}{\rightarrow}V\stackrel{f}{\rightarrow}e.(S_h\otimes_{E,\blacksquare}\cW)^a$ for some Banach $S_h$-module $V$ (that is a direct summand of an orthonormalizable Banach $S_h$-module). Then $\mathrm{cone}(1-U_t)\simeq \mathrm{cone}(1-g\circ f)$ with $g\circ f$ an $S_h$-compact operator on $V$ and is perfect over $S_h$ by the classical Fredholm theory (see \cite[Prop. 2.2.6]{emerton2006jacquet}).
   
    Finally, $\cH/(U_t-1)$ is a Noetherian regular ring (finite \'etale over a polynomial ring, of the form $E[U_{t_1},\cdots,U_{t_d}]/(U_{t_1}^{s_1}\cdots U_{t_d}^{s_d}-1)$), hence has finite global projective dimension \cite[\href{https://stacks.math.columbia.edu/tag/00OE}{Tag 00OE}]{stacks-project}. We get that $ \cH/\frm$ admits a finite projective resolution over $\cH/(U_t-1)$. Hence $S_h\otimes_{S_{\infty}^{\rm rig},\blacksquare} A_{\infty}^{\rig}(\cind_I^G\cW^{\vee})\otimes_{\cH}^L\cH/\frm=(S_h\otimes_{S_{\infty}^{\rm rig},\blacksquare} A_{\infty}^{\rig}(\cind_I^G\cW)\otimes_{\cH}^L\cH/(U_t-1))\otimes_{\cH/(U_t-1)}^L\cH/\frm$ is also a perfect $S_h$-complex (\cite[\href{https://stacks.math.columbia.edu/tag/066R}{Tag 066R}]{stacks-project}). We finished the proof.
\end{proof}

For $\cM=(M_{h})_{h} \in \mathrm{Mod}_{E_{\blacksquare}}^{\rm qc}(R_{\infty}^{\rig})$, we let 
\[\bbD_{\rm GS}(\cM):=\RiHom_{\mathrm{Mod}_{E_{\blacksquare}}^{\rm qc}(R_{\infty}^{\rig})}(M,f^!(S_h)_h)\]
where $f:\Spf(R_{\infty})^{\rm rig}\rightarrow \Spf(S_{\infty})^{\rm rig}$ is induced by $S_{\infty}\rightarrow R_{\infty}$ as in \S\ref{subsectionquasistein} and $(S_h)_h$ denotes the structure sheaf of $\Spf(S_{\infty})^{\rm rig}$. The following theorem should be compared with \cite[Conj. 4.5.1 (1)]{zhu2020coherent} and \cite[Conj. 6.1.14 (3) \& Rem. 6.2.22]{emerton2023introduction}.
\begin{thm}\label{theoremdualitypatchingmodule}
    Let $M\in \cO^{\frb}_{\rm alg}$ and $\chi_{\sm}$ be a smooth character of $T$. Let $\pi=\cF_{B}^G(M,\chi_{\sm})$ as in Theorem \ref{theoremBZdual}. Then there exists an isomorphism 
    \[\bbD_{\rm GS}(\frA_{\infty}^{\rm rig}(\pi))\simeq \eta^*\frA_{\infty}^{',\rig}(\bbD_{\rm BZ}(\pi)).\]
    in $\mathrm{Mod}_{E_{\blacksquare}}^{\rm qc}(R_{\infty}^{\rig})$.
\end{thm}
\begin{proof}
    Write $\cF_{B}^G(M,\chi_{\sm})= \cind_I^G\cW\otimes_{\cH}^L\cH/\frm$ as in Proposition \ref{propositionsolidresolution} using Theorem \ref{theoremresolution} such that $\cW$ is Smith. By Proposition \ref{propositionfiniteslopeperfect}, $A_{\infty}^{\rig}(\cind_I^G\cW\otimes_{\cH}^L\cH/\frm)$ is a perfect object in $\Mod_{E_{\blacksquare}}^{\rm qc}(S_{\infty}^{\rig})\subset \Mod_{E_{\blacksquare}}(S_{\infty}^{\rig})$. By Lemma \ref{lemmacompactlysupported}, $A_{\infty}^{\rig}(\cind_I^G\cW\otimes_{\cH}^L\cH/\frm)=j_*\frA_{\infty}^{\rig}(\cind_I^G\cW\otimes_{\cH}^L\cH/\frm)$ and
    \[f_!\frA_{\infty}^{\rig}(\cind_I^G\cW\otimes_{\cH}^L\cH/\frm)=f_*\frA_{\infty}^{\rig}(\cind_I^G\cW\otimes_{\cH}^L\cH/\frm).\]
    
    One can check formally that the isomorphism in Lemma \ref{lemmapatchingdualcompactinduction} is compatible with the actions of $\cH$, where $\cH$ acts on $A_{\infty}^{',\rig}(\cind_I^G\cW^{\vee})$ via the transpose in Lemma \ref{lemmatransposeUt}. By Theorem \ref{theoremsoliddualcompactinduction} and the Koszul duality \cite[Chap. IV.A.2]{serre1997algebre} (we omit the pullback $\eta^*$ in the following)
    \[\begin{array}{rl}
        &\RiHom_{S_{\infty}^{\rig}}(A_{\infty}^{\rig}(\cind_I^G\cW)\otimes_{\cH}^L\cH/\frm,S_{\infty}^{\rig})\\
        \simeq& R\Hom_{\cH}(\cH/\frm,\RiHom_{S_{\infty}^{\rig}}(A_{\infty}^{\rig}(\cind_I^G\cW),S_{\infty}^{\rig}))\\
        \simeq& A_{\infty}^{',\rig}(\cind_I^G\cW^{\vee})\otimes_{\cH}^L\cH/\frm[-d]\\
        =& A_{\infty}^{',\rig}(\mathbb{D}_{\rm BZ}(\cind_I^G\cW\otimes_{\cH}^L\cH/\frm)).
    \end{array}\] 
    Applying the adjunctions, we get identifications of $R_{\infty}^{\rm rig}$-module objects in $\mathrm{Mod}_{E_{\blacksquare}}^{\rm qc}(S_{\infty}^{\rm rig})$: 
    \[\begin{array}{rl} 
        &f_*\RiHom_{\mathrm{Mod}_{E_{\blacksquare}}^{\rm qc}(R_{\infty}^{\rig})}(\frA_{\infty}^{\rig}(\cind_I^G\cW\otimes_{\cH}^L\cH/\frm),f^!(S_h)_h)\\
        =&\RiHom_{\mathrm{Mod}_{E_{\blacksquare}}^{\rm qc}(S_{\infty}^{\rig})}(f_!\frA_{\infty}^{\rig}(\cind_I^G\cW\otimes_{\cH}^L\cH/\frm),(S_h)_h)\\
        =&\RiHom_{\mathrm{Mod}_{E_{\blacksquare}}^{\rm qc}(S_{\infty}^{\rig})}(f_*
        \frA_{\infty}^{\rig}(\cind_I^G\cW\otimes_{\cH}^L\cH/\frm),(S_h)_h)\\
        =&j^*\RiHom_{S_{\infty}^{\rig}}(A_{\infty}^{\rig}(\cind_I^G\cW\otimes_{\cH}^L\cH/\frm),S_{\infty}^{\rig})\\
        =&f_*\frA_{\infty}^{',\rig}(\mathbb{D}_{\rm BZ}(\cind_I^G\cW\otimes_{\cH}^L\cH/\frm))
    \end{array}\]
    where for the third equality, we used the fully faithfulness of $j_*$ (for $S_{\infty}$) and Lemma \ref{lemmapushforward}. The functor $f_*$ sending $\mathrm{Mod}_{E_{\blacksquare}}^{\rm qc}(R_{\infty}^{\rm rig})$ to $R_{\infty}^{\rm rig}$-modules in $\mathrm{Mod}_{E_{\blacksquare}}^{\rm qc}(S_{\infty}^{\rm rig})$ is fully faithful (composed with the global section functor for $\mathrm{Mod}_{E_{\blacksquare}}^{\rm qc}(S_{\infty}^{\rm rig})$ it gives that for $\mathrm{Mod}_{E_{\blacksquare}}^{\rm qc}(R_{\infty}^{\rm rig})$ in Lemma \ref{lemmafullyfaithfulglobalsection}). We conclude that 
    \[ \RiHom_{\mathrm{Mod}_{E_{\blacksquare}}^{\rm qc}(R_{\infty}^{\rig})}(\frA_{\infty}^{\rig}(\cind_I^G\cW\otimes_{\cH}^L\cH/\frm),f^!(S_h)_h)\simeq \frA_{\infty}^{',\rig}(\mathbb{D}_{\rm BZ}(\cind_I^G\cW\otimes_{\cH}^L\cH/\frm))\]
    which gives the desired isomorphism.
\end{proof}
\bibliographystyle{alpha}
\bibliography{mybib}

\end{document}